\documentclass[11pt,reqno]{amsart}
\usepackage{amssymb}
\usepackage{amsmath, amssymb, amsthm}
\usepackage{mathrsfs}
\usepackage{soul,color}

\newtheorem{theorem}{Theorem}[section]
\newtheorem{definition}{Definition}[section]
\newtheorem{lemma}{Lemma}[section]
\newtheorem{remark}{Remark}[section]

\numberwithin{equation}{section}

\newcommand{\R}{{\mathbb R}}

\begin{document}
\title[Vanishing Viscosity Limit]{Vanishing Viscosity Limit of the Navier-Stokes Equations to the Euler Equations for Compressible Fluid Flow with vacuum}

\author{Yongcai Geng}
\address[Y. C. Geng]{School of Science,
         Shanghai Institute of Technology,
         Shanghai 200235, P.R.China}
\email{\tt ycgengjj@163.com}

\author{Yachun Li}
\address[Y. C. Li]{Department of Mathematics, MOE-LSC, and SHL-MAC, Shanghai Jiao Tong University, Shanghai 200240, P.R.China} \email{\tt ycli@sjtu.edu.cn}

\author{Shengguo Zhu}
\address[S. G.  Zhu]{Mathematical Institute,
University of Oxford, Oxford,  OX2 6GG, UK.}
\email{\tt shengguo.zhu@maths.ox.ac.uk}

\begin{abstract}
We establish the vanishing viscosity limit of the Navier-Stokes
equations to the Euler equations  for three-dimensional compressible
isentropic  flow in the whole space.  When the viscosity coefficients  are given as  constant multiples of the  density's power ($\rho^\delta$ with $\delta>1$),  it is shown that there exists a
unique regular solution of compressible  Navier-Stokes equations
with arbitrarily large initial data
and vacuum, whose life span is uniformly positive in the
vanishing viscosity limit. It is worth paying special attention that, via introducing a  `` quasi-symmetric hyperbolic"--``degenerate elliptic" coupled structure to control the behavior of the  velocity of the fluid  near the vacuum, 
we can also give some uniform  estimates for 
$\displaystyle\big(\rho^{\frac{\gamma-1}{2}}, u\big)$ in $H^3$ 
and $\rho^{\frac{\delta-1}{2}}$ in $H^2$  with respect to the viscosity coefficients (adiabatic exponent $\gamma>1$ and $1<\delta \leq \min\{3, \gamma\}$), which
lead the strong convergence of the regular solution of  the viscous flow to
that of the  inviscid flow  in $L^{\infty}([0, T]; H^{s'})$ (for any
$s'\in [2, 3)$)  with the rate of $\epsilon^{2(1-s'/3)}$. Further more, we point out  that our framework  in this paper is applicable to other physical dimensions, say 1 and 2, with some minor modifications. 

\end{abstract}

\date{Apr. 15, 2019}

\maketitle

\section{Introduction}\ \\

In this paper, we investigate the inviscid limit problem  of the 3D isentropic compressible  Navier-Stokes equations with   degenerate viscosities, when   the initial data contain vacuum and  are arbitrarily large.  For this purpose, we consider the following isentropic compressible  Navier-Stokes equations (ICNS) in  $\mathbb{R}^3$:
\begin{equation}
\label{eq:1.1}
\begin{cases}
\rho_t+\text{div}(\rho u)=0,\\[4pt]
(\rho u)_t+\text{div}(\rho u\otimes u)
  +\nabla
   P =\text{div} \mathbb{T}.
\end{cases}
\end{equation}
We look for the above system's  local  regular solutions with initial data
\begin{equation} \label{eq:2.211m}
(\rho,u)|_{t=0}=(\rho_0(x),u_0(x)),\quad x\in \mathbb{R}^3,
\end{equation}
and the  far field behavior
\begin{equation} \label{eq:2.211}
(\rho,u)\rightarrow (0,0) \quad \text{as } \quad |x|\rightarrow +\infty,\quad t\geq 0.
\end{equation}
Usually, such kind of  far field behavior occurs naturally under some physical assumptions on \eqref{eq:1.1}'s solutions, such as
finite total mass and total energy.

In system (\ref{eq:1.1}), $x=(x_1,x_2,x_3)\in \mathbb{R}^3$, $t\geq 0$ are the space and time variables, respectively, $\rho$ is the density, and $u=\left(u^{(1)},u^{(2)},u^{(3)}\right)^\top\in \mathbb{R}^3$ is the velocity of the fluid. In considering the polytropic gases, the constitutive relation, which is also called the equations of state, is given by
\begin{equation}
\label{eq:1.2}
P=A\rho^{\gamma}, \quad \gamma> 1,
\end{equation}
where $A>0$ is  an entropy  constant and  $\gamma$ is the adiabatic exponent. $\mathbb{T}$ denotes the viscous stress tensor with the  form
\begin{equation}
\label{eq:1.3}
\begin{split}
&\mathbb{T}=\mu(\rho)\left(\nabla u+(\nabla u)^\top\right)+\lambda(\rho) \text{div}u\,\mathbb{I}_3,\\
\end{split}
\end{equation}
 where $\mathbb{I}_3$ is the $3\times 3$ identity matrix,
\begin{equation}
\label{fandan}
\mu(\rho)=\epsilon \alpha\rho^\delta,\quad \lambda(\rho)=\epsilon \beta\rho^\delta,
\end{equation}
 $\mu(\rho)$ is the shear viscosity coefficient, $\lambda(\rho)+\frac{2}{3}\mu(\rho)$ is the bulk viscosity coefficient, $\epsilon \in (0,1]$ is a   constant,  $\alpha$ and $\beta$ are both constants satisfying
 \begin{equation}\label{10000}\alpha>0,\quad  2\alpha+3\beta\geq 0,
 \end{equation}
and in this paper, we assume that the constant $\delta$ satisfies
 \begin{equation}\label{100001} 1< \min\{\delta, \gamma\}\leq 3.
 \end{equation}
In addition,  when $\epsilon=0$, from
\eqref{eq:1.1}, we naturally have the compressible isentropic Euler
equations for  the  inviscid flow:
\begin{equation}
\label{eq:1.1E}
\begin{cases}
\displaystyle
\rho_t+\text{div}(\rho u)=0,\\[8pt]
\displaystyle
(\rho u)_t+\text{div}(\rho u\otimes u)
  +\nabla
   P =0,
\end{cases}
\end{equation}
which is a fundamental example of a system of hyperbolic conservation laws.

Throughout this paper, we adopt the following simplified notations, most of which  are for the standard homogeneous and inhomogeneous Sobolev spaces:
\begin{equation*}\begin{split}
& |f|_p=\|f\|_{L^p(\mathbb{R}^3)},\quad \|f\|_s=\|f\|_{H^s(\mathbb{R}^3)},\quad  |f|_2=\|f\|_0=\|f\|_{L^2(\mathbb{R}^3)}, \\[8pt]
&D^{k,r}=\{f\in L^1_{loc}(\mathbb{R}^3):  |\nabla^kf|_{r}<+\infty\},\quad D^k=D^{k,2},\quad |f|_{D^{k,r}}=\|f\|_{D^{k,r}(\mathbb{R}^3)}\ (k\geq 2),\\[8pt]
&  D^{1}=\{f\in L^6(\mathbb{R}^3):  |\nabla f|_{2}<\infty\},\quad|f|_{D^{1}}=\|f\|_{D^{1}(\mathbb{R}^3)},\quad \int_{\mathbb{R}^3} f\text{d}x=\int f.
\end{split}
\end{equation*}
 A detailed study of homogeneous Sobolev spaces  can be found in Galdi \cite{gandi}.

\subsection{Existence theories of compressible flow with vacuum}\ \\

Before formulating  our problem,  we first  briefly recall a series of  frameworks on the   well-posedness of  multi-dimensional  strong solutions with initial vacuum  established for the hydrodynamics equations  mentioned above in the whole space.  For the inviscid flow, in  1987, via writing \eqref{eq:1.1E} as a symmetric hyperbolic form that allows the density to vanish, Makino-Ukai-Kawshima \cite{tms1} obtained     the  local-in-time existence of the unique  regular solution  with $\inf \rho_0=0$,  which  can be shown  by
\begin{theorem} \cite{tms1}\label{thmakio}
Let   $ \gamma >1$. If the initial data $( \rho_0, u_0)$ satisfy
\begin{equation}\label{th78asd}
\begin{split}
&\rho_0\geq 0,\quad \big(\rho^{\frac{\gamma-1}{2}}_0, u_0\big)\in
H^3(\mathbb{R}^3),
\end{split}
\end{equation}
then there exist a  time $T_0>0$ and a unique regular solution
$(\rho, u)$ to Cauchy problem \eqref{eq:1.1E} with
\eqref{eq:2.211m}--\eqref{eq:2.211} satisfying
\begin{equation}\label{reg11asd}\begin{split}
& \big(\rho^{\frac{\gamma-1}{2}},u\big) \in C([0,T_0];H^3),\quad \big((\rho^{\frac{\gamma-1}{2}})_t,u_t\big) \in C([0,T_0];H^2),
\end{split}
\end{equation}
where the  regular solution $(\rho, u)$  to \eqref{eq:1.1E} with
\eqref{eq:2.211m}--\eqref{eq:2.211}  is defined by
\begin{equation*}\begin{split}
&(\textrm{A})\quad  \big(\rho, u \big) \  \text{satisfies   (\ref{eq:1.1E}) with
\eqref{eq:2.211m}--\eqref{eq:2.211} in the sense of  distributions};\\
&(\textrm{B})\quad  \rho\geq 0,\ \big(\rho^{\frac{\gamma-1}{2}}, u\big)\in C^1([0,T_0]\times \mathbb{R}^3); \\
&(\textrm{C})\quad u_t+u\cdot\nabla u =0 \quad  \text{when } \quad
\rho(t,x)=0.
\end{split}
\end{equation*}

\end{theorem}
It should be pointed out that  the condition
$(\textrm{C})$  ensures the uniqueness of the 
regular solution and makes the velocity $u$ well defined in vacuum
region. Without $(\textrm{C})$, it is  difficult to get enough
information on  velocity even for  considering special cases such
as point vacuum or continuous vacuum on some surface. In 1997, by extracting a dispersive effect after some invariant transformation, Serre \cite{danni} obtained the global existence of the regular  solution shown in Theorem \ref{thmakio} with small density.

 For  the  constant viscous flow (i.e., $\delta=0$ in (\ref{fandan})),  the corresponding  local-in-time well-posedness of strong solutions with vacuum was firstly solved by Cho-Choe-Kim \cite{CK3,guahu} in 2004-2006. To compensate the lack of positive lower bound of the  initial  density, they introduced an initial compatibility condition \begin{equation}\label{th79bn}
\begin{split}
\text{div}\mathbb{T}(u_0)+\nabla P_0=\rho_0 f,\  \text{for\ some} \ f\in D^1 \ \text{and}\  \sqrt{\rho}_0f\in L^2,
\end{split}
\end{equation}
which plays a key  role in getting some uniform  a priori estimates with respect to  the lower bound of $\rho_0$;  see also Duan-Luo-Zheng \cite{yuxi}  for 2D case.   Later, based on the uniform estimate on the upper bound of the density, Huang-Li-Xin \cite{HX1} extended this solution to be a global one under some initial smallness assumption for the isentropic flow in $\mathbb{R}^3$.

Recently,  the  degenerate viscous flow (i.e., $\delta=0$ in (\ref{fandan})) described  by (\ref{eq:1.1})  has received extensive attentions from the mathematical community (see the review papers \cite{bd2,hyp}),  which is  based on the following  two main considerations. On the one hand, through the second-order Chapman-Enskog expansion from Boltzmann equations to the compressible Navier-Stokes equations,  it is known that the viscosity coefficients are not constants but functions of the absolute temperature  (cf. Chapman-Cowling \cite{chap} and Li-Qin \cite{tlt}), which can be reduced to the dependence on density for isentropic flow  from laws of Boyle and Gay-Lussac (see \cite{tlt}). On the other hand, we do have some good models in the following $2$D shallow water equations for the height $h$ of the free surface and the mean horizontal velocity field $U$:
\begin{equation}
\label{eq:1.1ww}
\begin{cases}
h_t+\text{div}(h U)=0,\\[6pt]
(h U)_t+\text{div}(h U\otimes U)
  +\nabla h^2 =\mathcal{V}(h, U).
\end{cases}
\end{equation}
It is clear that system \eqref{eq:1.1ww} is a special case or a simple variant of system \eqref{eq:1.1} for some appropriately chosen viscous term $\mathcal{V}(h, U)$ (see \cite{BN2,Gent,gp,Mar}).  Some important progress have been obtained  in the development of the global existence of weak solutions with vacuum for system (\ref{eq:1.1}) and related models, see Bresh-Dejardins \cite{bd3,bd6,bd2}, Mellet-Vassuer \cite{vassu} and some other interesting results, c.f. \cite{  hailiang, taiping,  zyj, tyc2}.

However, in the presence of vacuum, compared with the constant viscosity case \cite{CK3},  there appear some new mathematical challenges in dealing with such systems  for constructing solutions with high regularities. In particular, these systems become highly degenerate, which results in that the velocity can not even be defined in the  vacuum domain and hence it is difficult to get uniform estimates for the velocity near the vacuum. Recently in Li-Pan-Zhu \cite{sz3,sz333}, via  carefully analyzing  the  mathematical  structure of these systems,   they reasonably gave  the time  evolution mechanism of the fluid velocity    in the vacuum domain. Taking $\delta=1$ for example, via considering the following parabolic equations with a special source term:
\begin{equation}\label{c-1}
\begin{cases}
\displaystyle
u_t+u\cdot\nabla u +\frac{2A\gamma}{\gamma-1}\rho^{\frac{\gamma-1}{2}}\nabla\rho^{\frac{\gamma-1}{2}}+Lu=(\nabla \rho/\rho) \cdot \mathbb{S}(u),\\[8pt]
\displaystyle
Lu=-\alpha\triangle u-(\alpha+\beta)\nabla \mathtt{div}u,\\[8pt]
\displaystyle
 \mathbb{S}(u)=\alpha(\nabla u+(\nabla u)^\top)+\beta\mathtt{div}u\mathbb{I}_3,
\end{cases}
\end{equation}
we could transfer the degenearcy shown in system (\ref{eq:1.1}) caused by the far field vacuum to the possible singualrity of the  quantity $\nabla \rho/\rho$, which was luckily shown to be well defined in $L^6\cap D^1$.
Based on this, by making full use of the  symmetrical structure of the hyperbolic operator and  the weak smoothing effect of the elliptic operator, they established    a series of  a priori estimates independent of the lower bound of $\rho_0$, and successfully gave the local existence theory of classical solutions with arbitrarily large data and vacuum for the case $1\leq \delta \leq \min\Big\{3,\frac{\gamma+1}{2}\Big\}$. 
We refer readers to  Zhu \cite{sz33,sz34} for more details and progress.

\subsection{Vanishing viscosity limit from viscous flow to inviscid flow}\ \\

Based on the well-posedness theory mentioned above,  naturally  there is an important question that: can we regard  the regular solution  of inviscid flow \cite{tms1,danni} as those of viscous flow \cite{CK3,  yuxi, HX1, sz3,  sz333} with vanishing real physical viscosities?

Actually, there are  lots of literature on the uniform bounds and the vanishing viscosity limit in the whole space. The idea of  regarding inviscid flow as viscous flow with vanishing real physical viscosity can date back to    Dafermos \cite{guy}, Hugoniot \cite{H1}, Rankine \cite{R1}, Rayleigh \cite{R2} and Stokes \cite{stoke}. However, until 1951, Gilbarg \cite{DG}  gave us the first rigorous convergence analysis of vanishing physical viscosities from the Navier-Stokes equations (1.1) to the isentropic Euler equations (\ref{eq:1.1E}),  and   established the mathematical existence and vanishing viscous limit of the Navier-Stokes shock layers. The framework on the convergence analysis of piecewise smooth solutions has been established by G$\grave{\text{u}}$es-M$\acute{\text{e}}$tivier-Williams-Zumbrun
\cite{OG}, Hoff-Liu \cite{HL}, and the references cited therein. The convergence of vanishing physical viscosity with general initial data was first studied by Serre-Shearer \cite{DSJWS} for a $2\times 2$ system in nonlinear elasticity with severe growth conditions on the nonlinear function in the system.
In 2009, based on the uniform energy estimates and compactness compensated argument, Chen-Perepelitsa \cite{chen4}  established the first convergence result for the vanishing physical viscosity limit of solutions of the Navier-Stokes equations to a finite-energy entropy weak solution of the isentropic Euler equations with finite-energy initial data, which has been extended to the   density-dependent viscosity case (experiencing degeneracy near
vacuum states) by  Huang-Pan-Wang-Wang-Zhai
\cite{HPWWZ}. 

However, even in   1D space, due to the complex mathematical structure of  hydrodynamics equations near the vacuum,  the existence of strong solutions  to the viscous flow  and inviscid flow are usually established in totally different frameworks, for example,  \cite{CK3} and  \cite{tms1}. 
The proofs  shown in \cite{CK3, yuxi, sz3} essentially depend on the uniform  ellipticity of the Lam\'e operator $L$,
and the   a priori estimates on the solutions
 and their life spans $T^{v}$  obtained in the above references  both strictly depend on the  real physical viscosities.  For example, when $\delta=0$ (i.e., $\mu=\epsilon \alpha$ and $\lambda=\epsilon \beta$),  we have
\begin{equation}\label{c-33}
\begin{cases}
\displaystyle
|u|_{D^{k+2}}\leq C\Big(\frac{1}{\epsilon \alpha},\frac{1}{\epsilon \beta}\Big)\big(|u_t+u\cdot \nabla u+\nabla P|_{D^k}\big),\\[10pt]
\displaystyle
 \quad T^{v}\sim O\big(\epsilon \alpha\big)+O\big(\epsilon \beta\big),
\end{cases}
\end{equation}
which implies  that the current frameworks do not seem to work for verifying the expected  limit relation. 
Thus the vanishing viscosity limit for the multi-dimensional  strong solutions in the whole space from  Navier-Stokes equations to Euler equations for compressible flow with initial vacuum in some open set or at the far field    is still an open problem. 

  In this paper,  motivated by the uniform estimates on $\epsilon^{\frac{1}{2}}\nabla^3\rho^{\frac{\delta-1}{2}} $ and $\epsilon^{\frac{1}{2}}\rho^{\frac{\delta-1}{2}} \nabla^4 u$, when the viscous stress tensor has the form (\ref{fandan})-(\ref{100001}),  we aim at giving a positive answer for this question  by introducing a  ``quasi-symmetric hyperbolic"--``degenerate elliptic" coupled structure to control the behavior of the velocity near the vacuum. We believe that the method developed in this work could give us a good understanding of the mathematical theory of vacuum, and also  can be adapted  to some other related vacuum problems in a more general framework, such as the inviscid limit problem for multi-dimensional  finite-energy weak solutions in the whole space.

\subsection{Symmetric formulation and main results}\ \\

We first need to analyze the mathematical structure of the momentum equations $(\ref{eq:1.1})_2$ carefully, which can be divided into hyberbolic, elliptic and source  parts  as follows:
\begin{equation}
\label{eq:1.1mo}
\begin{split}
\displaystyle
\underbrace{ \rho \big(u_t+ u\cdot  \nabla u\big)+\nabla P }_{\text{Hyberbolic}}&=\underbrace{-\epsilon \rho^\delta Lu}_{\text{Elliptic} }+\underbrace{\epsilon \nabla \rho^\delta \cdot  \mathbb{S}(u)}_{\text{Source}}.
\end{split}
\end{equation}
For smooth solutions $(\rho,u)$ away from the vacuum,  these equations  could be written into
\begin{equation}\label{c-2}
\begin{split}
&\underbrace{ u_t+u\cdot\nabla u +\frac{A\gamma}{\gamma-1}\nabla \rho^{\gamma-1}-\frac{\delta}{\delta-1} \epsilon\nabla \rho^{\delta-1} \cdot \mathbb{S}(u)}_{\text{Lower \ order}}=\underbrace{ -\epsilon\rho^{\delta-1}Lu}_{\text{Higher order}},
\end{split}
\end{equation}
Then if  $\rho$ is smooth enough, we could  pass to the limit as   $\rho\rightarrow 0$ on both sides of (\ref{c-2}) and  formally have
\begin{equation}\label{zhenkong}
 u_t+ u\cdot \nabla u=0 \quad \text{when}\ \rho=0,
\end{equation}
which, along with (\ref{c-2}),  implies that  the velocity $u$ can be governed by a nonlinear degenerate  parabolic system  if the density function contains vacuum.

In order to establish uniform  a priori estimates  for $u$ in $H^3$  that is  independent of $\epsilon$ and the lower bound of the initial density,  we hope that    the first order terms on the left-hand side of \eqref{c-2} could be put into a symmetric hyperbolic structure, then we can deal with the estimates on u in  $H^3$ without being affected   by the $\epsilon$-dependent degenerate elliptic operator.  However, it is  impossible. The problem  is that  the term   $ \nabla \rho^{\delta-1} \cdot \mathbb{S}(u)$ is actually a product of   two first order derivatives with the form $\sim \rho^{\delta-2} \nabla \rho \cdot \nabla u$, and obviously there is no similar term in the continuity equation $(\ref{eq:1.1})_1$.  This also tells us that  it is not enough if we only have the estimates on $\rho$.  Some more elaborate estimates for the density related quantities  are really needed.

In fact, by introducing two new quantities:
$$
\varphi=\rho^{\frac{\delta-1}{2}},\quad \text{and} \quad \phi=\rho^{\frac{\gamma-1}{2}},
$$
equations \eqref{eq:1.1} can be rewritten into a new system that consists of a transport equation for $\varphi$, and
a  ``quasi-symmetric hyperbolic"--``degenerate elliptic" coupled system with some special lower order source terms for $(\phi,u)$: 
\begin{equation}\label{li46}
\begin{cases}
\displaystyle
\underbrace{\varphi_t+u\cdot\nabla\varphi+\frac{\delta-1}{2}\varphi\text{div} u=0}_{\text{Transport equations}},\\[10pt]
\displaystyle
\underbrace{A_0W_t+\sum_{j=1}^3A_j(W) \partial_j W}_{\text{Symmetric hyperbolic}}=\underbrace{-\epsilon \varphi^2\mathbb{{L}}(W)}_{\text{Degenerate elliptic}}+\underbrace{\epsilon \mathbb{{H}}(\varphi)  \cdot \mathbb{{Q}}(W)}_{\text{Lower order source}},
 \end{cases}
\end{equation}
where $W=(\phi, u)^\top$ and
\begin{equation} \label{sseq:5.2qq}
\begin{split}
\mathbb{{L}}(W)=\left(\begin{array}{c}0\\[2pt]
a_1Lu\\
\end{array}\right),\quad \mathbb{{H}}(\varphi)=\left(\begin{array}{c}0\\[8pt]
\nabla \varphi^2\\
\end{array}\right), \quad \mathbb{{Q}}(W)=\left(\begin{array}{cc}
0 & 0\\[8pt]
0 &a_1Q(u)
\end{array}\right),\\
\end{split}
\end{equation}
with $a_1=\frac{(\gamma-1)^2}{4A\gamma}>0$ and $Q(u)=\frac{\delta}{\delta-1}\mathbb{S}(u)$.
Meanwhile, $\partial_j W=\partial{x_j}W$, and
\begin{equation} \label{eq:5.2qq1}
\begin{split}
&A_0=\left(\begin{array}{cc}
1&0\\[8pt]
0&\frac{(\gamma-1)^2}{4A\gamma}\mathbb{I}_3
\end{array}
\right),\quad
\displaystyle
A_j=\left(\begin{array}{cc}
u^{(j)}&\frac{\gamma-1}{2}\phi e_j\\[8pt]
\frac{\gamma-1}{2}\phi e_j^\top &\frac{(\gamma-1)^2}{4A\gamma}u^{(j)}\mathbb{I}_3
\end{array}
\right),\quad j=1,2,3.\\
\end{split}
\end{equation}
Here  $e_j=(\delta_{1j},\delta_{2j},\delta_{3j})$ $(j=1,2,3)$ is the Kronecker symbol satisfying $\delta_{ij}=1$, when $ i=j$ and $\delta_{ij}=0$, otherwise. For any $\xi\in \R^4$, we have
\begin{equation} \label{sseq:5.3qq}
\xi^\top A_0\xi\geq a_2|\xi|^2 \quad \text{with} \quad  a_2=\min \left\{1, \frac{(\gamma-1)^2}{4A\gamma} \right\}>0.
\end{equation}
For simplicity, we denote the symmetric hyperbolic structure shown in the right-hand side of $(\ref{li46})_2$ as SH.

Considering the above system (\ref{li46}), first SH does not include all the first  order terms related on $\big(\rho^{\frac{\gamma-1}{2}},u\big)$ \big( i.e., $\epsilon \mathbb{{H}}(\varphi)  \cdot \mathbb{{Q}}(W)\big)$, so we called the equations $(\ref{li46})_2$  a ``quasi-symmetric hyperbolic"--``degenerate elliptic" system. 
Second, the characteristic speeds of SH  in the direction $\textbf{l}\in S^2$ are $u\cdot \textbf{l}$, with multiplicity two, and $u\cdot \textbf{l}\pm \sqrt{P_\rho}$, with multiplicity one, which means that  this structure fails to be strictly hyperbolic near the vacuum even in one-dimensional or two-dimensional spaces.
At last,   this formulation implies  that if we can give some reasonable  analysis  on the additional variable  $\rho^{\frac{\delta-1}{2}}$,  such that the $\rho^{\frac{\delta-1}{2}}$-related terms will vanish as $\epsilon\rightarrow 0$ and also do  not influence the estimates on $\big(\rho^{\frac{\gamma-1}{2}},u\big)$, then  it is hopeful for us to get  the desired  uniform estimates on  $\big(\rho^{\frac{\gamma-1}{2}},u\big)$ in $H^3$.

Based on the above observations, we first  introduce a proper class of solutions to  system (\ref{eq:1.1}) with arbitrarily large initial data and vacuum.
\begin{definition}[\text{\textbf{Regular solution to the Cauchy problem (\ref{eq:1.1})-(\ref{eq:2.211})}}]\label{d1}
 Let $T> 0$ be a finite constant. A solution $(\rho,u)$ to the  Cauchy problem (\ref{eq:1.1})-(\ref{eq:2.211})  is called a regular solution in $ [0,T]\times \mathbb{R}^3$ if $(\rho,u)$ satisfies this problem in the sense of distributions and:
\begin{equation*}\begin{split}
&(\textrm{A})\quad  \rho\geq 0, \  \rho^{\frac{\delta-1}{2}}\in C([0,T]; H^3), \ \  \rho^{\frac{\gamma-1}{2}}\in C([0,T]; H^3); \\
& (\textrm{B})\quad u\in C([0,T]; H^{s'})\cap L^\infty(0,T; H^{3}),\quad  \rho^{\frac{\delta-1}{2}}\nabla^4 u \in L^2(0,T; L^2);\\
&(\textrm{C})\quad u_t+u\cdot\nabla u =0\quad  \text{as } \quad  \rho(t,x)=0.
\end{split}
\end{equation*}
where $s'\in[2,3)$ is an arbitrary constant.
\end{definition}

In order to establish the vanishing viscosity limit from the  viscous
flow to the inviscid flow,  first we give the following uniform
(with respect to $\epsilon$) local-in-time well-posedness to the 
Cauchy problem (\ref{eq:1.1})-(\ref{eq:2.211}).

\begin{theorem}[\text{\textbf{Uniform Regularity}}]\label{th2} Let (\ref{100001}) hold. If  initial data $( \rho_0, u_0)$ satisfy
\begin{equation}\label{th78}
\rho_0\geq 0,\quad (\rho^{\frac{\gamma-1}{2}}_0, \rho^{\frac{\delta-1}{2}}_0, u_0)\in H^3,
\end{equation}
then there exists a time $T_*>0$ independent of $\epsilon$, and a unique regular solution $(\rho, u)$ in $[0,T_*]\times \mathbb{R}^3$ to  the Cauchy problem (\ref{eq:1.1})-(\ref{eq:2.211}) satisfying the following uniform estimates:
\begin{equation}\label{qiyu}
\begin{split}
  \sup_{0\leq t \leq T_*}\Big(||\rho^{\frac{\gamma-1}{2}}||^2_3+||\rho^{\frac{\delta-1}{2}}||^2_2+\epsilon|\rho^{\frac{\delta-1}{2}}|^2_{D^3}+||u||_2^2\Big)(t)&
\\
 +\text{ess}\sup_{0\leq t \leq T_*}| u(t)|^2_{D^3}+\int_0^t\epsilon
 |\rho^{\frac{\delta-1}{2}}\nabla^4 u|_2^2ds\leq& C^0,
\end{split}
\end{equation}
for arbitrary constant $s'\in [2,3)$ and  positive constant $C^0=C^0(\alpha, \beta, A, \gamma, \delta, \rho_0, u_0)$. Actually,   $(\rho, u)$ satisfies the Cauchy problem (\ref{eq:1.1})-(\ref{eq:2.211}) classically in positve time  $(0, T_*]$.

Moreover, if  the following  condition holds,
\begin{equation}\label{regco}
1< min\{\delta,\gamma\} \leq 5/3, \quad \text{or} \quad  \delta=2, 3, \quad \text{or} \quad \gamma=2, 3, 
\end{equation}
we still have
\begin{equation}\label{regco11}\begin{split}
& \rho\in C([0,T_*];H^3),\quad  \rho_t \in C([0,T_*];H^2).
\end{split}
\end{equation}
\end{theorem}

\begin{remark}\label{r1}
The new varible $\phi$ is actually the constant multiple of local sound speed  $c$ of the hydrodynamics equations:
$$
c=\sqrt{\frac{d}{d\rho}P(\rho)} \quad \big(=\sqrt{A\gamma}\rho^{\frac{\gamma-1}{2}}\ \ \text{for polytropic flows}\big).
$$
\end{remark}

\begin{remark}\label{r2}

Compared with \cite{sz333},   we  not only extend the viscosity power parameter $\delta$ to a broder region $1< \min\{\delta, \gamma\}\leq 3$, but also   establish a
more precise estimate that the life span of the regular solution
 has a uniformly positive
 lower bound with respect to $\epsilon$. 
 Actually, our desired a priori estimates  mainly come from the ``quasi-symmetric hyperbolic--``degenerate elliptic" coupled structure $(\ref{li46})_2$, and the details could be seen in Section 3. Moreover,  we point out that the regular solution obtained in the above theorem will break down in finite time, if  the initial data contain   ``isolated mass group" or  ``hyperbolic singularity set", which could be rigorously proved via the same arguments used  in \cite{sz333}.
\end{remark}

\begin{remark}\label{gengr2}

If we relax the  initial assumption from $\rho^{\frac{\gamma-1}{2}}_0 \in H^3$ to $\rho^{\gamma-1}_0 \in H^3$,  then the correspongding  local-in-time well-posedness for  quantities  $(\rho^{\gamma-1},\rho^{\frac{\delta-1}{2}},u)$   still can be obtained by the similar argument used to prove Theorem {\rm \ref{th2}}. However, for this case, the uniform positive lower bound of the life span and the uniform a priori estimates with respect to $\epsilon$ are not available, because the change of variable from $c$ to $c^2$ has directly destroyed the symmetric hyperbolic structure as shown in the left hand side of $(\ref{li46})_2$.
\end{remark}

Letting $\epsilon\rightarrow 0,$ the  solution obtained in Theorem \ref{th2} will strongly converge  to that of the compressible  Euler equations
(\ref{eq:1.1E}) in $C([0,T];H^{s'})$ for any $ s'\in[1,3)$. Meanwhile, we can also obtain the detail convergence rates, that is
\begin{theorem}[\text{\textbf{Inviscid Limit}}]\label{th3} Let (\ref{100001}) hold.  Suppose that $(\rho^\epsilon, u^\epsilon)$ is  the regular solution to
the Cauchy problem \eqref{eq:1.1}--\eqref{eq:2.211} obtained in Theorem $\ref{th2}$,  and $(\rho, u)$ is  the regular
solution to the Cauchy problem \eqref{eq:1.1E} with
\eqref{eq:2.211m}--\eqref{eq:2.211} obtained in Theorem $\ref{thmakio}$. If
\begin{equation}
(\rho^{\epsilon}, u^{\epsilon})|_{t=0}=(\rho, u)|_{t=0}=(\rho_0,
u_0)
\end{equation}
satisfies \eqref{th78}, then 
$\big(\rho^{\epsilon}, u^{\epsilon}\big)$
converges to $(\rho, u)$   as  $\epsilon\to 0$ in the following sense
\begin{equation}\label{shou1}
\lim_{\epsilon \to 0}\sup_{0\leq t\leq
T_*}\left(\Big\|\Big((\rho^{\epsilon})^{\frac{\gamma-1}{2}}-\rho^{\frac{\gamma-1}{2}}\Big)(t)\Big\|_{H^{s'}}+\big\|\big(u^\epsilon
-u\big)(t)\big\|_{H^{s'}}\right)=0,
\end{equation}
for any constant  $s'\in [0, 3)$.  Moreover, we also  have
\begin{equation}\label{decay}
\begin{split}
\sup_{0\leq t \leq T_*}\left(\Big\|\Big((\rho^{\epsilon})^{\frac{\gamma-1}{2}}-\rho^{\frac{\gamma-1}{2}}\Big)(t)\Big\|_1+\|\big(u^\epsilon -u\big)(t)\|_1\right)\leq& C\epsilon,\\
\sup_{0\leq t\leq
T_*}\left(\Big|\Big((\rho^{\epsilon})^{\frac{\gamma-1}{2}}-\rho^{\frac{\gamma-1}{2}}\Big)(t)\Big|_{D^2}+|\big(u^\epsilon
-u)(t)|_{D^2}\right)\leq& C\sqrt{\epsilon},
\end{split}
\end{equation}
where $C>0$ is a constant depending only on the fixed
 constants $A, \delta, \gamma, \alpha, \beta, T_*$ and $ \rho_0, u_0$.

Further more, if  the condition (\ref{regco}) holds, we still have
\begin{equation}\label{decay11}
\begin{split}
\lim_{\epsilon \to 0}\sup_{0\leq t\leq
T_*}\left(\big\|\big(\rho^{\epsilon}-\rho\big)(t)\big\|_{H^{s'}}+\big\|\big(u^\epsilon
-u\big)(t)\big\|_{H^{s'}}\right)=&0,\\
\sup_{0\leq t \leq T_*}\left(\big\|(\rho^{\epsilon}-\rho)(t)\big\|_1+\|\big(u^\epsilon -u\big)(t)\|_1\right)\leq& C\epsilon,\\
\sup_{0\leq t\leq
T_*}\left(\big|\big(\rho^{\epsilon}-\rho\big)(t)\big|_{D^2}+|\big(u^\epsilon
-u)(t)|_{D^2}\right)\leq& C\sqrt{\epsilon}.
\end{split}
\end{equation}

\end{theorem}

\begin{remark}\label{zhunbei1}
It should be pointed out that conclusions of  Theorems {\rm \ref{th2}-\ref{th3} } still hold when viscosities $\mu$ and $\lambda$ are in the general form:
\begin{equation}\label{nianxing}\mu(\rho)= \epsilon \rho^\delta \alpha(\rho),\quad \lambda(\rho)= \epsilon\rho^\delta \beta(\rho),
\end{equation}
where 
$\alpha(\rho)$ and $\beta(\rho)$ are functions of $\rho$, satisfying
 \begin{equation}\label{zhunbei2}
  \big(\alpha(\rho),\beta(\rho)\big)\in C^4(\mathbb{R}^+),\;\;  \alpha(\rho)\geq C_0>0, \ \ \text{and} \ \  2\alpha(\rho)+3\beta(\rho)\geq 0.
 \end{equation}
For    function pairs $(g_1(\rho),g_2(\rho))$ have the form $\sim \rho^p$, it is obviously that they do not belong to $C^4(\mathbb{R}^+)$ when $p<4$ due to  the presence of the vacuum. However, Theorems {\rm \ref{th2}-\ref{th3}} still hold if the initial data satisfy the  additional initial assumptions:
\begin{equation}\label{yanshen}
\begin{cases}
\alpha(\rho)=P_{k_1}(\rho)+1,\quad \beta(\rho)=P_{k_2}(\rho),\\[10pt]
P_{k_1}(\rho_0)\in H^3,\quad P_{k_2}(\rho_0)\in H^3,
\end{cases}
\end{equation}
where $P_{k_i}(\rho)$ $(i=1,2)$ is a $k_i$-th degree  polynomial of $\rho$ with vanishing constant term, and the minimum power of density in all the terms of  $P_{k_i}(\rho)$  should  be greater than or  equal to  $\frac{\delta-1}{2}$.
Further more, our framework  in this paper is applicable to other physical dimensions, say 1 and 2, after  some minor modifications. 
\end{remark}

The rest of the paper is organized as follows: In Section 2,  we list some basic lemmas
that will be used  in our proof.  In Section 3,  based on some uniform  estimates for 
$$\displaystyle\big(\rho^{\frac{\gamma-1}{2}}, u\big) \quad \text{in} \quad H^3, \quad \text{and} \quad \rho^{\frac{\delta-1}{2}} \quad \text{in} \quad H^2,$$
 we will give the proof for  the  uniform
(with respect to $\epsilon$) local-in-time well-posedness of the strong  solution to  the reformulated Cauchy problem (\ref{li41}), which is achieved in the following  four steps: 
  \begin{enumerate}
  
\item Via introducing a  uniform elliptic  operator $\epsilon (\varphi^2+\eta^2)Lu$ with artificial viscosity coefficients $\eta^2>0$ in  momentum equations, the  global well-posedness of the approximation solution to the corresponding   linearized problem  (\ref{li4}) for $(\varphi,\phi,u)$ 
 has been established (Section 3.1).\\
 
\item  We establish the uniform a priori estimates  with respect to  $(\eta, \epsilon)$ for 
$$\displaystyle\big(\phi, u\big) \quad \text{in} \quad H^3, \quad \text{and} \quad \varphi \quad \text{in} \quad H^2,$$
 to the  linearized problem  (\ref{li4})   in $[0,T_*]$, where the time $T_*$ is also  independent of  $(\eta, \epsilon)$ (Section 3.2).\\
\item Via passing  to the limit as $\eta\rightarrow 0$, we obtain the solution of the linearized problem \eqref{li4*},  which allows that the elliptic operator appearing in the reformulated momentum equations is degenerate  (Section 3.3).\\
\item  Based on the uniform analysis for the linearized problem, we prove the uniform
(with respect to $\epsilon$) local-in-time well-posedness of the non-linear reformulated problem through  the Picard iteration approach (Section 3.4).
  \end{enumerate}

According to the uniform  local-in-time well-posedness and a priori estimates (with respect to $\epsilon$)  to non-linearized problem (\ref{li41}) obtained in Section 3, in Section 4, we will give the proof for Theorem  \ref{th2}.  Finally in Section 5, the convergence
rates from  the viscous flow to inviscid flow will be obtained, which is the proof of Theorem  \ref{th3}.

\section{Preliminaries}

In this section, we show some basic lemmas  that will be frequently used in the following  proof.
The first one is the  well-known Gagliardo-Nirenberg inequality.
\begin{lemma}\cite{oar}\label{lem2as}\
For $p\in [2,6]$, $q\in (1,\infty)$, and $r\in (3,\infty)$, there exists some generic constant $C> 0$ that may depend on $q$ and $r$ such that for
$$f\in H^1(\mathbb{R}^3),\quad \text{and} \quad  g\in L^q(\mathbb{R}^3)\cap D^{1,r}(\mathbb{R}^3),$$
 we have
\begin{equation}\label{33}
\begin{split}
&|f|^p_p \leq C |f|^{(6-p)/2}_2 |\nabla f|^{(3p-6)/2}_2,\\[8pt]
&|g|_\infty\leq C |g|^{q(r-3)/(3r+q(r-3))}_q |\nabla g|^{3r/(3r+q(r-3))}_r.
\end{split}
\end{equation}
\end{lemma}
Some special  versions of this inequality can be written as
\begin{equation}\label{ine}\begin{split}
|u|_6\leq C|u|_{D^1},\quad |u|_{\infty}\leq C\|\nabla u\|_{1}, \quad |u|_{\infty}\leq C\|u\|_{W^{1,r}},\quad \text{for} \quad r>3.
\end{split}
\end{equation}

The second one  can be found in Majda \cite{amj}. Here we omit its proof.

\begin{lemma}\cite{amj}\label{zhen1}
Let constants $r$, $a$ and $b$ satisfy the relation
$$\frac{1}{r}=\frac{1}{a}+\frac{1}{b},\quad \text{and} \quad 1\leq a,\ b, \ r\leq \infty.$$  $ \forall s\geq 1$, if $f, g \in W^{s,a} \cap  W^{s,b}(\mathbb{R}^3)$, then we have
\begin{equation}\begin{split}\label{ku11}
&|\nabla^s(fg)-f \nabla^s g|_r\leq C_s\big(|\nabla f|_a |\nabla^{s-1}g|_b+|\nabla^s f|_b|g|_a\big),
\end{split}
\end{equation}
\begin{equation}\begin{split}\label{ku22}
&|\nabla^s(fg)-f \nabla^s g|_r\leq C_s\big(|\nabla f|_a |\nabla^{s-1}g|_b+|\nabla^s f|_a|g|_b\big),
\end{split}
\end{equation}
where $C_s> 0$ is a constant only depending on $s$,  and $\nabla^s f$ ($s>1$) is the set of  all $\partial^\zeta_x f$  with $|\zeta|=s$. Here $\zeta=(\zeta_1,\zeta_2,\zeta_3)\in \mathbb{R}^3$ is a multi-index.
\end{lemma}

The third one will show some compactness results  from  the Aubin-Lions Lemma.
\begin{lemma}\cite{jm}\label{aubin} Let $X_0$, $X$ and $X_1$ be three Banach spaces with $X_0\subset X\subset X_1$. Suppose that $X_0$ is compactly embedded in $X$ and that $X$ is continuously embedded in $X_1$.\\[1pt]

I) Let $G$ be bounded in $L^p(0,T;X_0)$ where $1\leq p < \infty$, and $\frac{\partial G}{\partial t}$ be bounded in $L^1(0,T;X_1)$, then $G$ is relatively compact in $L^p(0,T;X)$.\\[1pt]

II) Let $F$ be bounded in $L^\infty(0,T;X_0)$  and $\frac{\partial F}{\partial t}$ be bounded in $L^p(0,T;X_1)$ with $p>1$, then $F$ is relatively compact in $C(0,T;X)$.
\end{lemma}

The following lemma will be used to show the time continuity for the higher order terms of our solution.
\begin{lemma}\cite{bjr}\label{1}
If $f(t,x)\in L^2([0,T]; L^2)$, then there exists a sequence $s_k$ such that
$$
s_k\rightarrow 0, \quad \text{and}\quad s_k |f(s_k,x)|^2_2\rightarrow 0, \quad \text{as} \quad k\rightarrow+\infty.
$$
\end{lemma}

Next we give some Sobolev inequalities on the interpolation estimate, product estimate,  composite function estimate and so on in the following three lemmas.
\begin{lemma}\cite{amj}\label{gag111}
Let  $u\in H^s$, then for any $s'\in[0,s]$,  there exists  a constant $C_s$ only depending on $s$ such that
$$
\|u\|_{s'} \leq C_s \|u\|^{1-\frac{s'}{s}}_0 \|u\|^{\frac{s'}{s}}_s.
$$
\end{lemma}

\begin{lemma}\cite{amj}\label{gag113}
Let  functions $u,\ v \in H^s$ and $s>\frac{3}{2}$, then  $u\cdot v \in H^s$,  and  there exists  a constant $C_s$ only depending on $s$ such that
$$
\|uv\|_{s} \leq C_s \|u\|_s \|v\|_s.
$$
\end{lemma}

\begin{lemma}\cite{amj}
\begin{enumerate}
\item For functions $f,\ g \in H^s \cap L^\infty$ and $|\nu|\leq s$,   there exists  a constant $C_s$ only depending on $s$ such that
\begin{equation}\label{liu01}
\begin{split}
\|\nabla^\nu (fg)\|_s\leq C_s(|f|_{\infty}|\nabla^s g|_2+|g|_{\infty}|\nabla^s f|_{2}).
\end{split}
\end{equation}

\item Assume that  $g(u)$ is a smooth vector-valued function on $G$, $u(x)$ is a continuous function with $u\in H^s \cap L^\infty$. Then for $s\geq 1$,   there exists  a constant $C_s$ only depending on $s$ such that
\begin{equation}\label{liu02}
\begin{split}
|\nabla^s g(u)|_2\leq C_s\Big \|\frac{\partial g}{\partial u }\Big\|_{s-1}|u|^{s-1}_{\infty}|\nabla^s u|_{2}.
\end{split}
\end{equation}
\end{enumerate}
\end{lemma}

The last  lemma is one  useful tool to improve the weak convergence to a strong one.
\begin{lemma}\cite{amj}\label{zheng5}
If function sequence $\{w_n\}^\infty_{n=1}$ converges weakly in a Hilbert space $X$ to $w$, then $w_n$  converges strongly to $w$ in $X$ if and only if
$$
\|w\|_X \geq \lim \text{sup}_{n \rightarrow \infty} \|w_n\|_X.
$$
\end{lemma}

For simplicity,  by  introducing 
four matrices $A_1, A_2, A_3, B=(b_{ij})=(\textbf{b}_1,\textbf{b}_2,\textbf{b}_3)$, a vector $W=(w_1,w_2,w_3)^\top$,  and  letting $A=(A_1,A_2, A_3)$, we denote
\begin{equation}\begin{cases}\label{E:1.34}\displaystyle
\text{div}A=\sum_{j=1}^3 \partial_jA_j,\quad W^\top B W=\sum_{i,j=1}^3 b_{ij}w_i w_j,\\[8pt]
\displaystyle
 |B|_2^2=B: B=\sum_{i,j=1}^3 b^2_{ij},\quad 
W\cdot B=w_1 \textbf{b}_1+w_2  \textbf{b}_2+w_3\textbf{b}_3.\\
\end{cases}
\end{equation}
The above symbols will be used in the rest of the paper.

\section{Uniform Regularity}

In this section,  we will establish the desired uniform regularity shown in Theorem \ref{th2}.
As the discussion shown in Subsection $1.3$, for this purpose   we need to  consider the following reformulated problem:
\begin{equation}\label{li41}
\begin{cases}
\displaystyle
\varphi_t+u\cdot\nabla\varphi+\frac{\delta-1}{2}\varphi\text{div} u=0,\\[6pt]
\displaystyle
A_0W_t+\sum_{j=1}^3A_j(W) \partial_j W+\epsilon \varphi^2 \mathbb{{L}}(W)=\epsilon \mathbb{{H}}(\varphi)  \cdot \mathbb{{Q}}(W),\\[6pt]
\displaystyle
(\varphi,W)|_{t=0}=(\varphi_0,W_0),\quad x\in \mathbb{R}^3,\\[6pt]
(\varphi,W)\rightarrow (0,0),\quad \text{as}\quad  |x|\rightarrow +\infty, \quad t\geq 0,
 \end{cases}
\end{equation}
where $W=(\phi, u)^\top$ and
\begin{equation} \label{sfana1}
\begin{split}
&(\varphi_0,W_0)=(\varphi, \phi,u)|_{t=0}=(\varphi_0,\phi_0,u_0)=\big(\rho^{\frac{\delta-1}{2}}_0(x), \rho^{\frac{\gamma-1}{2}}_0(x),u_0(x)\big),\quad x\in \mathbb{R}^3.
\end{split}
\end{equation}
The definitions of $A_j$ ($j=0,1,...,3$), $\mathbb{L}$,   $\mathbb{H}$ and $\mathbb{Q}$ can  be found  in   (\ref{sseq:5.2qq})-(\ref{sseq:5.3qq}).

 To prove  Theorem \ref{th2}, our first step is to establish  the following existence of the unique strong solutions  for the reformulated problem (\ref{li41}):
\begin{theorem}\label{ths1}
 If the initial data $( \varphi_0, \phi_0, u_0)$ satisfy
\begin{equation}\label{th78qq}
\begin{split}
& \varphi_0\geq 0, \quad \phi_0\geq 0, \quad (\varphi_0, \phi_0,u_0)\in H^3,
\end{split}
\end{equation}
then there exists a positive time $T_*$ independent of $\epsilon$,  and a unique strong solution $(\varphi, \phi, u)$ in $[0,T_*] \times \mathbb{R}^3$ to the Cauchy problem (\ref{li41}) satisfying
\begin{equation}\label{reg11qq}\begin{split}
& \varphi \in C([0,T_*];H^3),\quad  \phi \in C([0,T_*];H^3),\quad  u\in C([0,T_*]; H^{s'})\cap L^\infty([0,T_*]; H^3), \\
& \varphi \nabla^4 u\in L^2([0,T_*] ; L^2), \quad  u_t \in C([0,T_*]; H^1)\cap L^2([0,T_*] ; D^2),
\end{split}
\end{equation}
for any constant $s' \in[2,3)$. Moreover, we  can also obtain the following uniform estimates:
\begin{equation}\label{jka}
\begin{split}
\sup_{0\leq t \leq T_*}\big(\| \varphi\|^2_{2}+\epsilon |\varphi|^2_{D^3}+\| \phi\|^2_{3}+\| u\|^2_{2}\big)(t)& \\
+\text{ess}\sup_{0\leq t \leq T_*}| u(t)|^2_{D^3}
+\int_{0}^{T_*} \epsilon |\varphi \nabla^4 u|^2_{2}
\text{d}t \leq& C^0,\\
\end{split}
\end{equation}
where $C^0$ is a positive constant depending only on $T_*$, $(\varphi_0, \phi_0,u_0 )$ and the fixed constants $A$, $\delta$, $\gamma$, $\alpha$ and $\beta$, and is independent of $\epsilon$.
\end{theorem}

We will subsequently prove Theorem \ref{ths1} through the following  four Subsections $3.1$-$3.4$, and in the next section,  we will show that this theorem indeed implies  Theorem \ref{th2}.

\subsection{Linearization with an artificial strong elliptic operator}\ \\

Let $T$ be any positive time.  In order to construct the local strong  solutions for the nonlinear problem, we need to consider the following  linearized approximation  problem:
\begin{equation}\label{li4}
\begin{cases}
\displaystyle
\varphi_t+v\cdot\nabla\varphi+\frac{\delta-1}{2}\omega\text{div} v=0,\\[8pt]
\displaystyle
A_0W_t+\sum_{j=1}^3A_j(V) \partial_j W+\epsilon (\varphi^2+\eta^2) \mathbb{{L}}(W)=\epsilon \mathbb{{H}}(\varphi)  \cdot \mathbb{{Q}}(V),\\[8pt]
\displaystyle
(\varphi,W)|_{t=0}=(\varphi_0,W_0),\quad x\in \mathbb{R}^3,\\[8pt]
(\varphi,W)\rightarrow (0,0),\quad \text{as}\quad  |x|\rightarrow +\infty, \quad t\geq 0,
 \end{cases}
\end{equation}
where $\eta \in (0,1]$ is a constant,   $W=(\phi, u)^\top, V=\left(\psi, v\right)^\top$ and $W_0=(\phi_0,u_0)^\top$.
$(\omega, \psi)$ are both known functions and $v=(v^{(1)},v^{(2)},v^{(3)})^\top\in \mathbb{R}^3$ is a known vector satisfying the initial assumption $(\omega, \psi, v)(t=0,x)=(\varphi_0,\phi_0, u_0)$ and :
\begin{equation}\label{vg}
\begin{split}
& \omega\in C([0,T]; H^3), \quad  \omega_t\in C([0,T]; H^2),\quad \psi\in C([0,T]; H^3), \\
& \psi_t\in C([0,T]; H^2),\quad  v\in C([0,T] ; H^{s'})\cap L^\infty([0,T]; H^3),\\
 & \omega \nabla^4 v\in L^2([0,T] ; L^2), \quad  v_t\in C([0,T]; H^1)\cap L^2([0,T]; D^2),
\end{split}
\end{equation}
for any constant $s'\in[2,3)$.
Moreover, we assume that
\begin{equation}\label{th78rr}
\begin{split}
&\varphi_0\geq 0,\quad \phi_0\geq 0,\quad   (\varphi_0, W_0)\in H^3.
\end{split}
\end{equation}

Now  we have  the following global existence of a strong solution $(\varphi, \phi, u)$ to (\ref{li4}) by the standard methods at least when $\eta>0$.

 \begin{lemma}\label{lem1}
 Assume  that the initial data $(\varphi_0, \phi_0,u_0)$ satisfy (\ref{th78rr}).
Then there exists  a unique strong solution $(\varphi,\phi,u)$ in $[0,T]\times \mathbb{R}^3$ to (\ref{li4}) when $\eta>0$ such that
\begin{equation}\label{reggh}\begin{split}
& \varphi \in C([0,T]; H^3), \ \phi \in C([0,T]; H^3), \\
& u\in C([0,T]; H^3) \cap L^2([0,T] ; D^4), \ u_t \in C([0,T]; H^1)\cap L^2([0,T] ; D^2).
\end{split}
\end{equation}
\end{lemma}
\begin{proof}
First, the existence and regularities of a unique solution $\varphi$ in $(0,T)\times \mathbb{R}^3$ to the equation  $(\ref{li4})_1$ can be obtained by the standard theory of transport equation (see \cite{evans}).

Second,   when $\eta>0$, based on the regularities of $\varphi$,  it is not difficult to solve $W$ from the linear  symmetric  hyperbolic-parabolic coupled system $(\ref{li4})_2$
to complete the proof of this lemma (see \cite{evans}).  Here we omit its details.
\end{proof}

In the following two subsections, we first  establish the uniform estimates for $(\phi,u)$ in $H^3$ space with respect to both  $\eta$ and $\epsilon$, then we  pass to the limit for the case: $\eta=0$.

\subsection{A priori estimates  independent of $(\eta, \epsilon)$}\ \\

Let $(\varphi, \phi, u)$ be the unique strong solution to (\ref{li4}) in $[0,T] \times \mathbb{R}^3$ obtained in Lemma \ref{lem1}. In this subsection, we will get some local (in time)  a priori estimates for $(\phi,u)$ in $H^3$ space, which are independent of $(\eta, \epsilon)$ listed in the following  Lemmas \ref{f2}-\ref{s7}.
For this purpose, we fix $T>0$ and a  positive constant $c_0$ large enough such  that
\begin{equation}\label{houmian}\begin{split}
2+\|\varphi_0\|_{3}+\|\phi_0\|_{3}+
\|u_0\|_{3}\leq c_0,
\end{split}
\end{equation}
and
\begin{equation}\label{jizhu1}
\begin{split}
\displaystyle
\sup_{0\leq t \leq T^*}\big(\| \omega(t)\|^2_{1} +\| \psi(t)\|^2_{1}+\| v(t)\|^2_{1})+\int_0^{T^*}{\epsilon}
|\omega \nabla^2 v|_2^2\text{d}t  & \leq c^2_1,\\
\displaystyle
\sup_{0\leq t \leq T^*}\big(| \omega(t)|^2_{D^2} +| \psi(t)|^2_{D^2}+|v(t)|^2_{D^2}\big)+\int_{0}^{T^*} \epsilon |\omega \nabla^3 v|^2_{2}
\text{d}t &\leq c^2_2,\\
\displaystyle
\text{ess}\sup_{0\leq t \leq T^*}\big(| \psi(t)|^2_{D^3} +|v(t)|^2_{D^3}+\epsilon | \omega(t)|^2_{D^3}\big)
+\int_{0}^{T^*}\epsilon |\omega \nabla^4 v|^2_{2}\text{d}t &\leq c^2_3,
\end{split}
\end{equation}
for some time $T^*\in (0,T)$ and constants $c_i$ ($i=1,2,3$) such that
$$1< c_0\leq c_1 \leq c_2 \leq c_3. $$ The constants $c_i$ ($i=1,2,3$) and $T^*$ will be determined later  (see (\ref{dingyi})) and depend only on $c_0$ and the fixed constants  $\alpha$, $\beta$, $\gamma$, $A$, $\delta$ and $T$.

Hereinafter, we use  $C\geq 1$ to denote  a generic positive constant depending only on fixed constants $\alpha$, $\beta$, $\gamma$, $A$, $\delta$ and $T$,  but is  independent of  $(\eta, \epsilon)$, which may be different from line to line. We start from the estimates for $\varphi$.

\begin{lemma}\label{f2} Let $(\varphi,W)$ be the unique strong solution to (\ref{li4}) on $[0,T] \times \mathbb{R}^3$. Then
\begin{equation*}\begin{split}
1+|\varphi(t)|^2_\infty+
\|\varphi(t)\|^2_2\leq Cc^2_0,\quad \epsilon |\varphi(t)|^2_{D^3}\leq& Cc^2_0,\\
|\varphi_t(t)|^2_2 \leq Cc^4_1,\quad |\varphi_t(t)|^2_{D^1} \leq Cc^4_2,\quad \epsilon |\varphi_t(t)|^2_{D^2} \leq& Cc^4_3,
\end{split}
\end{equation*}
for $0\leq t \leq T_1=\min (T^*, (1+c_3)^{-2})$.
\end{lemma}

\begin{proof}
We apply the operator $\partial^\zeta_x$  $(0\leq |\zeta|\leq 3)$ to $(\ref{li4})_1$, and obtain
\begin{equation}\label{guji1}
(\partial^\zeta_x {\varphi})_t+v\cdot \nabla \partial^\zeta_x {\varphi}=-(\partial^\zeta_x (v\cdot \nabla {\varphi})-v\cdot \nabla \partial^\zeta_x {\varphi})-\frac{\delta-1}{2} \partial^\zeta_x (\omega \text{div}v).
\end{equation}
Then multiplying both sides of (\ref{guji1}) by $\partial^\zeta_x {\varphi}$, and integrating over $\mathbb{R}^3$, we get
\begin{equation}\label{guji2}
\frac{1}{2}\frac{d}{dt}|\partial^\zeta_x {\varphi}|^2_2\leq C|\text{div}v|_\infty |\partial^\zeta_x {\varphi}|^2_2+C\Lambda^\zeta_1 |\partial^\zeta_x {\varphi}|_2+C\Lambda^\zeta_2 |\partial^\zeta_x {\varphi}|_2,
\end{equation}
where
$$
\Lambda^\zeta_1=|\partial^\zeta_x (v\cdot \nabla {\varphi})-v\cdot \nabla \partial^\zeta_x {\varphi}|_2,\quad \Lambda^\zeta_2=|\partial^\zeta_x (\omega \text{div}v)|_2.
$$

First, when $|\zeta|\leq 2$, we  consider the term $\Lambda^\zeta_1$ and $\Lambda^\zeta_2$. It follows from  Lemma \ref{lem2as} and  H\"older's inequality that 
\begin{equation}\label{qe1a}
\begin{split}
|\Lambda^\zeta_1|_2
\leq &C(|\nabla v\cdot \nabla{\varphi}|_2+|\nabla v\cdot \nabla^2{\varphi}|_2+|\nabla^2  v\cdot \nabla {\varphi}|_2)\\
\leq& C(|\nabla v|_\infty \|\nabla {\varphi}\|_1+|\nabla^2  v|_3|\nabla {\varphi}|_6)\leq C\|\nabla v\|_2\|\varphi\|_2,\\
|\Lambda^\zeta_2|_2&\leq C\|w\|_2 \|v\|_3,
\end{split}
\end{equation}
which, along with (\ref{guji2})-(\ref{qe1a}), implies that 
\begin{equation}\label{qe5}
\begin{split}
\frac{d}{dt} \|{\varphi}(t)\|_2\leq& C\|\nabla v\|_2\|\varphi\|_2+C\|w\|_2 \|v\|_3,
\end{split}
\end{equation}
Then according to the  Gronwall's inequality, one has
\begin{equation}\label{gb1}\begin{split}
\|{\varphi}(t)\|_2
\leq& \Big(\|\varphi_0\|_{2}+c^2_3t\Big)\exp (Cc_3t)\leq Cc^2_0,
\end{split}
\end{equation}
for $0\leq t\leq T_1=\min\{T^*,(1+c_3)^{-2}\}$.

Second, when $|\zeta|=3$, it follows  from Lemma \ref{lem2as} and  H\"older's inequality that 
\begin{equation}\label{qe2a}
\begin{split}
|\Lambda^\zeta_1|_2
\leq &C(|\nabla v\cdot \nabla^3{\varphi}|_2+|\nabla^2  v\cdot \nabla^2 {\varphi}|_2+|\nabla^3  v\cdot \nabla {\varphi}|_2)
\leq C\|\nabla v\|_2\|\nabla \varphi\|_2,\\
|\Lambda^\zeta_2|_2
\leq& C(|\omega \nabla^4 v|_2+|\nabla \omega\cdot \nabla^3 v|_2+|\nabla^2 \omega\cdot\nabla^2 v|_2+|\nabla^3 \omega\cdot \nabla v|_2)\\
\leq & C|\omega \nabla^4 v|_2+C\|w\|_3 \|v\|_3.\\
\end{split}
\end{equation}
Then combining (\ref{guji2})-(\ref{qe2a}), we arrive at
\begin{equation}\label{qe5}
\begin{split}
\frac{d}{dt} |\nabla^3{\varphi}(t)|_2\leq& C\big( \|\nabla v\|_2\|\nabla \varphi\|_2+|\omega \nabla^4 v|_2+\|w\|_3 \|v\|_3\big)\\
\leq& C(c_3|\nabla^3{\varphi}|_2+c^2_3+c^2_3\epsilon^{-\frac{1}{2}}+|\omega \nabla^4 v|_2),
\end{split}
\end{equation}
which, along with  Gronwall's inequality, implies that
\begin{equation}\label{gb2}\begin{split}
|\varphi(t)|_{D^3}
\leq& \Big(|\varphi_0|_{D^3}+c^2_3t+c^2_3\epsilon^{-\frac{1}{2}}t+\int_0^t |\omega \nabla^4 v|_2\text{d}s\Big) \exp (Cc_3t).
\end{split}
\end{equation}
Therefore, observing that
$$
\int_0^t |\omega \nabla^4 v|_{2}\text{d}s\leq \epsilon^{-\frac{1}{2}}t^{\frac{1}{2}}\Big(\int_0^t |\epsilon^{\frac{1}{2}}\omega \nabla^4 v|^2_{2}\text{d}s\Big)^{\frac{1}{2}}\leq Cc_3t^{\frac{1}{2}}\epsilon^{-\frac{1}{2}},
$$
 from (\ref{gb2}), one can obtain that 
$$|\varphi(t)|_{D^3} \leq C(c_0+\epsilon^{-\frac{1}{2}}), \quad \text{for}\quad 0\leq t \leq T_1.$$

At last,  the estimates for $\varphi_t$ follows  from the relation:
$$\varphi_t=-v\cdot \nabla \varphi-\frac{\delta-1}{2}\omega\text{div} v.$$
For $0\leq t \leq T_1$, we easily have
\begin{equation}\label{zhen6}
\begin{split}
|\varphi_t(t)|_2\leq & C\big(|v(t)|_6|\nabla \varphi(t)|_3+|\omega(t)|_\infty|\text{div} v(t)|_2\big)\leq Cc^2_2,\\
|\varphi_t(t)|_{D^1}\leq & C\big(|v(t)|_\infty|\nabla^2 \varphi(t)|_{2}+|\nabla v(t)|_6|\nabla \varphi(t)|_3+|\omega(t)|_\infty|\nabla^2 v(t)|_{2}\big)\\
&+C|\nabla v(t)|_6|\nabla \omega(t)|_3\leq Cc^2_2,\\
|\varphi_t(t)|_{D^2}\leq & C\big(|v(t)|_\infty|\nabla^3 \varphi(t)|_{2}+|\nabla v(t)|_\infty |\nabla^2 \varphi(t)|_2+|\nabla^2 v(t)|_6|\nabla \varphi(t)|_3\big)\\
&+C\big(|\omega(t)|_\infty|\nabla^3 v(t)|_{2}+|\nabla \omega(t)|_6|\nabla^2 v(t)|_3+|\nabla^2\omega(t)|_2|\nabla v(t)|_{\infty}\big)\\
\leq& C(c^2_3+c_3\epsilon^{-\frac{1}{2}}).
\end{split}
\end{equation}
Thus, we complete
the proof of this lemma.

\end{proof}

Based on Gagliardo-Nirenberg-Sobolev and interpolation inequalities, we firstly present several useful inequalities
\begin{equation}\label{zhen6a}
\begin{cases}
\displaystyle
|\varphi|_\infty\leq C|\varphi|_6^{\frac{1}{2}}|\nabla\varphi|_6^{\frac{1}{2}}\leq C|\nabla\varphi|_2^{\frac{1}{2}}|\nabla^2\varphi|_2^{\frac{1}{2}}\leq Cc_0,\\[8pt]
\displaystyle
|\nabla\varphi|_\infty\leq C|\nabla\varphi|_6^{\frac{1}{2}}|\nabla^2\varphi|_6^{\frac{1}{2}}\leq C|\nabla^2\varphi|_2^{\frac{1}{2}}|\nabla^3\varphi|_2^{\frac{1}{2}}\leq Cc_0 \epsilon^{-\frac{1}{4}},\\[8pt]
\displaystyle
|\nabla\varphi|_3\leq C|\nabla\varphi|_2^{\frac{1}{2}}|\nabla\varphi|_6^{\frac{1}{2}}\leq C|\nabla\varphi|_2^{\frac{1}{2}}|\nabla^2\varphi|_2^{\frac{1}{2}}\leq Cc_0,\\[8pt]
\displaystyle
|\nabla^2\varphi|_3\leq C|\nabla^2\varphi|_2^{\frac{1}{2}}|\nabla^2\varphi|_6^{\frac{1}{2}}\leq C|\nabla^2\varphi|_2^{\frac{1}{2}}|\nabla^3\varphi|_2^{\frac{1}{2}}\leq Cc_0\epsilon^{-\frac{1}{4}},\\[8pt]
\displaystyle
|\nabla\varphi|_6\leq C|\nabla^2\varphi|_2\leq Cc_0,\quad 
|\nabla^2\varphi|_6\leq C|\nabla^3\varphi|_2\leq Cc_0\epsilon^{-\frac{1}{2}},
\end{cases}
\end{equation}
which will be frequently used in the following proof.

Using the notations in  \eqref{E:1.34},  now we  show the estimate for $\|W\|_1$.
 \begin{lemma}\label{4}Let $(\varphi, W)$ be the unique strong solution to (\ref{li4}) on $[0,T] \times \mathbb{R}^3$. Then
\begin{equation*}
\begin{split}
\|W(t)\|^2_1+\epsilon\int_0^t |\sqrt{\varphi^{2}+\eta^{2}} \nabla^2 u|^2_2\text{d}s \leq& Cc^2_0,
\end{split}
\end{equation*}
for $0\leq t\leq T_2=\min\{T_1,(1+c_3)^{-4}\}$.
 \end{lemma}
\begin{proof}
 Applying the operator $\partial^\zeta_x$ to $(\ref{li4})_2$, we have
\begin{equation}\label{zhull1}
\begin{split}
&A_0\partial^\zeta_xW_t+\sum_{j=1}^3A_j(V) \partial_j \partial^\zeta_xW+\epsilon (\varphi^2+\eta^2) \mathbb{{L}}(\partial^\zeta_xW)\\
=&
\mathbb{H}( \varphi)  \cdot \partial^\zeta_x\mathbb{{Q}}(V)-\sum_{j=1}^3\Big(\partial^\zeta_x(A_j(V) \partial_j W\big)- A_j(V) \partial_j \partial^\zeta_xW\Big)
\\
&- \epsilon\Big(\partial^\zeta_x((\varphi^2+\eta^2) \mathbb{{L}}(W))-(\varphi^2+\eta^2) \mathbb{{L}}(\partial^\zeta_xW) \Big)\\
&+ \epsilon\Big(\partial^\zeta_x(\mathbb{H}( \varphi)  \cdot\mathbb{{Q}}(V))- \mathbb{H}( \varphi) \cdot \partial^\zeta_x \mathbb{Q}(V)\Big).
\end{split}
\end{equation}
Then multiplying (\ref{zhull1}) by $\partial^\zeta_x W$ on both sides   and integrating  over $\mathbb{R}^3$ by parts,  we have
\begin{equation}\label{zhu101}
\begin{split}
&\frac{1}{2} \frac{d}{dt}\int \big((\partial^\zeta_x W)^\top A_0 \partial^\zeta_x W )+a_1 \epsilon\alpha  | \sqrt{\varphi^2+\eta^2} \nabla \partial^\zeta_x u |^2_2
\\
&\qquad +a_1 \epsilon(\alpha+\beta)|\sqrt{\varphi^2+\eta^2} \text{div} \partial^\zeta_x u |^2_2\\[6pt]
\displaystyle
=&\int(\partial^\zeta_xW)^\top \text{div}A(V)\partial^\zeta_xW+a_1\epsilon \int  \Big(\nabla \varphi^2  \cdot Q(\partial^\zeta_x v)\Big) \cdot \partial^\zeta_x u\\
&- \frac{\delta-1}{\delta}a_1 \epsilon\int \Big(\nabla (\varphi^2+\eta^2) \cdot Q(\partial^\zeta_x u)\Big) \cdot \partial^\zeta_x u\\
&-\sum_{j=1}^3\int  \big(\partial^\zeta_x(A_j(V) \partial_j W\big)- A_j(V) \partial_j \partial^\zeta_xW\big) \cdot \partial^\zeta_x W\\
&-a_1 \epsilon\int \Big( \partial^\zeta_x ((\varphi^2+\eta^2) Lu)-(\varphi^2+\eta^2) L\partial^\zeta_x u \Big)\cdot \partial^\zeta_x u\\
&+a_1 \epsilon\int  \Big(\partial^\zeta_x(\nabla \varphi^2  \cdot Q(v))-\nabla \varphi^2  \cdot  Q(\partial^\zeta_xv)\Big)\cdot \partial^\zeta_x u
:=\sum_{i=1}^{6} I_i.
\end{split}
\end{equation}

Now we consider the terms on the right-hand side of (\ref{zhu101}) when  $|\zeta|\leq 1$. First, it follows  from Lemmas \ref{lem2as}, \ref{f2}, H\"older's inequality and Young's  inequality that 
\begin{equation}\label{szhu102}
\begin{split}
I_1=&\int(\partial_x^\zeta W)^\top\text{div}A(V)\partial_x^\zeta W\\
\leq& C|\nabla V|_\infty |\partial_x^\zeta W|^2_2
\leq C|\nabla V|^{\frac{1}{2}}_6 |\nabla^2 V|^{\frac{1}{2}}_6 |\partial_x^\zeta W|^2_2\\
\leq& C|\nabla^2 V|^{\frac{1}{2}}_2 |\nabla^3 V|^{\frac{1}{2}}_2 |\partial_x^\zeta W|^2_2
\leq Cc_3|\partial_x^\zeta W|^2_2,\\
I_2=&a_1\epsilon \int  \Big(\nabla \varphi^2  \cdot Q(\partial^\zeta_x v)\Big) \cdot \partial^\zeta_x u\\
\leq & C\epsilon |\varphi|_\infty|\nabla \varphi|_3|\nabla \partial^{\zeta}_xv|_6|\partial^\zeta_x u |_2\leq C\epsilon c^3_3|\partial^\zeta_x u |_2, \\
I_3=& -\frac{\delta-1}{\delta}a_1 \epsilon\int \Big(\nabla (\varphi^2+\eta^2) \cdot Q(\partial^\zeta_x u)\Big) \cdot \partial^\zeta_x u\\
\leq& C\epsilon|\nabla\varphi|_\infty |\varphi \nabla \partial^\zeta_x u |_2|\partial_x^\zeta u|_2\\
\leq & \frac{a_1 \epsilon \alpha}{20}|\sqrt{\varphi^2+\eta^2} \nabla \partial^\zeta_x u |^2_2+Cc^2_0\epsilon^{\frac{1}{2}} |\partial_x^\zeta u|^2_2,\\
I_{4}=&-\sum_{j=1}^3\int\Big(\partial_x^\zeta (A_j(V)\partial_j W)-A_j(V)\partial_j\partial_x^\zeta W\Big)\partial_x^\zeta W\\
\leq& C|\partial_x^\zeta (A_j(V)\partial_j W)-A_j(V)\partial_j\partial_x^\zeta W|_2|\partial_x^\zeta W|_2\\
\leq& C|\nabla V|_\infty|\nabla W|^2_2\leq Cc_3|\nabla W|^2_2,\\
\end{split}
\end{equation}
where we have used the fact \eqref{zhen6a}.

Similarly, for the terms $I_{5}$-$I_{6}$, using \eqref{zhen6a}, one can  obtain that 

\begin{equation}\label{szhu10w}
\begin{split}
I_{5}=&-a_1 \epsilon\int \Big( \partial^\zeta_x ((\varphi^2+\eta^2) Lu)-(\varphi^2+\eta^2) L\partial^\zeta_x u \Big)\cdot \partial^\zeta_x u \\
\leq& C\epsilon |\nabla \varphi|_\infty |\varphi Lu|_2 |\partial^\zeta_x u|_2\\
\leq & \frac{a_1 \epsilon \alpha}{20}|\sqrt{\varphi^2+\eta^2} \nabla^2  u |^2_2+Cc^2_0\epsilon^{\frac{1}{2}}  |\partial^\zeta_x u|^2_2,\\
I_{6}=&a_1 \epsilon\int  \Big(\partial^\zeta_x(\nabla \varphi^2  \cdot Q(v))-\nabla \varphi^2  \cdot  Q(\partial^\zeta_xv)\Big)\cdot \partial^\zeta_x u\\
\leq & C\epsilon \big( |\varphi|_\infty|\nabla v|_\infty|\nabla^2 \varphi|_2+|\nabla \varphi|_6|\nabla \varphi|_3|\nabla v|_\infty\big)|\partial^\zeta_x u|_2\\
\leq & Cc^3_3\epsilon |\partial^\zeta_x u|_2\leq  Cc^3_3\epsilon+ Cc^3_3\epsilon|\partial^\zeta_x u|^2_2.
\end{split}
\end{equation}

Then from (\ref{zhu101})-(\ref{szhu10w}), we have
\begin{equation}\label{zhu10r}
\begin{split}
&\frac{1}{2} \frac{d}{dt}\int \big((\partial^\zeta_x W)^\top A_0 \partial^\zeta_x W )+\frac{1}{2}a_1 \epsilon\alpha  | \sqrt{\varphi^2+\eta^2} \nabla \partial^\zeta_x u |^2_2
\leq C\big(c_3^2+c_3^4\epsilon\big)\|W\|^2_1+Cc_3^2\epsilon,
\end{split}
\end{equation}
which, along with Gronwall's inequality, implies that 
\begin{equation}\label{zhu10222a}
\begin{split}
&\|W(t)\|^2_1+\epsilon\int_0^t |\sqrt{\varphi^{2}+\eta^{2}} \nabla^2 u|^2_2\text{d}s
  \leq C \big( \|W_0\|^2_{1}+c_3^2\epsilon t\big)\exp (C(c^2_3+c_3^4\epsilon)t)\leq Cc^{2}_0,
\end{split}
\end{equation}
for $0\leq t\leq T_2=\min\{T_1,(1+c_3)^{-4}\}$.

\end{proof}

Next we show the estimate for $|W|_{D^2}$.
 \begin{lemma}\label{s6}Let $(\varphi, W)$ be the unique strong solution to (\ref{li4}) on $[0,T] \times \mathbb{R}^3$. Then
\begin{equation*}
\begin{split}
|W(t)|^2_{D^2}+\epsilon\int_0^t |\sqrt{\varphi^{2}+\eta^{2}} \nabla^3 u|^2_2\text{d}s  \leq&  Cc^{2}_0,  ,\quad \text{for} \quad 0 \leq t \leq T_2; \\
 |W_t(t)|^2_{2}+|\phi_t(t)|^2_{D^1}+\int_0^t |\nabla u_t|^2_2 \text{d}s\leq & Cc^6_2,\quad \text{for} \quad 0 \leq t \leq T_2.
\end{split}
\end{equation*}
 \end{lemma}
\begin{proof}We divide the proof into two steps.

\underline{Step\ 1}: the estimate of $|W|_{D^2}$. First we need to  consider the terms on the right-hand side of (\ref{zhu101}) when  $|\zeta| =2$. It follows from Lemmas \ref{lem2as} and  \eqref{zhu101},  H\"older's inequlaity and Young's  inequality that\begin{equation}\label{szhu102as}
\begin{split}
I_1=&\frac{1}{2}\int\text{div}A(V)|\partial_x^\zeta W|^2
\leq C|\text{div}A(V)|_\infty |\partial_x^\zeta W|^2_2
\leq Cc_3|\partial_x^\zeta W|^2_2,\\
I_2=&a_1\epsilon \int  \Big(\nabla \varphi^2  \cdot Q(\partial^\zeta_x v)\Big) \cdot \partial^\zeta_x u
\leq  C\epsilon |\nabla \varphi|_3|\nabla \partial^{\zeta}_xv|_2| \varphi \partial^\zeta_x u |_6\\
\leq & C\epsilon c^2_3\big(c_0\epsilon^{-\frac{1}{4}} |\partial^\zeta_x u |_2+| \varphi \nabla \partial^\zeta_x u|_2\big) \\
\leq & \frac{a_1 \epsilon \alpha}{20}|\sqrt{\varphi^2+\eta^2} \nabla^3  u |^2_2+Cc^2_3\epsilon^{\frac{1}{2}}  |\partial^\zeta_x u|^2_2+C c^4_3\epsilon,\\
\end{split}
\end{equation}
where we have used \eqref{zhen6a}
and
\begin{equation}\label{ghbu}\begin{split}
|\varphi\partial^\zeta_x u|_6\leq& C|\varphi\partial^\zeta_x u|_{D^1}
\leq C\big( |\varphi \nabla \partial^\zeta_x u|_2+|\nabla \varphi|_\infty |\partial^{\zeta}_xu|_2 \big)\\
\leq& C\big( |\varphi \nabla \partial^\zeta_x u|_2+c_0\epsilon^{-\frac{1}{4}} |\partial^{\zeta}_xu|_2 \big).\\
\end{split}\end{equation}

For the term $I_3$, via integration by parts and  \eqref{zhen6a}, one has
\begin{equation}
\label{szhu102aas}\begin{split}
I_3=& -\frac{\delta-1}{\delta}a_1 \epsilon\int \Big(\nabla (\varphi^2+\eta^2) \cdot Q(\partial^\zeta_x u)\Big) \cdot \partial^\zeta_x u\\
\leq& C\epsilon|\nabla\varphi|_\infty |\varphi \nabla \partial^\zeta_x u |_2|\partial_x^\zeta u|_2\\
\leq & \frac{a_1 \epsilon \alpha}{20}|\sqrt{\varphi^2+\eta^2} \nabla \partial^\zeta_x u |^2_2+Cc^2_0\epsilon^{\frac{1}{2}} |\partial_x^\zeta u|^2_2.
 \end{split}
 \end{equation}
 
For the term $I_4$, one gets
\begin{equation}\label{szhu102bas}\begin{split}
I_{4}=&-\sum_{j=1}^3\int\Big(\partial_x^\zeta (A_j(V)\partial_j W)-A_j(V)\partial_j\partial_x^\zeta W\Big)\partial_x^\zeta W\\
\leq& C|\partial_x^\zeta (A_j(V)\partial_j W)-A_j(V)\partial_j\partial_x^\zeta W|_2|\partial_x^\zeta W|_2\\
\leq & C|\nabla V|_\infty|\nabla^2 W|^2_2+C|\nabla^2 V|_3|\nabla W|_6|\partial_x^\zeta W|_2
\leq Cc_3|\nabla^2 W|^2_2.
\end{split}
\end{equation}

For the terms $I_{5}$-$I_6$, using \eqref{zhen6a}and \eqref{ghbu}, one has 
\begin{equation}\label{szhu10was}
\begin{split}
I_{5}=&-a_1 \epsilon\int \Big( \partial^\zeta_x ((\varphi^2+\eta^2) Lu)-(\varphi^2+\eta^2) L\partial^\zeta_x u \Big)\cdot \partial^\zeta_x u \\
\leq& C\epsilon\big( |\varphi \nabla Lu|_2+|\nabla \varphi|_\infty|Lu|_2 \big) |\nabla \varphi|_\infty|\partial^\zeta_x u|_2+C\epsilon |\nabla^2 \varphi|_3|Lu|_2|\varphi \partial^\zeta_x u|_6\\
\leq& C c_0\epsilon^{\frac{3}{4}}\Big(\big( |\varphi \nabla Lu|_2+c_0\epsilon^{-\frac{1}{4}}|Lu|_2 \big)|\partial^\zeta_x u|_2
+\big( |\varphi \nabla \partial^\zeta_x u|_2+c_0\epsilon^{-\frac{1}{2}} |\partial^\zeta_x u|_2 \big)|Lu|_2\Big)\\
\leq & \frac{a_1 \epsilon \alpha}{20}|\sqrt{\varphi^2+\eta^2} \nabla^3  u |^2_2+Cc^2_0\epsilon^{\frac{1}{2}}  |\nabla^2 u|^2_2,\\
I_{6}=&a_1 \epsilon\int  \Big(\partial^\zeta_x(\nabla \varphi^2  \cdot Q(v))-\nabla \varphi^2  \cdot  Q(\partial^\zeta_xv)\Big)\cdot \partial^\zeta_x u\\
\leq & C\epsilon \big( |\varphi|_\infty|\nabla v|_\infty|\nabla^3 \varphi|_2+|\nabla \varphi|_\infty |\nabla^2 \varphi|_2|\nabla v|_\infty\big)|\partial^\zeta_x u|_2\\
&+C\epsilon  |\nabla \varphi|_\infty |\nabla \varphi|_3|\nabla^2 v|_6|\partial^\zeta_x u|_2+C\epsilon|\nabla^2 \varphi|_2|\nabla^2 v|_3|\varphi\partial^\zeta_x u|_6\\
\leq & \frac{a_1 \epsilon \alpha}{20}|\sqrt{\varphi^2+\eta^2} \nabla^3  u |^2_2+Cc^2_3 \epsilon^{\frac{1}{2}} |\nabla^2 u|^2_2+Cc^4_3\epsilon.
\end{split}
\end{equation}

Then from (\ref{zhu101}) and (\ref{szhu102as})-(\ref{szhu10was}), we have
\begin{equation}\label{zhu10ras}
\begin{split}
&\frac{1}{2} \frac{d}{dt}\int \big((\partial^\zeta_x W)^\top A_0 \partial^\zeta_x W )+\frac{1}{2}a_1 \epsilon\alpha  | \sqrt{\varphi^2+\eta^2} \nabla \partial^\zeta_x u |^2_2\\
\leq &Cc^2_3\big(1+\epsilon)|W|^2_{D^2}+Cc^4_3\epsilon,
\end{split}
\end{equation}
which, along with Gronwall's inequaltiy, implies that 
\begin{equation}\label{zhu10222as}
\begin{split}
&|W(t)|^2_{D^2}+\epsilon\int_0^t |\sqrt{\varphi^{2}+\eta^{2}} \nabla\partial_x^\zeta u|^2_2\text{d}s\\
  \leq&C\big(  |W_0|^2_{D^2}+c^4_3\epsilon t\big)\exp (Cc^2_3(1+\epsilon)  t)\leq Cc^{2}_0,\quad \text{for} \quad  0\leq t\leq T_2.
\end{split}
\end{equation}

\underline{Step\ 2}: the estimate for $W_t$.  First, the estimate for $\phi_t$ follows  from:
\begin{equation}\label{time}\phi_t=-v\cdot \nabla \phi-\frac{\gamma-1}{2}\psi\text{div} u.\end{equation}
For $0\leq t \leq T_2$, we easily have
\begin{equation}\label{kaifa}
\begin{split}
|\phi_t(t)|_2\leq & C\big(|v(t)|_6|\nabla \phi(t)|_3+|\psi(t)|_6|\text{div} u(t)|_3\big)\leq Cc^2_1,\\
|\phi_t(t)|_{D^1}\leq & C\big(|v(t)|_\infty|\nabla^2 \phi(t)|_{2}+|\nabla v(t)|_6|\nabla \phi(t)|_3+|\psi(t)|_\infty|\nabla^2 u(t)|_{2}\big)\\
&+C|\nabla u(t)|_6|\nabla \psi(t)|_3\leq Cc^2_2.
\end{split}
\end{equation}

Second, we consider the estimate for $|\partial^\zeta_x u_t|_2$ when $|\zeta|\leq 1$.
From the relation:
$$u_t+v\cdot\nabla u +\frac{2A\gamma}{\gamma-1} \psi \nabla \phi+ \epsilon(\varphi^2+\eta^2) Lu=\epsilon\nabla \varphi^2 \cdot Q(v),$$
one can obtain 
\begin{equation}\label{zhu1019ss}
\begin{split}
|u_t|_2=&\Big |v\cdot\nabla u +\frac{2A\gamma}{\gamma-1} \psi \nabla \phi+\epsilon (\varphi^2+\eta^2) Lu-\epsilon\nabla \varphi^2 \cdot Q(v)\Big|_2\\
\leq &C\Big(|v|_6 |\nabla u|_3+|\psi|_6|\nabla\phi|_3 +\epsilon|\varphi^2+\eta^2|_\infty|u|_{D^2}+\epsilon|\varphi|_\infty|\nabla\varphi|_\infty|\nabla v|_2\Big)
\leq Cc^{3}_1.
\end{split}
\end{equation}
Similarly, for $|u_t|_{D^1}$, using \eqref{ghbu}, we  have
\begin{equation}\label{zhu1020ss}
\begin{split}
|u_t|_{D^1}=& \Big|v\cdot\nabla u +\frac{2A\gamma}{\gamma-1} \psi \nabla \phi+\epsilon (\varphi^2+\eta^2) Lu-\epsilon \nabla \varphi^2 \cdot Q(v)\Big|_{D^1}\\
\leq & C\|v\|_2\|\nabla u\|_1+C|\psi|_\infty|\nabla^2\phi|_2+C|\nabla\psi|_3|\nabla\phi|_6\\
&+C\epsilon|\sqrt{\phi^2+\eta^2}|_\infty|\sqrt{\phi^2+\eta^2} \nabla^3 u|_2+C\epsilon|\varphi \nabla^2 u|_6 |\nabla \varphi|_3+C\epsilon\|\varphi\|^2_2\| v\|_2\\
\leq&  Cc^{3}_2+Cc_0\epsilon|\sqrt{\phi^2+\eta^2} \nabla^3 u|_2,
\end{split}
\end{equation}
which  implies that
$$
\int_{0}^t |u_t|^2_{D^1} \text{d}s \leq C\int_0^t \big(c^{3}_2+c^2_0\epsilon|\sqrt{\phi^2+\eta^2} \nabla^3 u|^2_2\big) \text{d}s\leq Cc^{4}_2,\quad
\text{for} \quad  0\leq t \leq T_2.
$$
\end{proof}

Finally, we give the estimates on the  highest order  terms: $\nabla^3 W$ and $\varphi \nabla^4 u$.
 \begin{lemma}\label{s7}Let $(\varphi, W)$ be the unique strong solution to (\ref{li4}) on $[0,T] \times \mathbb{R}^3$. Then
\begin{equation*}
\begin{split}
| W(t)|^2_{D^3}+\epsilon\int_0^t |\sqrt{\varphi^{2}+\eta^{2}} \nabla^4 u|^2_2\text{d}s  \leq&  Cc^{4}_0,  \quad \text{for} \quad 0 \leq t \leq T_2;   \\
|u_t(t)|^2_{D^1}+|\phi_t(t)|^2_{D^2}+\int_0^t |\nabla^2 u_t|^2_2 \text{d}s\leq & Cc^6_3,\quad \text{for} \quad 0 \leq t \leq T_2.
\end{split}
\end{equation*}
 \end{lemma}
\begin{proof}  We divide the proof into two steps.

\underline{Step\ 1}: the estimate of $|W|_{D^3}$. First we need to  consider the terms on the right-hand side of (\ref{zhu101}) when  $|\zeta| =3$. For the term $I_1$, it is easy to see that
\begin{equation}\label{szhu102}
\begin{split}
I_1=&\int(\partial_x^\zeta W)^\top\text{div}A(V)\partial_x^\zeta W\\
\leq& C|\text{div}A(V)|_\infty |\partial_x^\zeta W|^2_2\leq Cc_3|\partial_x^\zeta W|^2_2.
\end{split}
\end{equation}

For the terms $I_2$-$I_3$, via integration by parts, \eqref{zhen6a} and \eqref{ghbu},   one has
\begin{equation}\label{szhu102ed}
\begin{split}
I_2=&a_1\epsilon \int  \Big(\nabla \varphi^2  \cdot Q(\partial^\zeta_x v)\Big) \cdot \partial^\zeta_x u\\
\leq & C\epsilon \int \Big(|\nabla^2 \varphi^2| |\nabla^3 v| |\nabla^3 u|+|\nabla \varphi^2| |\nabla^3 v| |\nabla^4 u|\Big)\\
\leq & C\epsilon  |\nabla^3 v|_2\big(|\nabla \varphi|^2_\infty|\nabla^3 u|_2+|\nabla \varphi|_\infty| \varphi \nabla^4 u |_2+|\nabla^2 \varphi|_3|\varphi \nabla^3 u|_6\big)\\
\leq & \frac{a_1 \epsilon \alpha}{20}|\sqrt{\varphi^2+\eta^2} \nabla^4  u |^2_2+Cc^4_3\epsilon^{\frac{1}{2}}  |\nabla^3 u|^2_2+C c^4_3\epsilon^{\frac{1}{2}},\\
I_3=& -\frac{\delta-1}{\delta}a_1 \epsilon\int \Big(\nabla (\varphi^2+\eta^2) \cdot Q(\partial^\zeta_x u)\Big) \cdot \partial^\zeta_x u\\
\leq& C\epsilon|\nabla\varphi|_\infty |\varphi \nabla \partial^\zeta_x u |_2|\partial_x^\zeta u|_2\\
\leq & \frac{a_1 \epsilon \alpha}{20}|\sqrt{\varphi^2+\eta^2} \nabla^4_x u |^2_2+Cc^2_0\epsilon^{\frac{1}{2}} |\nabla^3 u|^2_2.
 \end{split}
\end{equation}

For the term $I_4$,   letting
$r=b=2$, $a=\infty$, $f=A_j$, $g=\partial_j W$
 in (\ref{ku11}) of Lemma \ref{zhen1},  one can obtain 
\begin{equation}\label{szhu102qe}\begin{split}
I_{4}=&-\sum_{j=1}^3\int\Big(\partial_x^\zeta (A_j(V)\partial_j W)-A_j(V)\partial_j\partial_x^\zeta W\Big)\partial_x^\zeta W\\
\leq& C\sum_{j=1}^3|\partial_x^\zeta (A_j\partial_j W)-A_j(V)\partial_j\partial_x^\zeta W|_2|\partial_x^\zeta W|_2\\
\leq & C\big(|\nabla V|_\infty|\nabla^3 W|_2+|\nabla^3 V|_2|\nabla W|_\infty\big)|\nabla^3 W|_2\\
\leq & C\big(|\nabla V|_\infty|\nabla^3 W|_2+|\nabla^3 V|_2\|\nabla W\|_2\big)|\nabla^3 W|_2
\leq  Cc_3|\nabla^3 W|^2_2+Cc^3_3.
\end{split}
\end{equation}

For the term $I_{5}$, using \eqref{zhen6a} and \eqref{ghbu}, one gets
\begin{equation}\label{szhu10qe}
\begin{split}
I_{5}=&-a_1 \epsilon\int \Big( \partial^\zeta_x ((\varphi^2+\eta^2) Lu)-(\varphi^2+\eta^2) L\partial^\zeta_x u \Big)\cdot \partial^\zeta_x u \\
\leq &C\epsilon\int \Big(|\nabla^3 \varphi| |Lu| |\varphi\partial^\zeta_x u|+|\nabla \varphi| |\nabla^2 \varphi| |Lu| |\partial^\zeta_x u|+|\nabla \varphi|^2 |\nabla Lu| |\partial^\zeta_x u|\Big)\\
&+C\epsilon\int \Big(|\nabla^2\varphi| |\varphi \nabla Lu| |\partial^\zeta_x u|+|\varphi \nabla^2 Lu| |\nabla \varphi| |\partial^\zeta_x u| \Big)\\
\leq &C\epsilon|\nabla^3 \varphi|_2|\nabla^2u|_{3}|\varphi\nabla ^3 u|_6 +C\epsilon|\nabla \varphi|_\infty |\nabla ^2 \varphi|_3 |\nabla ^2 u|_6 | \nabla ^3 u|_2+C\epsilon|\nabla \varphi|^2_\infty   | \nabla ^3 u|^2_2\\[2pt]
&+C\epsilon|\nabla^2 \varphi|_3|\nabla^3u|_{2}|\varphi\nabla ^3 u|_6+C\epsilon|\nabla \varphi|_\infty  |\varphi \nabla ^4 u|_2 | \nabla ^3 u|_2\\[2pt]
\leq& \frac{a_1 \epsilon \alpha}{20}|\sqrt{\varphi^2+\eta^2} \nabla^4  u |^2_2+ Cc^{2}_0 (1+\epsilon^{\frac{1}{2}})|u|^2_{D^3}+Cc^4_0.
\end{split}
\end{equation}
For the term $I_{6}$, we notice that
\begin{equation*}
\begin{split}
&\partial^\zeta_x(\nabla \varphi^2  \cdot Q(v))-\nabla \varphi^2 \cdot \partial^\zeta_x Q(v)\\
=&\sum_{i,j,k}l_{ijk}\Big(C_{1ijk} \nabla \partial^{\zeta^i}_x \varphi^2\cdot  \partial^{\zeta^j+\zeta^k}_x Q(v) +C_{2ijk} \nabla \partial^{\zeta^j+\zeta^k}_x \varphi^2 \cdot \partial^{\zeta^i}_x Q(v)\Big)+ \nabla \partial^\zeta_x \varphi^2 \cdot Q(v),
\end{split}
\end{equation*}
 where $\zeta=\zeta^1+\zeta^2+\zeta^3$ are three multi-indexes $\zeta^i\in \mathbb{R}^3$ $(i=1,2,3)$ satisfying $|\zeta^i|=1$; $C_{1ijk}$ and  $C_{2ijk}$  are all constants;  $l_{ijk}=1$ if $i$, $j$ and $k$ are different from each other, otherwise $l_{ijk}=0$.
Then, one has 
\begin{equation}\label{szhu1015}
\begin{split}
I_{6}=&a_1 \epsilon\int  \big(\partial^\zeta_x(\nabla \varphi^2 \cdot  Q(v))-\nabla \varphi^2 \cdot \partial^\zeta_x Q(v)\big)\cdot \partial^\zeta_x u\\
=&a_1 \epsilon\int \Big( \sum_{i,j,k}l_{ijk}C_{1ijk} \nabla \partial^{\zeta^i}_x \varphi^2\cdot \partial^{\zeta^j+\zeta^k}_x  Q(v)\Big)\cdot \partial^\zeta_x u\\
&+a_1 \epsilon\int \Big(\sum_{i,j,k}l_{ijk}C_{2ijk} \nabla \partial^{\zeta^j+\zeta^k}_x \varphi^2\cdot  \partial^{\zeta^i}_x  Q(v)\Big)\cdot \partial^\zeta_x u \\
& +a_1 \epsilon\int \nabla \partial^\zeta_x \varphi^2 \cdot Q(v) \cdot \partial^\zeta_x u
\equiv:I_{61}+I_{62}+I_{63}.
\end{split}
\end{equation}
Using  \eqref{zhen6a} and  \eqref{ghbu}, we first consider the first two terms on the right-hand side of (\ref{szhu1015}):
\begin{equation}\label{zhu1018}
\begin{split}
I_{61}=&a_1 \epsilon\int \Big( \sum_{i,j,k}l_{ijk}C_{1ijk} \nabla \partial^{\zeta^i}_x \varphi^2 \partial^{\zeta^j+\zeta^k}_x Q(v)\Big)\cdot \partial^\zeta_x u\\
\leq &C\epsilon|\nabla \varphi|^2_\infty|\nabla^3 u|_2|\nabla^3 v|_2+C\epsilon|\nabla^3 v|_2|\varphi\nabla^3 u|_6 |\nabla^2 \varphi|_3\\
\leq & \frac{a_1\epsilon \alpha}{20}|\sqrt{\varphi^2+\eta^2} \nabla^4 u|^2_2+ Cc^{2}_3\epsilon^{\frac{1}{2}}  |u|^2_{D^3}+Cc^{4}_3\epsilon^{\frac{1}{2}} ,\\
I_{62}=&a_1 \epsilon\int \Big(\sum_{i,j,k}l_{ijk}C_{2ijk} \nabla \partial^{\zeta^j+\zeta^k}_x \varphi^2 \cdot \partial^{\zeta^i}_x Q(v) \Big)\cdot \partial^\zeta_x u\\
\leq& C\epsilon|\nabla^3 \varphi|_2|\nabla^2 v|_3|\varphi\nabla^3 u|_6+C\epsilon|\nabla \varphi|_\infty|\nabla^2 \varphi|_3|\nabla^2 v|_6|\nabla^3 u|_2\\
\leq & \frac{a_1 \epsilon\alpha}{20}|\sqrt{\varphi^2+\eta^2} \nabla^4 u|^2_2+ Cc^{2}_3\epsilon^{\frac{1}{2}}|u|^2_{D^3}+Cc^4_3.
\end{split}
\end{equation}
For the  term $I_{63}$,  it follows from the  integration by parts that 
\begin{equation}\label{szhu1016}
\begin{split}
I_{63}
=&a_1 \epsilon\int \Big( \nabla \partial^\zeta_x \varphi^2 Q(v)\Big) \cdot \partial^\zeta_x u
=-a_1 \epsilon\int \sum_{i=1}^3 \Big(  \partial^{\zeta-\zeta^i}_x \nabla \varphi^2 \cdot \partial^{\zeta^i}_x Q(v) \cdot \partial^\zeta_x u\\
&+ \partial^{\zeta-\zeta^i}_x \nabla \varphi^2 \cdot  Q(v) \cdot  \partial^{\zeta+\zeta^i}_x u\Big)
= \sum_{i=1}^3\big((I^{(i)}_{631}+I^{(i)}_{632}\big).
\end{split}
\end{equation}
For simplicity, we only consider the case that $i=1$,  the rest terms can be estimated similarly. When $i=1$, similarly to  the estimates on $I_{61}$ in  (\ref{zhu1018}), we first  have
\begin{equation}\label{szhu1016}
\begin{split}
I^{(1)}_{631}
\leq\frac{a_1 \epsilon\alpha}{20}|\sqrt{\varphi^2+\eta^2} \nabla^4 u|^2_2+ Cc^{2}_3\epsilon^{\frac{1}{2}}|u|^2_{D^3}+Cc_3^4\epsilon^{\frac{1}{2}}.
\end{split}
\end{equation}
Next, for the term $I_{632}$, one has 
\begin{equation}\label{zhu1016kl}
\begin{split}
I^{(1)}_{632}=&-a_1 \epsilon\int  \partial^{\zeta^2+\zeta^3}_x \nabla \varphi^2 \cdot  Q(v) \cdot  \partial^{\zeta+\zeta^1}_x u\\
=&-2a_1 \epsilon\int \Big(\varphi\partial^{\zeta^2+\zeta^3}_x \nabla \varphi+\partial^{\zeta^2}_x\nabla \varphi \partial^{\zeta^3}_x\varphi\Big) \cdot  Q(v) \cdot  \partial^{\zeta+\zeta^1}_x u\\
&-2a_1 \epsilon\int \Big(\partial^{\zeta^3}_x \nabla \varphi \partial^{\zeta^2}_x\varphi+ \nabla \varphi \partial^{\zeta^2+\zeta^3}_x\varphi\Big) \cdot  Q(v) \cdot  \partial^{\zeta+\zeta^1}_x u\\
=&I_A+I_B+I_C+I_D.
\end{split}
\end{equation}
For the term $I_A$, it is not hard to show that
\begin{equation}\label{szhu1016kl}
\begin{split}
I_A=&-2a_1 \epsilon\int \Big(\varphi\partial^{\zeta^2+\zeta^3}_x \nabla \varphi \Big) \cdot  Q(v) \cdot  \partial^{\zeta+\zeta^1}_xu\\
\leq &C\epsilon|\nabla^3 \varphi|_2|\nabla v|_\infty|\varphi\nabla^4 u|_2
\leq \frac{a_1 \epsilon\alpha}{20}|\sqrt{\varphi^2+\eta^2}\nabla^4 u|^2_2+Cc_3^4.
\end{split}
\end{equation}
For the term $I_B$,  it follows from integration by parts  that
\begin{equation}\label{sszhu1016kls}
\begin{split}
I_B=&-2a_1 \epsilon\int \Big(\partial^{\zeta^2}_x\nabla \varphi \partial^{\zeta^3}_x\varphi)\Big) \cdot  Q(v) \cdot  \partial^{\zeta+\zeta^1}_xu\\
\leq & C\epsilon\int \Big(\big(|\nabla \varphi| |\nabla^3 \varphi|+|\nabla^2 \varphi|^2\big)|\nabla^3 u||\nabla v|+|\nabla^2 \varphi||\nabla \varphi||\nabla^3 u||\nabla^2 v|\Big)\\
\leq& C\epsilon|\nabla^3 u|_2|\nabla v|_\infty\big(|\nabla^3 \varphi|_2|\nabla \varphi|_\infty+|\nabla^2 \varphi|_3|\nabla^2 \varphi|_6\big)\\
&+C\epsilon|\nabla^3 u|_2|\nabla^2 v|_3|\nabla^2 \varphi|_6|\nabla \varphi|_\infty
\leq  Cc^3_3\epsilon^{\frac{1}{4}}|\nabla^3 u|_2.
\end{split}
\end{equation}
Via the similar argument used in (\ref{sszhu1016kls}), we also have
\begin{equation}\label{sszhu1016kl}
\begin{split}
I_C+I_{D}
\leq  Cc^3_3\epsilon^{\frac{1}{4}}|\nabla^3 u|_2,
\end{split}
\end{equation}
which, along with (\ref{szhu1015})-(\ref{szhu1016kl}), implies that
\begin{equation}\label{szhu10177}
\begin{split}
I_{6}\leq \frac{a_1\epsilon\alpha}{20}|\sqrt{\varphi^2+\eta^2} \nabla^4 u|^2_2+ Cc^4_3 |u|^2_{D^3}+Cc^{4}_3.
\end{split}
\end{equation}

Then from (\ref{szhu102})-(\ref{szhu10qe}) and (\ref{szhu10177}), one has 
\begin{equation}\label{zhu101qw}
\begin{split}
&\frac{1}{2} \frac{d}{dt}\int \big((\partial^\zeta_x W)^\top A_0 \partial^\zeta_x W )+\frac{1}{2}a_1 \epsilon\alpha  | \sqrt{\varphi^2+\eta^2} \nabla^4 u |^2_2
\leq Cc^4_3|W|^2_{D^3}+Cc^4_3,
\end{split}
\end{equation}
which, along with Gronwall's inequaltiy, implies that 
\begin{equation}\label{zhu10222}
\begin{split}
| W(t)|^2_{D^3}+\epsilon\int_0^t |\sqrt{\varphi^{2}+\eta^{2}} \nabla^4 u|^2_2\text{d}s  \leq  C\big(| W_0|^2_{D^3}+c^{4}_3 t\big)\exp (Cc^{4}_3t)\leq Cc^{2}_0,
\end{split}
\end{equation}
for $0\leq t\leq T_2$.

\underline{Step\ 2}: the estimates for $|\partial^\zeta_x W_t|_2$  when $1\leq |\zeta|\leq 2$.
First from (\ref{time}), we have
\begin{equation}\label{kaifa2}
\begin{split}
|\phi_t(t)|_{D^2}\leq & C\big(|v(t)|_\infty|\nabla^3 \phi(t)|_{2}+|\nabla v(t)|_6|\nabla^2 \phi(t)|_3+|\nabla^2 v|_6|\nabla \phi|_3\big)\\
&+C\big(|\psi(t)|_\infty|\nabla^3 u(t)|_{2}+|\nabla \psi(t)|_\infty|\nabla^2 u(t)|_2+|\nabla^2 \psi|_3|\text{div}u|_6\big)
\leq Cc^2_3.
\end{split}
\end{equation}

Second, it follows  from (\ref{zhu1020ss}) that 
\begin{equation}\label{szhu1020ss}
\begin{split}
|u_t|_{D^1}\leq& Cc^{3}_2+Cc_0|\sqrt{\phi^2+\eta^2} \nabla^3 u|_2\leq Cc^3_2.
\end{split}
\end{equation}

At last, for $|W_t|_{D^2}$,  from Lemma \ref{gag113}, one gets
\begin{equation}\label{zhu1020ss1}
\begin{split}
|u_t|_{D^2}=&  \Big|v\cdot\nabla u +\frac{2A\gamma}{\gamma-1} \psi \nabla \phi+\epsilon (\varphi^2+\eta^2) Lu-\epsilon \nabla \varphi^2 \cdot Q(v)\Big|_{D^2}\\
\leq & C\big(\|v\|_2\|\nabla u\|_2+\|\psi\|_2\|\nabla \phi\|_2\big)+C\epsilon|\sqrt{\varphi^{2}+\eta^{2}}|_\infty|\sqrt{\varphi^{2}+\eta^{2}} \nabla^4 u|_2\\
&+C\epsilon|\varphi|_\infty|\nabla \varphi|_\infty|u|_{D^3}+C\epsilon\big(|\varphi|_\infty |\nabla^2 \varphi|_3|Lu|_6+|\nabla \varphi|_6|\nabla \varphi|_6|u|_6\big)\\
&+C\epsilon\|\nabla \varphi^2\|_2\|\nabla v\|_2
\leq Cc^4_3+Cc_0\epsilon |\sqrt{\varphi^2+\eta^2} \nabla^4 u|_2,
\end{split}
\end{equation}
which  implies that
$$
\int_{0}^{t} |u_t|^2_{D^2} \text{d}s \leq C\int_0^{t} \big(c^{8}_3+\epsilon^2 c^2_0 |\sqrt{\phi^{2}+\eta^{2}} \nabla^4 u|^2_2\big) \text{d}s\leq Cc^{4}_3,\quad \text{for}\quad 0\leq t \leq T_2.
$$
\end{proof}

Combining the estimates obtained in Lemmas \ref{f2}-\ref{s7}, we have
\begin{equation}\label{jkkll}
\begin{split}
1+|\varphi(t)|^2_\infty+
\|\varphi(t)\|^2_2\leq Cc^2_0,\quad \epsilon |\varphi|^2_{D^3}\leq& Cc^2_0,\\
|\varphi_t(t)|^2_2 \leq Cc^4_1,\quad |\varphi_t(t)|^2_{D^1} \leq Cc^4_2,\quad \epsilon |\varphi_t(t)|^2_{D^2} \leq& Cc^4_3,\\
\|W(t)\|^2_1+\epsilon\int_0^t |\varphi \nabla^2 u|^2_2\text{d}s \leq& Cc^2_0,\\
|W(t)|^2_{D^2}+\epsilon\int_0^t |\varphi \nabla^3 u|^2_2\text{d}s  \leq&  Cc^{2}_0,\\
 |W_t(t)|^2_{2}+|\phi_t(t)|^2_{D^1}+\int_0^t |\nabla u_t|^2_2 \text{d}s\leq & Cc^6_2,\\
| W(t)|^2_{D^3}+\epsilon\int_0^t |\varphi \nabla^4 u|^2_2\text{d}s  \leq&  Cc^{2}_0,\\
|u_t(t)|^2_{D^1}+|\phi_t(t)|^2_{D^2}+\int_0^t |\nabla^2 u_t|^2_2 \text{d}s\leq & Cc^6_3,
\end{split}
\end{equation}
for $0 \leq t \leq \min\{T, (1+c_3)^{-4}\}$.
Therefore, if we define the constants $c_i$ ($i=1,2,3$) and $T^*$ by
\begin{equation}\label{dingyi}
\begin{split}
&c_1=c_2=c_3=C^{\frac{1}{2}}c_0,  \quad T^*=\min\{T, (1+c_3)^{-4}\},
\end{split}
\end{equation}
then we deduce that
\begin{equation}\label{jkk}
\begin{split}
\sup_{0\leq t \leq T^*}\big(\| \varphi(t)\|^2_{1} +\| \phi(t)\|^2_{1}+\| u(t)\|^2_{1})+\int_0^{T^*}{\epsilon}
|\varphi \nabla^2 u|_2^2& \leq c^2_1,\\
\displaystyle
\sup_{0\leq t \leq T^*}\big(| \varphi(t)|^2_{D^2} +| \phi(t)|^2_{D^2}+|u(t)|^2_{D^2}\big)+\int_{0}^{T^*} \epsilon |\varphi \nabla^3 u|^2_{2}
\text{d}t &\leq c^2_2,\\
\displaystyle
 \text{ess}\sup_{0\leq t \leq T^*}\big(| \phi(t)|^2_{D^3} +|u(t)|^2_{D^3}+\epsilon | \varphi(t)|^2_{D^3}\big)
+\int_{0}^{T^*}\epsilon |\varphi \nabla^4 u|^2_{2}\text{d}t &\leq c^2_3,\\
\displaystyle
\text{ess}\sup_{0\leq t \leq T^*}\big(\| W_t(t)\|^2_{1} +|\phi_t(t)|^2_{D^2}+\epsilon | \varphi_t(t)|^2_{D^2}\big)
+\int_{0}^{T^*} |u_t|^2_{D^2}\text{d}t &\leq c^6_3.
\end{split}
\end{equation}
In other words, given fixed $c_0$ and $T$, there exist positive
constants $T^*$ and $c_i$ ($i=1,2,3$), depending solely on $c_0$, $T$ and the
generic constant $C$, independent of $(\eta,\epsilon)$, such that if
\eqref{jizhu1} holds for $\omega$ and  $V$, then \eqref{jkk} holds for the strong 
solution of \eqref{li4} in $[0, T^*]\times \mathbb{R}^3$.

\subsection{Passing to the  limit as  $\eta\rightarrow 0$}\ \\

Based on the local (in time) a priori estimates (\ref{jkk}), we have the following existence result under the assumption that $\varphi_0\geq 0$ to the following Cauchy problem:
\begin{equation}\label{li4*}
\begin{cases}
\displaystyle
\varphi_t+v\cdot\nabla\varphi+\frac{\delta-1}{2}\omega\text{div} v=0,\\[8pt]
\displaystyle
A_0W_t+\sum_{j=1}^3A_j(V) \partial_j W+\epsilon \varphi^2 \mathbb{{L}}(W)=\epsilon \mathbb{{H}}(\varphi)  \cdot \mathbb{{Q}}(V),\\[8pt]
\displaystyle
(\varphi, W)|_{t=0}=(\varphi_0, W_0),\quad x\in \mathbb{R}^3,\\[8pt]
(\varphi, W)\rightarrow (0
,0),\quad \text{as}\quad  |x|\rightarrow +\infty, \quad t\geq 0.
 \end{cases}
\end{equation}
\begin{lemma}\label{lem1q}
 Assume $(\varphi_0, W_0)$ satisfy (\ref{th78rr}).
Then there exists a time $T^*>0$ that is independent of $\epsilon$,  and a  unique strong solution $(\varphi, W)$ in $[0,T^*]\times \mathbb{R}^3$ to (\ref{li4*}) such that
\begin{equation}\label{regghq}\begin{split}
& \varphi \in C([0,T^*]; H^3),\quad  \phi \in C([0,T^*]; H^3),\quad   u\in C([0,T^*]; H^{s'})\cap L^\infty([0,T^*]; H^3), \\
& \varphi \nabla^4 u \in L^2([0,T^*] ; L^2),\quad  u_t \in C([0,T^*]; H^1)\cap L^2([0,T^*] ; D^2),
\end{split}
\end{equation}
for any constant $s'\in[2,3)$. Moreover,  $(\varphi, W)$ also satisfies the  a priori estimates (\ref{jkk}).
\end{lemma}
\begin{proof} We prove the existence, uniqueness and time continuity in three steps.

\underline{Step 1}: existence.
Due to Lemma \ref{lem1} and the uniform estimates  \eqref{jkk}, for every $\eta>0$,  there exists a unique strong solution $(\varphi^\eta, W^\eta)$ in $[0,T^*]\times \mathbb{R}^3$ to the linearized problem (\ref{li4}) satisfying   estimates (\ref{jkk}), where the time $T^*>0$ is also  independent of $(\eta, \epsilon)$.

By virtue of the uniform estimates  (\ref{jkk}) independent of  $(\eta,\epsilon)$ and  the compactness in Lemma \ref{aubin} (see \cite{jm}), we know that for any $R> 0$,  there exists a subsequence of solutions (still denoted by) $(\varphi^\eta, W^\eta)$, which  converges  to  a limit $(\varphi,W)=(\varphi, \phi, u)$ in the following  strong sense:
\begin{equation}\label{ert}\begin{split}
&(\varphi^\eta, W^\eta) \rightarrow (\varphi, W) \quad \text{ in } \ C([0,T^*];H^2(B_R)),\quad \text{as}\ \eta\rightarrow 0.
\end{split}
\end{equation}

Again by virtue of the uniform estimates (\ref{jkk}) independent of  $(\eta,\epsilon)$, we also know that  there exists a subsequence of solutions (still denoted by) $(\varphi^\eta, W^\eta)$, which    converges to  $(\varphi, W)$ in the following  weak  or  weak* sense:
\begin{equation}\label{ruojixian}
\begin{split}
( \varphi^\eta, W^\eta)\rightharpoonup  (\varphi, W) \quad &\text{weakly* \ in } \ L^\infty([0,T^*];H^3(\mathbb{R}^3)),\\
(\phi^\eta_t,\varphi^\eta_t)\rightharpoonup ( \phi_t,\varphi_t) \quad &\text{weakly* \ in } \ L^\infty([0,T^*];H^2(\mathbb{R}^3)),\\
 u^\eta_t\rightharpoonup  u_t \quad &\text{weakly* \ in } \ L^\infty([0,T^*];H^1(\mathbb{R}^3)),\\
 u^\eta_t\rightharpoonup  u_t \quad &\text{weakly \ \  in } \ \ L^2([0,T^*];D^2(\mathbb{R}^3)),
\end{split}
\end{equation}
which, along with  the lower semi-continuity of weak convergence, implies  that $(\varphi, W)$ also satisfies the corresponding  estimates (\ref{jkk}) except those of $\varphi \nabla^4 u$.

Combining  the strong convergence in (\ref{ert}) and the weak convergence in (\ref{ruojixian}), we easily obtain that  $(\varphi, W) $ also satisfies the local estimates (\ref{jkk}) and
\begin{equation}\label{ruojixian1}
\begin{split}
\varphi^\eta \nabla^4 u^\eta \rightharpoonup  \varphi\nabla^4 u \quad &\text{weakly \ in } \ L^2([0,T^*]\times \mathbb{R}^3).
\end{split}
\end{equation}

Now we  want to show that $(\varphi, W) $ is a weak solution in the sense of distributions  to the linearized problem (\ref{li4*}).
Multiplying $(\ref{li4})_2$ by  test function  $f(t,x)=(f^1,f^2,f^3)\in C^\infty_c ([0,T^*)\times \mathbb{R}^3)$ on both sides, and integrating over $[0,T^*]\times \mathbb{R}^3$, we have

\begin{equation}\label{zhenzheng}
\begin{split}
&\int_0^t \int_{\mathbb{R}^3} u^\eta\cdot f_t \text{d}x\text{d}s-\int_0^t \int_{\mathbb{R}^3} (v\cdot \nabla) u^\eta \cdot f \text{d}x\text{d}s-\int_0^t \int_{\mathbb{R}^3}  \frac{2A\gamma}{\gamma-1}\psi \nabla \phi^\eta f \text{d}x\text{d}s\\
=&-\int u_0 \cdot f(0,x)+\int_0^t \int_{\mathbb{R}^3}\epsilon\Big( ((\varphi^\eta)^2+\eta^2) Lu^\eta- \nabla (\varphi^\eta)^2 \cdot Q(v)\Big) \cdot f \text{d}x\text{d}s.
\end{split}
\end{equation}
Combining  the strong convergence in (\ref{ert}) and  the weak convergences in (\ref{ruojixian})-(\ref{ruojixian1}),  and  letting $\eta\rightarrow 0$ in (\ref{zhenzheng}),  we  have
\begin{equation}\label{zhenzhengxx}
\begin{split}
&\int_0^t \int_{\mathbb{R}^3} u \cdot f_t \text{d}x\text{d}s-\int_0^t \int_{\mathbb{R}^3} (v\cdot \nabla) u \cdot f \text{d}x\text{d}s-\frac{2A\gamma}{\gamma-1}\int_0^t \int_{\mathbb{R}^3}\psi \nabla \phi f \text{d}x\text{d}s\\
=&-\int u_0 \cdot f(0,x)+\int_0^t \int_{\mathbb{R}^3} \epsilon\Big(\varphi^2 Lu-\nabla \varphi^2 \cdot Q(v) \Big) \cdot f \text{d}x\text{d}s.
\end{split}
\end{equation}
So it is obvious that $(\varphi, W) $ is a weak solution in the sense of distributions  to the linearized problem (\ref{li4*}), satisfying the  regularities
\begin{equation}\label{zheng}
\begin{split}
& \varphi  \in L^\infty([0,T^*]; H^3),\ \varphi _t \in L^\infty([0,T^*]; H^2),\ \phi \in L^\infty([0,T^*]; H^3),   \\
&   \phi _t \in L^\infty([0,T^*]; H^2),  \  u\in L^\infty([0,T^*]; H^3),\\
& \varphi \nabla^4 u \in L^2([0,T^*] ; L^2), \quad u_t \in L^\infty([0,T^*]; H^1)\cap L^2([0,T^*] ; D^2).
\end{split}
\end{equation}

\underline{Step 2}: uniqueness.  Let $(\varphi_1,W_1)$ and $(\varphi_2,W_2)$ be two solutions obtained in the above step. We denote
$$
\overline{\varphi}=\varphi_1-\varphi_2,\quad \overline{W}=W_1-W_2.
$$
Then from $(\ref{li4*})_{1}$, we have
$$
\overline{\varphi}_t+v\cdot\nabla\overline{\varphi}=0,
$$
which quickly implies that $\overline{\varphi}=0$.

Let
$\overline{W}=(\overline{\phi},\overline{u})^\top$, from $(\ref{li4*})_{2}$ and $\varphi_1=\varphi_2$, we have
\begin{equation}\label{qaq}
\begin{split}
A_0\overline{W}_t+\sum_{j=1}^3A_j(V) \partial_j\overline{W}&=-\epsilon  \varphi^2_1 L(\overline{W}).
\end{split}
\end{equation}
Then multiplying  (\ref{qaq}) by $\overline{W}$  on both sides, and integrating    over $\mathbb{R}^3$,  we have
 \begin{equation}\label{bq}
\begin{split}
\frac{1}{2}\frac{d}{dt} \int \overline{W}^\top A_0\overline{W} +a_1\epsilon\alpha|\varphi_1 \nabla \overline{u}|^2_2
\leq& C|\nabla V|_\infty |\overline{W}|^2_2+\epsilon|\overline{u}|_2 |\nabla \varphi_1|_\infty |\varphi_1 \nabla \overline{u}|_2\\
\leq & \frac{a_1\epsilon\alpha}{10}|\varphi_1 \nabla \overline{u}|^2_2+C(|\nabla V|_\infty+\epsilon|\nabla \varphi_1|^2_\infty)|\overline{W}|^2_2.
\end{split}
\end{equation}
From Gronwall's inequality, we  obtain that $\overline{W}=0$ in $\mathbb{R}^3$, which gives the uniqueness.

\underline{Step 3}: time continuity.   First for $ \varphi$,   via the regularities shown in (\ref{zheng}) and the classical Sobolev imbedding theorem, we  have
\begin{equation}\label{liu02}
 \varphi \in C([0,T^*];H^2) \cap C([0,T^*]; \text{weak}-H^3).
\end{equation}

Using the same arguments as  in Lemma \ref{f2}, we have
\begin{equation}\label{qgb}\begin{split}
&\|\varphi(t)\|^2_3
\leq \Big(\|\varphi_0\|^2_{3}+C\int_0^t \big(\|\nabla \omega\|^2_2\|v\|^2_3+|w\nabla^4 v|^2_2\big)\text{d}s\Big) \exp\Big(C\int_0^t \| v(s)\|_{3}\text{d}s\Big),\\
\end{split}
\end{equation}
which implies that
\begin{equation}\label{liu03}\begin{split}
\displaystyle
&\lim \text{sup}_{t\rightarrow 0} \|\varphi(t)\|_3 \leq \|\varphi_0\|_3.\\
\end{split}\end{equation}
Then according to  Lemma \ref{zheng5} and  (\ref{liu02}), we know that  $\varphi$ is right continuous at $t=0$ in $H^3$ space. From the reversibility on the time to equation $(\ref{li4*})_1$, we know
\begin{equation}\label{xian}
\varphi \in C([0,T^*];H^3).
\end{equation}

For $\varphi_t$, from
\begin{equation}\label{liu04}\begin{split}
\varphi_t&=-v\cdot \nabla \varphi-\frac{\delta-1}{2}\omega \text{div}v,\\
\end{split}\end{equation}
 we only need to consider the term $\omega\text{div}v$.
Due to
\begin{equation}\label{liu05}
\omega \text{div} v \in L^2([0,T^*];H^3),\quad (\omega \text{div} v)_t  \in L^2([0,T^*];H^1),
\end{equation}
and  the Sobolev imbedding theorem, we have
\begin{equation}\label{liu06}
\omega \text{div}v\in C([0,T^*];H^2),
\end{equation}
which implies that
$$
 \varphi_t \in C([0,T^*];H^2).
$$
The similar arguments can be used to deal with $\phi$.

For velocity $u$, from the regularity shown in (\ref{zheng}) and Sobolev's imbedding theorem, we obtain that
\begin{equation}\label{zheng1}
\begin{split}
 u\in C([0,T^*]; H^2)\cap  C([0,T^*]; \text{weak}-H^3).
\end{split}
\end{equation}

Then from  Lemma \ref{gag111},  for any $s'\in [2,3)$, we have

$$
\|u\|_{s'} \leq C_3 \|u\|^{1-\frac{s'}{3}}_0 \|u\|^{\frac{s'}{3}}_3.
$$
Together with the upper bound shown in (\ref{jkk}) and the time continuity (\ref{zheng1}), we have
\begin{equation}\label{zheng2}
\begin{split}
 u\in C([0,T^*]; H^{s'}).
\end{split}
\end{equation}

Finally, we consider $u_t$.  From equations   $(\ref{li4*})_2$ we have
\begin{equation}\label{zheng3}
u_t=-v\cdot\nabla u -\frac{2A\gamma}{\gamma-1}\psi\nabla \phi+\epsilon\varphi^{2} Lu+\epsilon\nabla \varphi^{2}\cdot Q(v),
\end{equation}
where
$$
Q(v)=\frac{\delta}{\delta-1}\left(\alpha(\nabla v+(\nabla v)^\top)+\beta \text{div}v \mathbb{I}_3\right)\in L^2([0,T^*];H^2).
$$

From (\ref{zheng}), we have
\begin{equation}\label{zheng4}
\epsilon\varphi^{2} Lu\in L^2([0,T^*];H^2), \quad \epsilon(\varphi^{2} Lu)_t\in L^2([0,T^*];L^2),
\end{equation}
which means that
\begin{equation}\label{gong4}
\varphi^{2} Lu\in C([0,T^*];H^1).
\end{equation}
Then combining (\ref{vg}), (\ref{xian}), (\ref{zheng2}) and  (\ref{gong4}), we deduce that
$$
u_t \in C([0,T^*];H^1).
$$

\end{proof}

\subsection{Proof of Theorem \ref{ths1}}  \ \\

Our proof  is based on the classical iteration scheme and the existence results for the linearized problem obtained  in  Section $3.4$. Like in Section 3.3, we define constants $c_{0}$ and  $c_{i}$ ($i=1,2,3$), and assume that
\begin{equation*}\begin{split}
&2+\|\varphi_0\|_{3}+\|W_0\|_3\leq c_0.
\end{split}
\end{equation*}
Let $(\varphi^0, W^0=(\phi^0,u^0))$ with the regularities
\begin{equation}\label{zheng6}
\begin{split}
&\varphi^0  \in C([0,T^*];H^{3}),\quad \phi^0  \in C([0,T^*];H^{3}),\quad \varphi^0 \nabla^4 u^0 \in L^2([0,T^*];L^2), \\
& u^0\in C([0,T^*];H^{s'})\cap  L^\infty([0,T^*];H^3) \quad \text{for\ any } \ s' \in [2,3)
\end{split}
\end{equation}
be the solution to the  problem
\begin{equation}\label{zheng6}
\begin{cases}
X_t+u_0 \cdot\nabla X=0 \quad \text{in} \quad (0,+\infty)\times \mathbb{R}^3,\\[10pt]
Y_t+u_0 \cdot \nabla Y=0 \quad \text{in} \quad (0,+\infty)\times \mathbb{R}^3,\\[10pt]
Z_t- X^2\triangle Z=0 \quad \ \text{in} \quad  (0,+\infty)\times \mathbb{R}^3 ,\\[10pt]
(X,Y,Z)|_{t=0}=(\varphi_0,\phi_0,u_0) \quad \text{in} \quad \mathbb{R}^3,\\[10pt]
(X,Y,Z)\rightarrow (0,0,0) \quad \text{as } \quad |x|\rightarrow +\infty,\quad t\geq 0.
\end{cases}
\end{equation}
Taking a  time $T^{**}\in (0,T^*]$ small enough such that
\begin{equation}\label{jizhu}
\begin{split}
&\sup_{0\leq t \leq T^{**}}\big(|| \varphi^0(t)||^2_{1} +|| \phi^0(t)||^2_{1}+||u^0(t)||^2_{1}\big)+\int_0^{T^{**}}\epsilon|\varphi^0\nabla^2u^0|dt\leq c_1^2, \\
&\sup_{0\leq t \leq T^{**}}\big(| \varphi^0(t)|^2_{D^2} +| \phi^0(t)|^2_{D^2}+|u^0(t)|^2_{D^2}\big)+\int_0^{T^{**}}\epsilon|\varphi^0\nabla^3u^0|dt\leq c_2^2, \\
&\text{ess}\sup_{0\leq t \leq T^{**}}\big(\epsilon| \varphi^0(t)|^2_{D^3} +| \phi^0(t)|^2_{D^3}+|u^0(t)|^2_{D^3}\big)+\int_0^{T^{**}}\epsilon|\varphi^0\nabla^4u^0|dt\leq c_3^2. \\
\end{split}
\end{equation}

\begin{proof} We prove the existence, uniqueness and time continuity in three steps.

\underline{Step 1}: existence. Let $(\omega,\psi, v)=(\varphi^0,\phi^0,u^0)$,  we define $(\varphi^1, W^1)$ as a strong solution to problem (\ref{li4*}). Then we construct approximate solutions 
$$(\varphi^{k+1}, W^{k+1})=(\varphi^{k+1}, \phi^{k+1},u^{k+1})$$
 inductively, as follows: assuming that $(\varphi^{k}, W^{k})$ was defined for $k\geq 1$, let $(\varphi^{k+1},W^{k+1})$  be the unique solution to problem (\ref{li4*}) with $(\omega,\psi,v)$ replaced by $ (\varphi^k, W^k)$ as following:

\begin{equation}\label{li6}
\begin{cases}
\displaystyle
 \varphi^{k+1}_t+u^{k}\cdot \nabla \varphi^{k+1}+\frac{\delta-1}{2}\varphi^{k}\text{div} u^{k}=0,\\[12pt]
\displaystyle
 A_0 W^{k+1}_t+\sum_{j=1}^3A_j({W^k})\partial_j W^{k+1}+\epsilon(\varphi^{k+1})^{2}\mathbb{L}(u^{k+1})=\epsilon\mathbb{H}(\varphi^{k+1}) \cdot \mathbb{Q}(u^{k}),\\[12pt]
  (\varphi^{k+1}, W^{k+1})|_{t=0}=(\varphi_0, W_0),\quad x\in \mathbb{R}^3,\\[12pt]
(\varphi^{k+1}, W^{k+1})\rightarrow (0,0) \quad \text{as } \quad |x|\rightarrow +\infty,\quad t\geq  0.
 \end{cases}
\end{equation}
It follows from Lemma \ref{lem1q} that the sequence $(\varphi^k, W^{k})$ satisfies the  uniform a priori estimates (\ref{jkk})
for $0 \leq t \leq T^{**}$.

Now we  prove the  convergence of the whole sequence  $(\varphi^k, W^k)$ of approximate solutions to a limit $(\varphi, W)$  in some strong sense.
Let
\begin{equation*}\begin{split}
\overline{\varphi}^{k+1}=\varphi^{k+1}-\varphi^{k},\quad
\overline{W}^{k+1}=(\overline{\phi}^{k+1},\overline{u}^{k+1})^\top,
\end{split}
\end{equation*}
with
$$
\overline{\phi}^{k+1}=\phi^{k+1}-\phi^k,\quad  \overline{u}^{k+1}=u^{k+1}-u^k.
$$
 Then from (\ref{li6}), one has
 \begin{equation}
\label{eq:1.2w}
\begin{cases}
\displaystyle
\ \  \overline{\varphi}^{k+1}_t+u^k\cdot \nabla\overline{\varphi}^{k+1} +\overline{u}^k\cdot\nabla\varphi ^{k}+\frac{\delta-1}{2}(\overline{\varphi}^{k} \text{div}u^{k-1} +\varphi ^{k}\text{div}\overline{u}^k)=0,\\[6pt]
\displaystyle
\ \ A_0\overline{W}^{k+1}_t+ \sum\limits_{j=1}^3A_j(W^k)\partial_{j}\overline{W}^{k+1}+\epsilon(\varphi^{k+1})^{2}\mathbb{L}(\overline{W}^{k+1})\\[6pt]
\displaystyle
=\sum\limits_{j=1}^3A_j(\overline{W}^k)\partial_{j}{W}^{k}-\epsilon \overline{\varphi}^{k+1}(\varphi^{k+1}+\varphi^k)\mathbb{L}(W^k)\\[6pt]
\displaystyle
\ \  +\epsilon\big(\mathbb{H}({\varphi}^{k+1})-\mathbb{H}({\varphi}^{k})\big)\cdot \mathbb{Q}(W^{k})+\mathbb{H} (\varphi^{k+1})\cdot \mathbb{Q}(\overline{W}^{k}).\\[12pt]
\end{cases}
\end{equation}
First, we consider  $|\overline{\varphi}^{k+1}|_2$. Multiplying $(\ref{eq:1.2w})_1$ by $2\overline{\varphi}^{k+1}$ and integrating  over $\mathbb{R}^3$, one has
\begin{equation*}
\begin{split}
\frac{d}{dt}|\overline{\varphi}^{k+1}|^2_2=& -2\int\Big(u^k\cdot \nabla\overline{\varphi}^{k+1} +\overline{u}^k\cdot\nabla\varphi ^{k}+\frac{\delta-1}{2}(\overline{\varphi}^{k} \text{div}u^{k-1} +\varphi ^{k}\text{div}\overline{u}^k)\Big)\overline{\varphi}^{k+1}\\
\leq& C|\nabla u^k|_\infty|\overline{\varphi}^{k+1}|^2_2+C |\overline{\varphi}^{k+1}|_2\big(|\overline{u}^k|_2|\nabla \varphi^k|_\infty
+|\overline{\varphi}^{k}|_2|\nabla u^{k-1}|_\infty+|\varphi^k \text{div}\overline{u}^k|_2 \big),
\end{split}
\end{equation*}
which means that ($0<\nu \leq \frac{1}{10}$ is a constant)
\begin{equation}\label{go64a}
\displaystyle
\frac{d}{dt}|\overline{\varphi}^{k+1}(t)|^2_2\leq A^k_\nu(t)|\overline{\varphi}^{k+1}(t)|^2_2+\nu\Big(\epsilon^{-\frac{1}{2}} |\overline{u}^k(t)|^2_2+|\overline{\varphi}^k(t)|^2_2+|\varphi^k \text{div}\overline{u}^k(t)|^2_2\Big)
\end{equation}
with $A^k_\nu(t)=C\big(1+\nu^{-1} \big)$.

Second, we consider $|\overline{W}^{k+1}|_2$. Multiplying $(\ref{eq:1.2w})_2$ by $\overline{W}^{k+1}$ and integrating  over $\mathbb{R}^3$, one gets
\begin{equation}\label{zheng8}
\begin{split}
&\frac{d}{dt}\int(\overline{W}^{k+1})^\top A_0\overline{W}^{k+1}+2a_1\epsilon\alpha|\varphi^{k+1}\nabla\overline{u}^{k+1} |^2_2+2(\alpha+\beta)a_1\epsilon|\varphi^{k+1}\text{div}\overline{u}^{k+1} |^2_2\\
\leq & \int(\overline{W}^{k+1})^\top\text{div}A({W}^{k})\overline{W}^{k+1}+\int  \sum_{j=1}^3(\overline{W}^{k+1})^\top A_j(\overline{W}^{k}) \partial_{j}W^{k}\\
&-2a_1\epsilon \frac{\delta-1}{\delta}\int\nabla( \varphi^{k+1})^2\cdot Q(\overline{u}^{k+1})\cdot  \overline{u}^{k+1}\\
&-2a_1\epsilon\int  \Big(\overline{\varphi}^{k+1}(\varphi^{k+1}+\varphi^k)\cdot  L(u^k)-\nabla (\overline{\varphi}^{k+1}(\varphi^{k+1}+\varphi^k))\cdot Q(u^k) \Big)\cdot  \overline{u}^{k+1} \\
&+2a_1\epsilon\int \nabla (\varphi^{k})^2\cdot (Q(u^{k})-Q(u^{k-1})) \cdot  \overline{u}^{k+1}:=\sum_{i=1}^{6} J_i.
\end{split}
\end{equation}

For the terms $J_1-J_4$,  it follows from  \eqref{zhen6a} and \eqref{ghbu} that 
\begin{equation}\label{ya1}
\begin{split}
J_1=&\int(\overline{W}^{k+1})^\top\text{div}A({W}^k)\overline{W}^{k+1}
\leq C|\nabla W^k|_\infty| \overline{W}^{k+1}|^2_2\leq C|\overline{W}^{k+1}|^2_2,\\
J_2=&\int  \sum_{j=1}^3A_j(\overline{W}^{k})\partial_{j}W^{k}\cdot \overline{W}^{k+1}\\
\leq & C|\nabla W^k|_\infty | \overline{W}^{k}|_2 | \overline{W}^{k+1}|_2
\leq C\nu^{-1}|\overline{W}^{k+1}|^2_2+\nu| \overline{W}^{k}|^2_2,\\
J_3=&-2 a_1\frac{\delta-1}{\delta}\epsilon   \int \nabla(\varphi^{k+1})^2\cdot Q( \overline{u}^{k+1})\cdot \overline{u}^{k+1}\\
\leq & C\epsilon|\nabla \varphi^{k+1}|_\infty |\varphi^{k+1}\nabla \overline{u}^{k+1}|_2|\overline{u}^{k+1}|_2\\
\leq&  C\epsilon^{\frac{1}{2}} |\overline{W}^{k+1}|^2_2+\frac{a_1\epsilon\alpha}{20} |\varphi^{k+1}\nabla \overline{u}^{k+1}|^2_2,\\
J_4=&-2a_1\epsilon\int  \overline{\varphi}^{k+1}(\varphi^{k+1}+\varphi^k)Lu^k \cdot  \overline{u}^{k+1} \\
\leq & C\epsilon| \overline{\varphi}^{k+1}|_2| \varphi^{k+1}\overline{u}^{k+1} |_6|Lu^k|_3+C\epsilon|\overline{\varphi}^{k+1}|_2|\varphi^kLu^k|_\infty| \overline{u}^{k+1}|_2\\
\leq&C\epsilon |Lu^k|^2_3 |\overline{\varphi}^{k+1}|^2_2 +\frac{a_1\epsilon\alpha}{20}  | \varphi^{k+1}\nabla \overline{u}^{k+1} |^2_2+C\epsilon| \overline{\varphi}^{k+1}|^2_2\\
&+C\epsilon|\nabla\varphi^{k+1} |^2_\infty|Lu^k|^2_3|\overline{u}^{k+1} |^2_2+C\epsilon(\epsilon^{-1}+|\varphi^k\nabla^4 u^k|^2_2)|\overline{u}^{k+1} |^2_2\\
\leq & C\epsilon|\overline{\varphi}^{k+1}|^2_2+\frac{a_1\epsilon\alpha}{20}  | \varphi^{k+1}\nabla \overline{u}^{k+1} |^2_2
+C(1+\epsilon |\varphi^k\nabla^4 u^k|^2_2)|\overline{W}^{k+1} |^2_2,
\end{split}
\end{equation}
where we have used the fact (see Lemma \ref{lem2as}) that
\begin{equation}\label{qianru}
\begin{split}
&|\varphi^{k}\nabla^2 u^k|_\infty\leq  C|\varphi^{k}\nabla^2 u^k|^{\frac{1}{2}}_6|\nabla (\varphi^{k}\nabla^2 u^k)|^{\frac{1}{2}}_6\\
\leq& C|\varphi^{k}\nabla^2 u^k|^{\frac{1}{2}}_{D^1}|\nabla (\varphi^{k}\nabla^2 u^k)|^{\frac{1}{2}}_{D^1}\leq C\|\nabla (\varphi^{k}\nabla^2 u^k)\|_1\\
\leq &C(|\nabla \varphi^k|_\infty\|\nabla^2 u^k\|_1+|\varphi^k|_\infty|\nabla^3 u|_2+|\varphi^k\nabla^4 u|_2+|\nabla^2 \varphi^k|_6|\nabla^2u^k|_3)\\
\leq &C(\|\nabla\varphi^k\|_2\|\nabla^2u^k\|_1+|\varphi^k\nabla^4 u^k|_2)\leq C(\epsilon^{-\frac{1}{2}}+|\varphi^k\nabla^4 u^k|_2).
\end{split}
\end{equation}

Next, we begin to consider the term $J_5$. First we have
\begin{equation}\label{ya7}
\begin{split}
J_5=&2a_1\epsilon\int \nabla (\overline{\varphi}^{k+1}(\varphi^{k+1}+\varphi^k))\cdot Q(u^k) \cdot  \overline{u}^{k+1}\\
=&-2a_1\epsilon \int \sum_{ij}\overline{\varphi}^{k+1} (\varphi^{k+1}+\varphi^k)\partial_i( a^{ij}_k\overline{u}^{k+1,j})\\
=& J_{51}+J_{52}+J_{53}+J_{54},
\end{split}
\end{equation}
where $u^{k,j}$ represents the $j$-$th$ component of $u^k$ ($k\geq 1$),
$$\overline{u}^{k,j}=u^{k,j}-u^{k-1,j},\quad \text{for}\ k\geq 1, \quad j=1,2,3,$$ and the quantity $a^{ij}_k$ is given by
$$
a^{ij}_k=\frac{\delta}{\delta-1}\Big(\alpha (\partial_i u^{k,j}+\partial_j u^{k,i}\big)+\text{div}u^k \delta_{ij}\Big)\quad \text{for}\ i,\ j=1,2,3,
$$
and  $\delta_{ij}$ is the Kronecker symbol satisfying $\delta_{ij}=1,\ i=j$ and $\delta_{ij}=0$, otherwise.

For  terms $J_{51}$-$J_{53}$, using \eqref{zhen6a},  \eqref{ghbu} and (\ref{qianru}), one can obtain 
\begin{equation}\label{ya91}
\begin{split}
J_{51}=&-2a_1\epsilon \int \sum_{ij}\overline{\varphi}^{k+1} \varphi^{k+1}\partial_i a^{ij}_k\overline{u}^{k+1,j}\\
\leq& C\epsilon  |\nabla^2 u^k|_6 |\overline{\varphi}^{k+1}|_2|\varphi^{k+1}\overline{u}^{k+1}|_3\\
\leq & C\epsilon   |\overline{\varphi}^{k+1}|_2| \varphi^{k+1}\overline{u}^{k+1} |^{\frac{1}{2}}_2|\varphi^{k+1} \overline{u}^{k+1} |^{\frac{1}{2}}_6\\
\leq& C\epsilon|\overline{ \varphi}^{k+1}|^2_2+C\epsilon|{\varphi}^{k+1}\overline{u}^{k+1}|_2|{\varphi}^{k+1}\overline{u}^{k+1}|_6\\
\leq& C\epsilon|\overline{\varphi}^{k+1}|^2_2+C\epsilon^{\frac{1}{2}}|\overline{W}^{k+1}|_2^2+\frac{a_1\epsilon\alpha}{20}|\varphi^{k+1}\nabla\overline{u}^{k+1}|^2_2,\\
%
J_{52}=&-2a_1\epsilon \int \sum_{i,j}\overline{\varphi}^{k+1} \varphi^k \partial_i a^{ij}_k\overline{u}^{k+1,j}\\
\leq& C\epsilon |\varphi^{k}\nabla^2 u^k|_\infty |\overline{\varphi}^{k+1}|_2|\overline{u}^{k+1}|_2\\
\leq & C\epsilon   |\overline{\varphi}^{k+1}|^2_2+C\epsilon(\epsilon^{-1}+|\varphi^k\nabla^4 u^k|^2_2)|\overline{u}^{k+1}|^2_2\\
\leq & C\epsilon|\overline{\varphi}^{k+1}|^2_2+C(1+\epsilon |\varphi^k\nabla^4 u^k|^2_2)|\overline{W}^{k+1}|^2_2,\\
J_{53}=&-2a_1\epsilon\int  \sum_{i,j} \overline{\varphi}^{k+1}\varphi^{k+1} a^{ij}_k\partial_i \overline{u}^{k+1,j}\\
\leq &C\epsilon| \overline{\varphi}^{k+1}|_2 |\varphi^{k+1}\nabla\overline{u}^{k+1}|_2|\nabla u^k|_\infty \\
\leq & C\epsilon|\nabla u^k|^2_\infty| \overline{\varphi}^{k+1}|^2_2+\frac{a_1\epsilon\alpha}{20}  |\varphi^{k+1}\nabla\overline{u}^{k+1}|^2_2\\
\leq& C\epsilon| \overline{\varphi}^{k+1}|^2_2+\frac{a_1\epsilon\alpha}{20}  |\varphi^{k+1}\nabla\overline{u}^{k+1}|^2_2.\\
\end{split}
\end{equation}
For the last term on the right-hand side of (\ref{ya7}), one has
\begin{equation}\label{ya94}
\begin{split}
J_{54}=&-2a_1\epsilon\int  \sum_{i,j} \overline{\varphi}^{k+1}\varphi^k a^{ij}_k\partial_i \overline{u}^{k+1,j}\\
=&-2a_1\epsilon\int  \sum_{i,j} \overline{\varphi}^{k+1}(\varphi^k-\varphi^{k+1}+\varphi^{k+1}) a^{ij}_k\partial_i \overline{u}^{k+1,j}\\
\leq &C\epsilon|\nabla u^k|_\infty |\nabla\overline{u}^{k+1}|_\infty|\overline{\varphi}^{k+1}|^2_2+C\epsilon|\nabla u^k|_\infty|\varphi^{k+1} \nabla\overline{u}^{k+1}|_2|\overline{\varphi}^{k+1}|_2\\
\leq &C\epsilon|\nabla u^k|^2_\infty|\overline{\varphi}^{k+1}|^2_2+C\epsilon|\nabla u^k|_\infty |\nabla\overline{u}^{k+1}|_\infty|\overline{\varphi}^{k+1}|^2_2+\frac{a_1\epsilon\alpha}{20}  |\varphi^{k+1} \nabla\overline{u}^{k+1}|^2_2\\
\leq& C\epsilon |\overline{\varphi}^{k+1}|^2_2+\frac{a_1\epsilon\alpha}{20}  |\varphi^{k+1} \nabla\overline{u}^{k+1}|^2_2,\\
\end{split}
\end{equation}
which, along with (\ref{ya7})-(\ref{ya94}),  implies that
\begin{equation}\label{ya95}
\begin{split}
J_5\leq & \frac{a_1\epsilon\alpha}{20} |\varphi^{k+1} \nabla\overline{u}^{k+1}|^2_2+C(+\epsilon |\varphi^k\nabla^4 u^k|^2_2)|\overline{W}^{k+1}|^2_2+ C\epsilon|\overline{\varphi}^{k+1}|^2_2.\\
\end{split}
\end{equation}

For the term $J_6$, we have
\begin{equation}\label{ya8}
\begin{split}
J_6=&2a_1\epsilon\int \nabla (\varphi^k)^2\cdot (Q(u^{k})-Q(u^{k-1})) \cdot  \overline{u}^{k+1} \\
\leq&C\epsilon |\nabla \varphi^k|_\infty|\varphi^k \nabla \overline{u}^k|_2| \overline{u}^{k+1}|_2
\leq C\nu^{-1}| \overline{W}^{k+1}|^2_2+\nu\epsilon |\varphi^k \nabla \overline{u}^k|^2_2,\\
\end{split}
\end{equation}
which, together with (\ref{zheng8})-\eqref{ya1} and  (\ref{ya95})-(\ref{ya8}),
immediately implies  that
\begin{equation}\label{gogo1}\begin{split}
&\frac{d}{dt}\int (\overline{W}^{k+1})^\top A_0\overline{W}^{k+1}+a_1\epsilon\alpha |\varphi^{k+1}\nabla\overline{u}^{k+1} |^2_2\\
\leq&C(\nu^{-1}+\epsilon |\varphi^k\nabla^4 u^k|^2_2)|\overline{W}^{k+1}|^2_2
+C\epsilon|\overline{\varphi}^{k+1}|^2_{2}+\nu (\epsilon|\varphi^k \nabla \overline{u}^k|^2_2+\epsilon|\varphi^k|_2^2+|\overline{W}^k|_2^2).
\end{split}
\end{equation}

Finally, we denote
\begin{equation*}\begin{split}
\Gamma^{k+1}(t)=&\sup_{s\in [0,t]}|\overline{W}^{k+1}(s)|^2_{2}+\sup_{s\in [0,t]}\epsilon|\overline{\varphi}^{k+1}(s)|^2_{2}.
\end{split}
\end{equation*}
From (\ref{go64a}) and (\ref{gogo1}), one has
\begin{equation*}\begin{split}
&\frac{d}{dt}\int\Big((\overline{W}^{k+1})^\top A_0\overline{W}^{k+1}+\epsilon|\overline{\varphi}^{k+1}(t)|^2_{2}\Big)+a_1\epsilon\alpha |\varphi^{k+1}\nabla\overline{u}^{k+1} |^2_2\\
\leq& E^k_\nu (|\overline{W}^{k+1}|^2_{2}+\epsilon|\overline{\varphi}^{k+1}|^2_{2})+ \nu (\epsilon|\varphi^k \nabla \overline{u}^k|^2_2
+\epsilon|\overline{\varphi}^k|_2^2+| \overline{W}^{k}|^2_2),
\end{split}
\end{equation*}
for some $E^k_\nu$ such that  $\int_{0}^{t}E^k_\nu(s)\text{d}s\leq C+C_\nu t$.
From Gronwall's inequality, one ges
\begin{equation*}\begin{split}
&\Gamma^{k+1}+\int_{0}^{t}a_1\epsilon\alpha|\varphi^{k+1}\nabla\overline{u}^{k+1} |^2_2\text{d}s\\
\leq&   C\nu\int_{0}^{t}   \Big(\epsilon|\varphi^k \nabla \overline{u}^k|^2_2+\epsilon | \overline{\varphi}^{k}|^2_2+| \overline{W}^{k}|^2_2\Big)\text{d}s\exp{(C+C_\nu t)}\\
\leq & \Big( C\nu\int_{0}^{t}   \epsilon|\varphi^k \nabla \overline{u}^k|^2_2\text{d}s+Ct\nu \sup_{s\in [0,t]} [| \overline{W}^{k}|^2_2
+\epsilon| \overline{\varphi}^{k}|^2_2]\Big)\exp{(C+C_\nu t)}.
\end{split}
\end{equation*}
We can choose $\nu_0>0$ and $T_*\in (0,\min (1,T^{**}))$ small enough such that
\begin{equation*}\begin{split}
C\nu_0 \exp{C}&=\frac{1}{8} \min \big\{1, a_1\alpha\big \}, \quad \exp (C_\nu T_*) \leq 2,
\end{split}\end{equation*}
which implies that
\begin{equation*}\begin{split}
\sum_{k=1}^{\infty}\Big( \Gamma^{k+1}(T_*)+\int_{0}^{T_*} \alpha\epsilon|\varphi^{k+1}\nabla\overline{u}^{k+1} |^2_2\text{d}t\Big)\leq C<+\infty.
\end{split}
\end{equation*}
Thus, from the above estimate for $\Gamma^{k+1}(T_*)$ and (\ref{jkk}),  we know  that the whole sequence $(\varphi^k,W^k)$ converges to a limit $(\varphi, W)$ in the following strong sense:
\begin{equation}\label{str}
\begin{split}
&(\varphi^k,W^k)\rightarrow (\varphi,W)\ \text{in}\ L^\infty([0,T_*];H^2(\mathbb{R}^3)).
\end{split}
\end{equation}
Due to the local estimates (\ref{jkk}) and the lower-continuity of norm for weak or weak$^*$ convergence, we know that $(\varphi,  W)$ satisfies the estimates (\ref{jkk}).
According to the strong convergence in (\ref{str}), it is easy to show that $(\varphi,  W)$ is a weak solution of
\eqref{li41} in  the sense of distribution  with the regularities:
\begin{equation}\label{rjkqq}\begin{split}
& \varphi \in L^\infty([0,T_*];H^3),\quad  \varphi_t \in L^\infty([0,T_*];H^2),\quad \phi \in L^\infty([0,T_*];H^3),\\
&  \phi_t \in L^\infty([0,T_*];H^2),\quad   u\in L^\infty([0,T_*]; H^3),\\
& \varphi \nabla^4 u\in L^2([0,T_*];L^2), \quad  u_t \in L^\infty([0,T_*]; H^1)\cap L^2([0,T_*] ; D^2).
\end{split}
\end{equation}
So the existence of  strong solutions is proved.

\underline{Step 2}: uniqueness.   Let $(W_1,\varphi_1)$ and $(W_2,\varphi_2)$ be two strong solutions  to Cauchy problem (\ref{li41})  satisfying the uniform a priori estimates (\ref{jkk}). We denote that
\begin{equation*}\begin{split}
&\overline{\varphi}=\varphi_1-\varphi_2,\quad \overline{W}=(\overline{\phi},\overline{u})=(\phi_1-\phi_2,u_1-u_2),\\
\end{split}
\end{equation*}
then according to (\ref{eq:1.2w}), $(\overline{\varphi},\overline{\phi},\overline{u})$ satisfies the  system
 \begin{equation}
\label{zhuzhu}
\begin{cases}
\displaystyle
\ \ \overline{\varphi}_t+u_1\cdot \nabla\overline{\varphi} +\overline{u}\cdot\nabla\varphi_{2}+\frac{\delta-1}{2}(\overline{\varphi} \text{div}u_2 +\varphi_{1}\text{div}\overline{u})=0,\\[8pt]
\ \ A_0\overline{W}_t+\sum\limits_{j=1}^3A_j(W_1)\partial_{j} \overline{W}+ \epsilon\varphi^2_1\mathbb{L}(\overline{u})\\
=-\sum\limits_{j=1}^3A_j(\overline{W})\partial_{j} W_{2} -\overline{\varphi}(\varphi_1+\varphi_2)\mathbb{L}(u_2)\\[8pt]
\ \ +\epsilon\big(\mathbb{H}(\varphi_1)-\mathbb{H}(\varphi_2)\big)\cdot \mathbb{Q}(W_2)+\mathbb{H} (\varphi_1)\cdot \mathbb{Q}(\overline{W}),\\[8pt]
\end{cases}
\end{equation}

Using the same arguments as in the derivation of (\ref{go64a})-(\ref{gogo1}), and letting
$$
\Lambda(t)=|\overline{W}(t)|^2_{2}+\epsilon|\overline{\varphi}(t)|^2_{2},
$$
we similarly have
\begin{equation}\label{gonm}
\begin{cases}
\displaystyle
\frac{d}{dt}\Lambda(t)+C\epsilon|\varphi_1\nabla \overline{u}(t)|^2_2\leq F(t)\Lambda(t),\\[10pt]
\displaystyle
\int_{0}^{t}F(s)ds\leq C\quad  \text{for} \quad 0\leq t\leq T_*.
\end{cases}
\end{equation}
From Gronwall's inequality, we have $\overline{\varphi}=\overline{\phi}=\overline{u}=0$.
Then the uniqueness is obtained.

\underline{Step 3}: the time-continuity. It can be obtained via the  same arguments used in  in the proof of Lemma \ref{lem1q}.

\end{proof}

\section{Proof of  Theorem \ref{th2}}
Based on Theorem \ref{ths1}, now we are  ready to prove the uniform local-in-time well-posedenss (with respect to $\epsilon$) of the regular solution to the original Cauchy problem (\ref{eq:1.1})-(\ref{eq:2.211}), i.e., the proof of Theorem \ref{th2}. Moreover, we will show that the regular solutions that we obtained satisfy system (\ref{eq:1.1}) classically in positive time $(0,T_*]$.\\

\begin{proof} We divide the proof into three  steps.

\underline{Step 1}: existence of regular solutions.
First, for the initial assumption (\ref{th78}), it follows  from Theorem \ref{ths1} that
there exists  a positive  time $T_*$ independent of $\epsilon$ such that the problem (\ref{li41}) has a unique strong solution $(\varphi,\phi,u)$ in $[0,T_*]\times \mathbb{R}^3$ satisfying the regularities in (\ref{reg11qq}), which means that
\begin{equation}\label{reg2}
(\rho^{\frac{\delta-1}{2}},\rho^{\frac{\gamma-1}{2}} )=(\varphi,\phi)\in C^1((0,T_*)\times \mathbb{R}^3),\quad \text{and} \quad (u,\nabla u)\in C((0,T_*)\times \mathbb{R}^3).
\end{equation}
Noticing that  $\rho=\varphi^{\frac{2}{\delta-1}}$ and
$\frac{2}{\delta-1}\geq 1$, it is
easy to show that
$$\rho \in C^1((0,T_*)\times\mathbb{R}^3).$$

Second, the system $(\ref{li41})_2$ for $W=(\phi,u)$ could be written as 
\begin{equation}
\begin{cases}
\label{fanzheng}
\displaystyle
\phi_t+u\cdot \nabla \phi+\frac{\gamma-1}{2}\phi \text{div} u=0,\\[10pt]
\displaystyle
u_t+u\cdot\nabla u +\frac{A\gamma}{\gamma-1}\nabla \phi^2+\epsilon\varphi^2 Lu=\epsilon \nabla \varphi^2 \cdot Q(u).
 \end{cases}
\end{equation}

Multiplying $(\ref{fanzheng})_1$ by
$
\frac{\partial \rho}{\partial \phi}(t,x)=\frac{2}{\gamma-1}\phi^{\frac{3-\gamma}{\gamma-1}}(t,x)\in C((0,T_*)\times \mathbb{R}^3)
$ on both sides,
we get the continuity equation in $(\ref{eq:1.1})_1$:
\begin{equation} \label{eq:2.58}
\rho_t+u \cdot\nabla \rho+\rho\text{div} u=0.
\end{equation}

Multiplying $(\ref{fanzheng})_2$ by
$
\phi^{\frac{2}{\gamma-1}}=\rho(t,x)\in C^1((0,T_*)\times \mathbb{R}^3)
$ on both sides,
we get the momentum equations in $(\ref{eq:1.1})_2$:
\begin{equation} \label{eq:2.60}
\begin{split}
&\rho u_t+\rho u\cdot \nabla u+\nabla P=\text{div}\Big(\mu(\rho)(\nabla u+ (\nabla u)^\top)+\lambda(\rho)\text{div}u I_3\Big).
\end{split}
\end{equation}

Finally, recalling that  $\rho$ can be represented by the formula
$$
\rho(t,x)=\rho_0(U(0,t,x))\exp\Big(\int_{0}^{t}\textrm{div} u(s,U(s,t,x))\text{d}s\Big),
$$
where  $U\in C^1([0,T_*]\times[0,T_*]\times \mathbb{R}^3)$ is the solution to the initial value problem
\begin{equation}
\label{eq:bb1}
\begin{cases}
\frac{d}{ds}U(t,s,x)=u(s,U(s,t,x)),\quad 0\leq s\leq T_*,\\[4pt]
U(t,t,x)=x, \quad\quad\quad 0\leq t\leq T_*,\quad  x\in \mathbb{R}^3,
\end{cases}
\end{equation}
 it is obvious that
$$
\rho(t,x)\geq 0, \ \forall (t,x)\in (0,T_*)\times \mathbb{R}^3.
$$
That is to say, $(\rho,u)$ satisfies the problem (\ref{eq:1.1})-(\ref{eq:2.211}) in the sense of distributions, and has  the regularities shown in Definition \ref{d1}, which means that  the Cauchy problem (\ref{eq:1.1})-(\ref{eq:2.211}) has a unique regular solution $(\rho,u)$.

\end{proof}

\underline{Step 2}: the smoothness of regular solutions. Now we will show that the regular solution that we obtained in the above step is indeed a classcial one in positive time $(0, T_*]$.

Due to  the definition of regular solution and  the classical Sobolev imbedding theorem, we immediately know that
$$
(\rho,\nabla \rho, \rho_t, u, \nabla u) \in C([0,T_*]\times \mathbb{R}^3).
$$
So now we  only need to prove that
$$
(u_t, \text{div}\mathbb{T}) \in C((0,T_*]\times \mathbb{R}^3).
$$
\begin{proof}
According to  Theorem \ref{th2} in Section $3$,
the solution $(\varphi, \phi,u)$ of problem (\ref{li41}) satisfies the regularities (\ref{reg11qq}) and $(\phi,\varphi) \in C^1([0,T_*]\times \mathbb{R}^3)$.

Next, we first give the continuity of $u_t$.
We differentiate $\eqref{fanzheng}_2$ with respect to $t$:
\begin{equation}\label{zhd1}
\begin{split}
u_{tt}+\epsilon\varphi^2 Lu_t=-(\varphi^2)_t Lu-(u\cdot\nabla u)_t -\frac{A\gamma}{\gamma-1}\nabla (\phi^2)_t+\epsilon(\nabla \varphi^2 \cdot Q(u))_t,
\end{split}
\end{equation}
which, along with (\ref{reg11qq}), easily implies that
\begin{equation}\label{zhd11}
u_{tt}\in L^2([0,T_*];L^2).
\end{equation}
Applying the operator $\partial^\zeta_x$  $(|\zeta|=2)$ to $(\ref{zhd1})$,  multiplying the resulting equations by $\partial^\zeta_x u_t$ and integrating over $\mathbb{R}^3$, we have
\begin{equation}\label{zhd2}
\begin{split}
&\frac{1}{2} \frac{d}{dt}|\partial^\zeta_xu_t|^2_2+\alpha\epsilon|\varphi \nabla \partial^\zeta_x u_t|^2_2+(\alpha+\beta)\epsilon|\varphi \text{div} \partial^\zeta_x u_t|^2_2\\
=&\int  \Big( -\epsilon\nabla \varphi^2 \cdot Q(\partial^\zeta_x u_t)-\epsilon\Big(\partial^\zeta_x(\varphi^2Lu_t)-\varphi^2 L\partial^\zeta_x u_t\Big)\Big)\cdot   \partial^\zeta_x u_t\\
&+\int  \Big(-\epsilon\partial^\zeta_x\big((\varphi^2)_t Lu\big)-\partial^\zeta_x(u\cdot\nabla u)_t -\frac{A\gamma}{\gamma-1}\partial^\zeta_x\nabla (\phi^2)_t \Big)\cdot   \partial^\zeta_x u_t \\
&+\int  \partial^\zeta_x(\epsilon\nabla \varphi^2 \cdot Q(u))_t\cdot   \partial^\zeta_x u_t  \equiv: \sum_{i=7}^{12}J_i.
\end{split}
\end{equation}

Now we consider the terms on the right-hand side of (\ref{zhd2}).
It follows from the  H\"older's inequality, Lemma \ref{lem2as} and Young's inequality that 
\begin{equation}\label{zhd3}
\begin{split}
J_{7}=&\epsilon\int  \Big( -\nabla \varphi^2 \cdot  Q(\partial^\zeta_x u_t)\Big)\cdot   \partial^\zeta_x u_t\\
\leq & C\epsilon|\varphi\nabla^3 u_t|_2|\nabla^2 u_t|_2|\nabla \varphi|_\infty
\leq \frac{\alpha\epsilon}{20}|\varphi\nabla^3 u_t|^2_2+C\epsilon^{\frac{1}{2}}| u_t|^2_{D^2},\\
J_{8}=&\int -\epsilon\Big(\partial^\zeta_x(\varphi^2Lu_t)-\varphi^2 L\partial^\zeta_x u_t\Big)\cdot   \partial^\zeta_x u_t\\
\leq & C\epsilon\Big(|\varphi\nabla^3 u_t|_2|\nabla \varphi|_\infty+|\nabla \varphi|^2_\infty|u_t|_{D^2}+|\nabla^2 \varphi|_3|\varphi \nabla^2 u_t|_6\Big)|u_t|_{D^2}\\
\leq & \frac{\alpha\epsilon}{20}|\varphi\nabla^3 u_t|^2_2+C\epsilon^{\frac{1}{2}}| u_t|^2_{D^2},\\
\end{split}
\end{equation}
and
\begin{equation}\label{zhd3jkxx}
\begin{split}
J_{9}=&\int  -\epsilon\partial^\zeta_x\big((\varphi^2)_t Lu\big)\cdot   \partial^\zeta_x u_t \\
\leq &C\epsilon\Big(|\nabla^2 \varphi|_3|Lu|_6| \varphi_t|_\infty|u_t|_{D^2}+|\varphi \nabla^2 u_t|_6|\varphi_t|_{D^2}|Lu|_3\\
&+|\nabla \varphi|_\infty|\nabla \varphi_t|_6|Lu|_{3}|u_t|_{D^2}+|\varphi \nabla^3 u|_6|\nabla \varphi_t|_{3}|u_t|_{D^2}\\
&+|\nabla \varphi|_\infty| \varphi_t|_\infty|\nabla^3 u|_{2}|u_t|_{D^2}+|\varphi_t|_\infty|u_t|_{D^2} |\varphi \nabla^4 u|_2\Big)\\
\leq & \frac{\alpha\epsilon}{20}|\varphi\nabla^3 u_t|^2_2+C\epsilon^{\frac{1}{2}}| u_t|^2_{D^2}+C\epsilon|\varphi \nabla^4 u|^2_2+C\epsilon^{\frac{1}{2}},\\
J_{10}=&\int  -\partial^\zeta_x(u\cdot\nabla u)_t \cdot   \partial^\zeta_x u_t \\
\leq & C(\|u_t\|_1+| u_t|_{D^2})\|u\|_3-\int \big(u\cdot \nabla\big) \partial^\zeta_x u_t \cdot   \partial^\zeta_x u_t \\
\leq &  C+C| u_t|^2_{D^2}+C|\nabla u|_\infty|\partial^\zeta_x u_t|^2_2 
\leq   C+C| u_t|^2_{D^2}, \\
J_{11}=&\int   -\frac{A\gamma}{\gamma-1}\partial^\zeta_x\nabla (\phi^{2})_t \cdot   \partial^\zeta_x u_t\\
\leq& C\Big(|\nabla^2 \phi_t|_2|\phi \nabla^3 u_t|_2+|\phi_t|_\infty|\nabla^3\phi|_2|\nabla^2u_t|_2\\
&+C|\nabla^2\phi|_6 |\nabla\phi_t|_3|\nabla^2 u_t|_2+|\nabla^2\phi_t|_2|\nabla\phi|_\infty|\nabla u_t|_2\Big)\\
\leq & \frac{\alpha}{20}|\phi\nabla^3 u_t|^2_2+C(1+|u_t|_{D^2}),\\
J_{12}=&\int \epsilon \partial^\zeta_x(\nabla \varphi^2 \cdot Q(u))_t\cdot   \partial^\zeta_x u_t\\
\leq &C\epsilon\Big(\|\nabla \varphi\|^2_2| u_t|^2_{D^2}+\big(\|\nabla \varphi\|_2|\nabla u_t|_3+\|u\|_3\|\varphi_t\|_2\big)|\varphi \nabla^2 u_t|_6\\
&+\big(\|\nabla \varphi\|_2|\varphi \nabla^3 u_t|_2+\|\nabla \varphi\|_2|\varphi \nabla^2 u_t|_6\big)| u_t|_{D^2}\\
&+\big(\|\nabla \varphi\|_2\|\varphi_t\|_2\|u\|_3+\|\varphi_t\|_2|\varphi \nabla^3 u|_6\big)| u_t|_{D^2}\Big)\\
&+\epsilon\int  \partial^\zeta_x(\nabla \varphi^2)_t \cdot Q(u)\cdot   \partial^\zeta_x u_t(=J_{121})\\
\leq & \frac{\alpha\epsilon}{20}|\varphi\nabla^3 u_t|^2_2+C\epsilon| u_t|^2_{D^2}+C\epsilon|\varphi\nabla^4 u|^2_2+J_{121},
\end{split}
\end{equation}
where the term $J_{121}$ can be estimated as follows
\begin{equation}\label{zhd15}
\begin{split}
J_{121}
\leq&\epsilon  \|\nabla \varphi\|_2\|\varphi_t\|_2\|u\|_3|u_t|_{D^2}+\epsilon\int  \varphi \partial^\zeta_x \nabla \varphi_t \cdot Q(u)\cdot   \partial^\zeta_x u_t(=J_{1211}).
\end{split}
\end{equation}
Via intergration by parts, for the last term $J_{1211}$ on the right-hand side of (\ref{zhd15}), one has 
\begin{equation}\label{zhd16}
\begin{split}
J_{1211}\leq&  C\epsilon\Big(|\nabla \varphi|_\infty|\varphi_t|_{D^2}|\nabla u|_\infty|u_t|_{D^2}+|\varphi_t|_{D^2}|\nabla^2u|_{3}|\varphi \nabla^2 u_t|_6\\
&+|\varphi\nabla^3 u_t|_2|\nabla u|_\infty\|\varphi_t\|_2\Big)
\leq  \frac{\alpha\epsilon}{20}|\varphi\nabla^3 u_t|^2_2+C\epsilon^{\frac{1}{2}}| u_t|^2_{D^2}+C\epsilon.
\end{split}
\end{equation}
It follows  from (\ref{zhd2})-(\ref{zhd16}) that 
\begin{equation}\label{zhd5}
\begin{split}
&\frac{1}{2} \frac{d}{dt}|u_t|^2_{D^2}+\frac{\alpha}{2} |\varphi \nabla^3 u_t|^2_2
 \leq C\epsilon^{\frac{1}{2}}| u_t|^2_{D^2}+C\epsilon|\varphi \nabla^4 u|^2_2+C\epsilon.
\end{split}
\end{equation}
Then multiplying both sides of (\ref{zhd5}) with $t$ and integrating  over $[\tau,t]$ for any $\tau \in (0,t)$, one gets
\begin{equation}\label{zhd6}
\begin{split}
&t|u_t|^2_{D^2}+\int_{\tau}^t s|\varphi \nabla^3 u_t|^2_2 \text{d}s
 \leq C\tau| u_t(\tau)|^2_{D^2}+C(1+t).
\end{split}
\end{equation}

According to the definition of the regular solution, we know that
$$
\nabla^2 u_t \in L^2([0,T_*]; L^2),
$$
which, along with Lemma \ref{1}, implies that
there exists a sequence $s_k$ such that
$$
s_k\rightarrow 0, \quad \text{and}\quad s_k |\nabla^2 u_t(s_k,\cdot)|^2_2\rightarrow 0, \quad \text{as} \quad k\rightarrow+\infty.
$$
Then letting $\tau=s_k \rightarrow 0$ in (\ref{zhd6}), we have
\begin{equation}\label{zhd7}
\begin{split}
&t|u_t|^2_{D^2}+\int_{0}^t s|\varphi \nabla^3 u_t|^2_2 \text{d}s
 \leq C(1+t)\leq C.
\end{split}
\end{equation}
So we have
\begin{equation}\label{zhd12}
t^{\frac{1}{2}}u_t \in L^\infty([0,T_*]; H^2).
\end{equation}

Based on the classical  Sobolev imbedding theorem:
 \begin{equation}\label{yhn}
\begin{split}
L^\infty([0,T];H^1)\cap W^{1,2}([0,T];H^{-1})\hookrightarrow C([0,T];L^q),\\
\end{split}
\end{equation}
for any $q\in (3,6)$,  from (\ref{zhd11}) and (\ref{zhd12}), we have
$$
tu_t \in C([0,T_*];W^{1,4}),
$$
which implies that
$
u_t \in C((0,T_*]\times \mathbb{R}^3)
$.

Finally, we consider the continuity of $\text{div}\mathbb{T}$. Denote $\mathbb{N}=\epsilon\varphi^2 Lu-\epsilon\nabla\varphi^2\cdot
Q(u).$
Based on  \eqref{reg11qq} and  \eqref{zhd12}, we have
$$
t\mathbb{N}\in L^\infty(0,T_*;H^2).
$$
From $\mathbb{N}_t\in L^2(0,T_*; L^2)$ and \eqref{yhn}, we obtain $t\mathbb{N}\in C([0,T_*];W^{1,4}),$
which implies that $\mathbb{N}\in C((0,T_*]\times\R^3)$. Since $\rho\in C([0,T_*]\times\R^3)$ and
$\text{div}\mathbb{T}=\rho \mathbb{N},$  then we obtain the desired conclusion.
\end{proof}

\underline{Step 3}: the proof of (\ref{regco11}).  If $1<\delta\leq \frac{5}{3},$ that is $\frac{2}{\delta-1}\geq 3.$ Due to
$$
\phi\in C([0,T_*]; H^3)\cap C^1([0,T_*]; H^2),\quad \text{and}\quad \rho(t,x)=\phi^{\frac{2}{\delta-1}}(t,x),
$$
then we have
$$
\rho(t,x)\in C([0,T_*];H^3).
$$

Noticing
\begin{equation}
\label{eq:bb1a}
\begin{split}
& u\in C([0,T_*];H^{s'})\cap L^\infty(0,T_*;H^{s'})\quad \text{for} \quad s'\in[2,3),\\
&\rho^{\frac{\delta-1}{2}}\nabla^4 u\in C(0,T_*;L^2),\quad u_t\in C([0,T_*];H^1)\cap L^2(0,T_*;D^2),\\
\end{split}
\end{equation}
it is not to show that 
\begin{equation}
\label{eq:bb1ab}
\rho\text{div} u\in L^2(0,T_*;H^3)\cap C([0,T^*];H^2).
\end{equation}
From the continuity equation $\eqref{eq:1.1}_1$, and \eqref{eq:bb1a}-\eqref{eq:bb1ab}, we have
$$
\rho\in C([0,T_*];H^3)\cap C^1([0,T_*];H^2).
$$

Similarly, we can deal with   cases: $\delta=2$ or $3$.
Then the proof of Theorem \ref{th2} is finished.

 \section{Vanishing viscosity limit}
In this section, we will establish  the vanishing  viscosity limit stated in Theorem 1.3.
First we denote by 
$$(\varphi^\epsilon, W^\epsilon)=(\varphi^\epsilon,\phi^\epsilon,u^\epsilon)^\top=((\rho^{\epsilon})^{\frac{\delta-1}{2}},(\rho^{\epsilon})^{\frac{\gamma-1}{2}},u^\epsilon)^\top$$
 the solution of problem $\eqref{li41}$, that is:
\begin{equation}
\begin{cases}
\label{E:1.1}
\displaystyle
\varphi^\epsilon_t+u\cdot\nabla\varphi^\epsilon+\frac{\delta-1}{2}\varphi^\epsilon\text{div} u^\epsilon=0,\\[8pt]
\displaystyle
A_0W^\epsilon_t+\sum_{j=1}^3A_j(W^\epsilon) \partial_j W^\epsilon=-\epsilon\Big((\varphi^\epsilon)^2\mathbb{L}(W^\epsilon)- \mathbb{H}(\varphi^\epsilon)\cdot\mathbb{Q}(W^\epsilon)\Big),\\[8pt]
\displaystyle
(\varphi^\epsilon, W^\epsilon )|_{t=0}=(\varphi_0,W_0),\\[8pt]
\displaystyle
(\varphi,W)\rightarrow (0,0),\quad \text{as}\quad  |x|\rightarrow +\infty, \quad t>0.
 \end{cases}
\end{equation}
The definitions of $A_j$ ($j=0,1,...,3$), $\mathbb{L}$,   $\mathbb{H}$ and $\mathbb{Q}$ could be find in   (\ref{sseq:5.2qq})-(\ref{sseq:5.3qq}).

Second, based on \cite{tms1}, we denote by 
$$W=(\phi,u)=(\rho^{\frac{\gamma-1}{2}},u)$$ the regular solution of compressible  Euler equations \eqref{eq:1.1E}, which can be written as the following  symmetric system,
\begin{equation}
\label{E:4.2}\begin{cases}
\displaystyle
A_0W_t+\sum_{j=1}^3A_j(W) \partial_j W=0,\\[10pt]
\displaystyle
W(x,0)=W_0=(\phi_0,u_0),\\[10pt]
\displaystyle
W\rightarrow 0,\quad \text{as}\quad  |x|\rightarrow +\infty, \quad t>0.
\end{cases}
\end{equation}

From Theorem \ref{ths1}, there exists a time $T^1_*>0$ that is  independent of $\epsilon$ such that the solution $(\varphi^\epsilon, W^\epsilon)$ of Cauchy problem \eqref{E:1.1} satisfies
\begin{equation}\label{jkka}
\begin{split}
\sup_{0\leq t \leq T^1_*}\big(\| \varphi^\epsilon(t)\|^2_{2}+\| \phi^\epsilon(t)\|^2_{3}+\| u^\epsilon(t)\|^2_{2}+\epsilon |\varphi^\epsilon(t)|^2_{D^3}\big)& \\
+\text{ess}\sup_{0\leq t \leq T^1_*}| u^\epsilon(t)|^2_{D^3}+\int_{0}^{T^1_*} \epsilon |\varphi^\epsilon \nabla^4 u^\epsilon(t)|^2_{2}
\text{d}t \leq& C^0,\\
\end{split}
\end{equation}
where $C^0$ is a positive constant depending only on $T^1_*$, $( \varphi_0, W_0)$ and the fixed constants $A$, $\delta$, $\gamma$, $\alpha$ and $\beta$, and is independent of $\epsilon$.

From \cite{tms1} (see Theorem \ref{thmakio}), we know that there exits a time $T^2_{*}$ such that  there is a unique regular solution $W$ of Cauchy problem \eqref{E:4.2} satisfies
\begin{equation}\label{jkkab}
\sup_{0\leq t \leq T^2_*}\| W(t)\|^2_{3}  \leq C^0.\\
\end{equation}
Denote $T^*=\min \{T^1_*, T^2_* \}$. Then, letting $\overline{W}^\epsilon=W^\epsilon-W$, $\eqref{E:1.1}_2$ and $\eqref{E:4.2}_1$ lead to
\begin{equation}\begin{split}\label{E:4.3}
&A_0\overline{W}^\epsilon_t+\sum_{j=1}^3A_j(W^\epsilon)\partial_j\overline{W}^\epsilon\\
=&-\sum_{j=1}^3A_j(\overline{W}^\epsilon)\partial_j W-\epsilon\Big((\varphi^\epsilon)^2\mathbb{L}(W^\epsilon)- \mathbb{H}(\varphi^\epsilon)\cdot\mathbb{Q}(W^\epsilon)\Big).\\
\end{split}\end{equation}
 If the initial data
\begin{equation}
(\rho^{\epsilon}, u^{\epsilon})|_{t=0}=(\rho, u)|_{t=0}=(\rho_0,
u_0)
\end{equation}
satisfies \eqref{th78}, it is obviously to see that $\overline{W}^\epsilon(x,0)\equiv0.$

Next we give the following several lemmas for establishing the vanishing viscosity limit.
\begin{lemma}\label{L:4.1} If $(\varphi^\epsilon, W^\epsilon)$ and $W$ are the regular solutions of Navier-Stokes \eqref{E:1.1} and Euler equations
\eqref{E:4.2} respectively, then we have
\begin{equation}\label{E:4.1}
|\overline{W}^\epsilon(t)|_2\leq C\epsilon,\quad \text{for} \quad 0\leq t\leq T_*,
\end{equation}
where the constant  $C>0$ depends on $A$, $\alpha$, $\beta$, $\delta$, $(\rho_0, u_0)$ and $ T_*$.
\end{lemma}
\begin{proof}
Multiplying $\eqref{E:4.3}$ by $2\overline{W}^\epsilon$ on both sides and integrating over $\R^3$, then one has
\begin{equation}
\label{E:4.3a}\begin{split}
\displaystyle
\frac{d}{dt} &\int (\overline{W}^\epsilon)^\top A_0\overline{W}^\epsilon+2\sum_{j=1}^3\int(\overline{W}^\epsilon)^\top A_j({W}^\epsilon) \partial_j \overline{W}^\epsilon\\
=&-2\sum_{j=1}^3\int (\overline{W}^\epsilon)^\top A_j(\overline{W}^\epsilon) \partial_j {W}
-2\epsilon \int\Big((\varphi^\epsilon)^2 \mathbb{L}(W^\epsilon)-\mathbb{H}(\varphi^\epsilon) \cdot \mathbb{Q} ({W}^\epsilon)\Big) \cdot  \overline{W}^\epsilon.
\end{split}
\end{equation}
It follows from the  integrating by parts and H\"older's inequality that 
\begin{equation*}
\begin{split}
\displaystyle
\frac{d}{dt} \int (\overline{W}^\epsilon)^\top A_0\overline{W}^\epsilon
\leq& \int(\overline{W}^\epsilon)^\top \text{div}A({W}^\epsilon)  \overline{W}^\epsilon
+C|\nabla W|_\infty|\overline{W}^\epsilon|^2_2\\
&+2\epsilon
\big|(\varphi^\epsilon)^2\mathbb{L} ({W}^\epsilon)-\mathbb{H}(\varphi^\epsilon)\cdot\mathbb{Q}({W}^\epsilon)\big|_2|\overline{W}^\epsilon|_2\\[8pt]
\leq& C\big(|\nabla W^\epsilon|_\infty+|\nabla W|_\infty)|\overline{W}^\epsilon|^2_2
\\
&+C\epsilon\big(|\varphi^\epsilon|^2_\infty
|\nabla^2 u^\epsilon|_2+|\varphi^\epsilon|_\infty
|\nabla\varphi^\epsilon|_3|\nabla u^\epsilon|_6
\big)|\overline{W}^\epsilon|_2\\
\leq & C|\overline{W}^\epsilon|^2_2+C\epsilon^{2}\\
\end{split}
\end{equation*}
where we used \eqref{zhen6a}.
According to the  Gronwall's inequality,  (\ref{E:4.1}) follows immediately.

\end{proof}
Now we consider  the estimate of $|\partial_x^\zeta \overline{W}^\epsilon|_2,$ as $|\zeta|=1.$
\begin{lemma}\label{L:4.2}  If $(\varphi^\epsilon, W^\epsilon)$ and $W$ are the regular solutions of Navier-Stokes \eqref{E:1.1} and Euler equations \eqref{E:4.2} respectively,   then we have
\begin{equation}\label{E:4.6}
|\overline{W}^\epsilon|_{D^1}\leq C\epsilon,\quad \text{for} \quad 0\leq t \leq T_*,
\end{equation}where the constant  $C>0$ depends on $A$, $\alpha$, $\beta$, $\delta$, $(\rho_0, u_0)$ and $ T_*$.
\end{lemma}
\begin{proof}
Applying the operator $\partial_x^\zeta$ on $\eqref{E:4.3}$, one gets
\begin{equation}\label{E:4.7}
\begin{split}
&A_0(\partial_x^\zeta \overline{W}^\epsilon)_t+\sum_{j=1}^3A_j(W^\epsilon)\partial_j(\partial_x^\zeta \overline{W}^\epsilon)\\
=&\sum_{j=1}^3\left(\partial_x^\zeta (A_j(W^\epsilon)\partial_j\overline{W}^\epsilon)
-A_j(W^\epsilon)\partial_j(\partial_x^\zeta \overline{W}^\epsilon)\right)-\sum_{j=1}^3\partial_x^\zeta (A_j(\overline{W}^\epsilon))\partial_j W\\
&-\sum_{j=1}^3 \left(\partial_x^\zeta \big(A_j(\overline{W}^\epsilon)\partial_j W\big)
-\partial_x^\zeta \big(A_j(\overline{W}^\epsilon)\big)\partial_j  W\right)\\
&-\epsilon(\varphi^\epsilon)^2\mathbb{L}(\partial_x^\zeta W^\epsilon)
+\epsilon\mathbb{H}(\varphi^\epsilon)^2\cdot\partial_x^\zeta \mathbb{Q}(W^\epsilon)
-\epsilon\left(\partial_x^\zeta\big((\varphi^\epsilon)^2
\mathbb{L}(W^\epsilon)\big)-(\varphi^\epsilon)^2\mathbb{L}(\partial_x^\zeta W^\epsilon)\right)\\
&+\epsilon\left(\partial_x^\zeta\big(\mathbb{H}(\varphi^\epsilon)^2\cdot\mathbb{Q}(W^\epsilon)\big)-\mathbb{H} (\varphi^\epsilon)^2
\cdot\partial_x^\zeta \mathbb{Q}(W^\epsilon)\right).\\
\end{split}
\end{equation}
Then multiplying (\ref{E:4.7}) by $2\partial^\zeta_x \overline{W}^\epsilon$  and integrating  over $\mathbb{R}^3$ by parts,  one can obtain that 
\begin{equation}\label{E:4.8}
\begin{split}
& \frac{d}{dt}\int(\partial^\zeta_x \overline{W}^\epsilon)^\top A_0 \partial^\zeta_x \overline{W}^\epsilon\\
=&\int (\partial^\zeta_x \overline{W}^\epsilon)^\top\text{div}A (W^\epsilon)\partial^\zeta_x \overline{W}^\epsilon -2\sum_{j=1}^3\int(\partial_x^\zeta \overline{W}^\epsilon)^\top\partial_x^\zeta (A_j(\overline{W}^\epsilon))\partial_j W\\
&+2\sum_{j=1}^3\int\left(\partial_x^\zeta \left(A_j(W^\epsilon)\partial_j\overline{W}^\epsilon\right)-A_j(W^\epsilon)\partial_j(\partial_x^\zeta \overline{W}^\epsilon)\right)\cdot\partial_x^\zeta \overline{W}^\epsilon\\
&-2\sum_{j=1}^3 \int  \left(\partial_x^\zeta\big(A_j(\overline{W}^\epsilon)\partial_j W \big) -\partial_x^\zeta A_j(\overline{W}^\epsilon)\partial_j  W \right)\cdot\partial_x^\zeta \overline{W}^\epsilon \\
&-2a_1 \epsilon\int \Big((\varphi^\epsilon)^2  L(\partial^\zeta_x u^\epsilon)\Big) \cdot \partial^\zeta_x \overline{u}^\epsilon
+2a_1 \epsilon\int \Big(\nabla (\varphi^\epsilon)^2 \cdot Q(\partial^\zeta_x u^\epsilon)\Big) \cdot \partial^\zeta_x \overline{u}^\epsilon
\\
&-2a_1 \epsilon\int \Big( \partial^\zeta_x ((\varphi^\epsilon)^2 Lu^\epsilon)-(\varphi^\epsilon)^2 L\partial^\zeta_x u^\epsilon \Big)\cdot \partial^\zeta_x \overline{u}^\epsilon\\
&+2a_1 \epsilon\int  \Big(\partial^\zeta_x(\nabla (\varphi^\epsilon)^2  \cdot Q(u^\epsilon))-\nabla (\varphi^\epsilon)^2  \cdot  Q(\partial^\zeta_x u^\epsilon)\Big)\cdot \partial^\zeta_x \overline{u}^\epsilon
:=\sum_{i=1}^{8} I_i,\\
\end{split}
\end{equation}
where $\overline{u}^\epsilon=u^\epsilon-u$.  For $|\zeta|=1$,   one gets
\begin{equation}\label{E:4.9}\begin{split}\displaystyle
I_1=&\int (\partial^\zeta_x \overline{W}^\epsilon)^\top\text{div}A (W^\epsilon)\partial^\zeta_x \overline{W}^\epsilon
\leq C|\nabla W^\epsilon|_\infty|\nabla \overline{W}^\epsilon|^2_2\leq  C| \overline{W}^\epsilon|^2_{D^1},\\
\displaystyle
I_2=&-2\sum_{j=1}^3\int(\partial_x^\zeta \overline{W}^\epsilon)^\top\partial_x^\zeta (A_j(\overline{W}^\epsilon))\partial_j W
\leq C|\nabla W|_\infty|\overline{W}^\epsilon|^2_{D^1}\leq C|\overline{W}^\epsilon|^2_{D^1},\\
\displaystyle
I_3=&2\sum_{j=1}^3\int\left(\partial_x^\zeta (A_j(W^\epsilon)\partial_j\overline{W}^\epsilon)-A_j(W^\epsilon)\partial_j(\partial_x^\zeta \overline{W}^\epsilon)\right)\cdot\partial_x^\zeta \overline{W}^\epsilon\\[2pt]
\leq& C|\nabla W^\epsilon|_\infty|\nabla\overline{W}^\epsilon|^2_{2}\leq C|\overline{W}^\epsilon|^2_{D^1},\\
\end{split}\end{equation}
and 
\begin{equation}\label{E:4.9a}\begin{split}
\displaystyle
 I_4=&- 2\sum_{j=1}^3 \int  \left(\partial_x^\zeta\big(A_j(\overline{W}^\epsilon)\partial_j W \big) -\partial_x^\zeta (A_j(\overline{W}^\epsilon))\partial_j  W \right)\cdot\partial_x^\zeta \overline{W}^\epsilon \\
\leq &C|\nabla^2 W|_3|\nabla \overline{W}^\epsilon|_2|\overline{W}^\epsilon|_6
\leq C| \overline{W}^\epsilon|^2_{D^1},\\
I_5=&-2a_1 \epsilon\int \Big( (\varphi^\epsilon)^2 \cdot L(\partial^\zeta_x u^\epsilon)\Big) \cdot \partial^\zeta_x \overline{u}^\epsilon\\
\leq& C\epsilon|\varphi^\epsilon|^2_\infty| \nabla^3 u^\epsilon|_2|\nabla \overline{u}^\epsilon|_2
\leq  C\epsilon|\nabla\overline{u}^\epsilon|_2
\leq   C\epsilon^2+C|\overline{W}^\epsilon|^2_{D^1},\\
I_6=&2a_1\epsilon\int \Big(\nabla (\varphi^\epsilon)^2 \cdot Q(\partial^\zeta_x u^\epsilon)\Big) \cdot \partial^\zeta_x \overline{u}^\epsilon\\
\leq& C\epsilon|\varphi^\epsilon|_\infty|\nabla \varphi^\epsilon|_3|\nabla^2 u^\epsilon|_6|\nabla \overline{u}^\epsilon|_2
\leq C\epsilon|\nabla \overline{u}^\epsilon|_2\leq  C\epsilon^2+C|\overline{W}^\epsilon|^2_{D^1},\\
I_7=&-2a_1 \epsilon\int \Big( \partial^\zeta_x ((\varphi^\epsilon)^2 Lu^\epsilon)-(\varphi^\epsilon)^2 L\partial^\zeta_x u^\epsilon \Big)\cdot \partial^\zeta_x \overline{u}^\epsilon\\
\leq& C\epsilon|\varphi^\epsilon|_\infty|\nabla \varphi^\epsilon|_3|\nabla^2 u^\epsilon|_6|\nabla \overline{u}^\epsilon|_2
\leq  C\epsilon^{2}+C|\overline{W}^\epsilon|^2_{D^1},\\
I_8=&2a_1 \epsilon\int  \Big(\partial^\zeta_x(\nabla (\varphi^\epsilon)^2  \cdot Q(u^\epsilon))-\nabla (\varphi^\epsilon)^2  \cdot  Q(\partial^\zeta_x u^\epsilon)\Big)\cdot \partial^\zeta_x \overline{u}^\epsilon\\
\leq& C\epsilon(|\varphi^\epsilon|_\infty|\nabla u^\epsilon|_\infty|\nabla^2 \varphi^\epsilon|_2|\nabla \overline{u}^\epsilon|_2
+|\nabla u^\epsilon|_\infty|\nabla \varphi^\epsilon|_3|\nabla \varphi^\epsilon|_6|\nabla \overline{u}^\epsilon|_2)\\[8pt]
\leq&C\epsilon|\nabla \overline{u}^\epsilon|_2\leq  C\epsilon^2+C|\overline{W}^\epsilon|^2_{D^1}.
\end{split}\end{equation}
Substituting \eqref{E:4.9}-\eqref{E:4.9a}  into \eqref{E:4.8}, one has
$$
 \frac{d}{dt}\int(\partial^\zeta_x \overline{W}^\epsilon)^\top A_0 \partial^\zeta_x \overline{W}^\epsilon\leq C\epsilon^2+C|\overline{W}^\epsilon|^2_{D^1}.
$$

Then according to the  Gronwall's inequality,  \eqref{E:4.6} follows immediately.

\end{proof}
For $|\zeta|=2,$ we have
\begin{lemma}\label{L:4.3} If $(\varphi^\epsilon,W^\epsilon)$ and $W$ are the regular solutions of Navier-Stokes \eqref{E:1.1} and Euler equations \eqref{E:4.2} respectively, then we have
\begin{equation}\label{E:4.13}
|\overline{W}^\epsilon|_{D^2}\leq C\epsilon^{\frac{1}{2}},\quad \text{for} \quad 0\leq t\leq T_*,
\end{equation}
where the constant $C>0$ depends on $A$, $\alpha$, $\beta$, $\delta$, $(\rho_0, u_0)$ and $ T_*$.
\end{lemma}
\begin{proof}
For $|\zeta|=2,$ it follows from \eqref{zhen6a} and  \eqref{E:4.8} that 
\begin{equation}\label{E:4.15}\begin{split}\displaystyle
I_1=&\int (\partial^\zeta_x \overline{W}^\epsilon)^\top\text{div}A (W^\epsilon)\partial^\zeta_x \overline{W}^\epsilon
\leq C|\nabla W^\epsilon|_\infty|\nabla^2 \overline{W}^\epsilon|^2_2
\leq  C| \overline{W}^\epsilon|^2_{D^2},\\
I_2=&2\sum_{j=1}^3\int\left(\partial_x^\zeta (A_j(W^\epsilon)\partial_j\overline{W}^\epsilon)-A_j(W^\epsilon)\partial_j(\partial_x^\zeta \overline{W}^\epsilon)\right)\cdot\partial_x^\zeta \overline{W}^\epsilon\\
\leq&  C(|\nabla W^\epsilon|_\infty|\nabla^2\overline{W}^\epsilon|_2+|\nabla^2 W^\epsilon|_3|\nabla\overline{W}^\epsilon|_6)|\nabla^2\overline{W}^\epsilon|_2
\leq C|\overline{W}^\epsilon|^2_{D^2},\\
\end{split}\end{equation}
and 
\begin{equation}\label{E:4.18}\begin{split}
\displaystyle
I_3=&-2\sum_{j=1}^3\int(\partial_x^\zeta \overline{W}^\epsilon)^\top\partial_x^\zeta (A_j(\overline{W}^\epsilon))\partial_j W\\
\leq& C|\nabla W|_\infty|\overline{W}^\epsilon|^2_{D^2}\leq C|\overline{W}^\epsilon|^2_{D^2},\\
 I_4=&- 2\sum_{j=1}^3 \int  \left(\partial_x^\zeta\big(A_j(\overline{W}^\epsilon)\partial_j W \big) -\partial_x^\zeta (A_j(\overline{W}^\epsilon))\partial_j  W \right)\cdot\partial_x^\zeta \overline{W}^\epsilon \\
 \leq &C\int\left(|\nabla\overline{W}^\epsilon||\nabla^2 W|+|\overline{W}^\epsilon||\nabla^3 W|\right)|\nabla^2\overline{W}^\epsilon|\\
\leq& C\left( |\nabla \overline{W}^\epsilon|_6|\nabla^2 W|_3|\nabla^2 \overline{W}^\epsilon|_2+|\nabla^3 W|_2|\overline{W}^\epsilon|_\infty|\nabla^2 \overline{W}^\epsilon|_2\right)\\
\leq& C\left( |\nabla^2 \overline{W}^\epsilon|^2_2+\epsilon^{\frac{1}{2}}|\nabla^2 \overline{W}^\epsilon|_2^{\frac{3}{2}}\right)
\leq C| \overline{W}^\epsilon|^2_{D^2}+C\epsilon^2,\\
I_5=&-2a_1 \epsilon\int \Big( (\varphi^\epsilon)^2 \cdot L(\partial^\zeta_x u^\epsilon)\Big) \cdot \partial^\zeta_x \overline{u}^\epsilon\\
\leq& C\epsilon|\varphi^\epsilon|_\infty|\varphi^\epsilon\nabla^4 u^\epsilon|_2|\nabla^2 \overline{u}^\epsilon|_2
\leq  C\epsilon^2|\varphi^\epsilon\nabla^4 u^\epsilon|_2^2+C|\overline{W}^\epsilon|^2_{D^2},\\
I_6=&2a_1\epsilon\int \Big(\nabla (\varphi^\epsilon)^2 \cdot Q(\partial^\zeta_x u^\epsilon)\Big) \cdot \partial^\zeta_x \overline{u}^\epsilon\\
\leq& C\epsilon|\varphi^\epsilon|_\infty|\nabla \varphi^\epsilon|_\infty|\nabla^3 u^\epsilon|_2|\nabla^2 \overline{u}^\epsilon|_2
\leq C\epsilon^{\frac{3}{4}} |\nabla^2 \overline{u}^\epsilon|_2\leq  C\epsilon^{\frac{3}{2}}+C|\overline{W}^\epsilon|^2_{D^2},\\
I_7=&-2a_1 \epsilon\int \Big( \partial^\zeta_x ((\varphi^\epsilon)^2 Lu^\epsilon)-(\varphi^\epsilon)^2 L\partial^\zeta_x u^\epsilon \Big)\cdot \partial^\zeta_x \overline{u}^\epsilon\\
\leq& C\epsilon\int\Big(\big(|\varphi^\epsilon||\nabla^2\varphi^\epsilon|+|\nabla \varphi^\epsilon|^2\big)|\nabla^2 u^\epsilon|+|\varphi^\epsilon||\nabla \varphi^\epsilon||\nabla^3 u^\epsilon|\Big)|\nabla^2\overline{u}^\epsilon|\\
\leq &C\epsilon|\varphi^\epsilon|_\infty
|\nabla^2\varphi^\epsilon|_3|\nabla^2u^\epsilon|_6|\nabla^2 \overline{u}^\epsilon|_2+C\epsilon|\nabla \varphi^\epsilon|^2_6|\nabla^2 u^\epsilon|_6|\nabla^2 \overline{u}^\epsilon|_2\\
&+C\epsilon|\varphi^\epsilon|_\infty|\nabla \varphi^\epsilon|_\infty|\nabla^3 u^\epsilon|_2|\nabla^2 \overline{u}^\epsilon|_2
\leq C\epsilon^{\frac{3}{2}}+C|\overline{u}^\epsilon|^2_{D^2},\\
%
I_8=& \epsilon\int\left(\partial_x^\zeta \big(\nabla(\varphi^\epsilon)^2\cdot Q(u^\epsilon)\big)-\nabla(\varphi^\epsilon)^2
\cdot Q(\partial_x^\zeta u^\epsilon)\right)\cdot\partial_x^\zeta \overline{u}^\epsilon\\
\leq& C\epsilon\int \big(|\varphi^\epsilon ||\nabla^3\varphi^\epsilon|+|\nabla\varphi^\epsilon||\nabla^2\varphi^\epsilon|\big)|\nabla u^\epsilon||\nabla^2\overline{u}^\epsilon|\\
&+ C\epsilon\int \big(|\varphi^\epsilon||\nabla^2 \varphi^\epsilon|+|\nabla\varphi^\epsilon|^2\big)|\nabla^2 u^\epsilon||\nabla^2\overline{u}^\epsilon|\\
\leq& C\epsilon|\varphi^\epsilon|_\infty|\nabla u^\epsilon|_\infty |\nabla^3 \varphi^\epsilon|_2|\nabla^2 \overline{u}^\epsilon|_2+
C\epsilon|\nabla\varphi^\epsilon|_3|\nabla^2\varphi^\epsilon|_6|\nabla u^\epsilon|_\infty|\nabla^2\overline{u}^\epsilon|_2\\[2pt]
 &+C\epsilon|\varphi^\epsilon|_\infty|\nabla^2 \varphi^\epsilon|_3|\nabla^2 u^\epsilon|_6 |\nabla^2 \overline{u}^\epsilon|_2+
C\epsilon|\nabla\varphi^\epsilon|_6|\nabla\varphi^\epsilon|_6|\nabla^2 u^\epsilon|_6|\nabla^2\overline{u}^\epsilon|_2\\[2pt]
\leq& C\epsilon^{\frac{1}{2}}|\nabla^2\overline{u}^\epsilon|_2\leq
C\epsilon+C|\nabla^2\overline{u}^\epsilon|^2_{D^2}.
\end{split}\end{equation}
Substituting \eqref{E:4.15}-\eqref{E:4.18} into \eqref{E:4.8},  we have
$$
 \frac{d}{dt}\int(\partial^\zeta_x \overline{W}^\epsilon)^\top A_0 \partial^\zeta_x \overline{W}^\epsilon\leq
C\epsilon+C|\overline{W}^\epsilon|^2_{D^2}+C\epsilon^2|\varphi^\epsilon\nabla^4 u^\epsilon|_2^2.
$$

Then according to the Gronwall's inequality, 
 \eqref{E:4.13} follows immediately.
\end{proof}

Finally, we have
\begin{lemma}\label{L:4.4}If $(\varphi^\epsilon,W^\epsilon)$ and $W$ are the regular solutions of Navier-Stokes \eqref{E:1.1} and Euler equations \eqref{E:4.2} respectively, then we have
\begin{equation}\label{E:4.13L}
\sup\limits_{0\leq t\leq T_*}\|\overline{W}^\epsilon(t)\|_{H^{s'}}\leq C\epsilon^{1-\frac{s'}{3}},
\end{equation}
where $s'\in(2,3)$ and the constant $C>0$ depends on $A$, $\alpha$, $\beta$, $\delta$, $(\rho_0, u_0)$ and $ T_*$.
\end{lemma}
\begin{proof}
Based on Lemma \ref{gag111}, \eqref{th78qq}, \eqref{jkka} and  \eqref{jkkab}, we have
$$
\|\overline{W}^\epsilon(t)\|_{s'}\leq C\|\overline{W}_\epsilon(t)\|^{1-\frac{s'}{3}}_0\|\overline{W}_\epsilon(t)\|_3^{\frac{s'}{3}}\leq
C\epsilon^{1-\frac{s'}{3}}.
$$
\end{proof}

Finally, we give the proof for Theorem \ref{th3}.
\begin{proof}

First, according to  Lemmas \ref{L:4.1}-\ref{L:4.4},
 when $\epsilon\to 0$, the solutions 
$\big(\rho^{\epsilon}, u^{\epsilon}\big)$ of compressible Navier-Stokes solutions
converges to the solution $(\rho, u)$ of compressible Euler equations in the following sense
\begin{equation}\label{shou1gg}
\lim_{\epsilon \to 0}\sup_{0\leq t\leq
T_*}\left(\Big\|\Big((\rho^{\epsilon})^{\frac{\gamma-1}{2}}-\rho^{\frac{\gamma-1}{2}}\Big)(t)\Big\|_{H^{s'}}+\big\|\big(u^\epsilon
-u\big)(t)\big\|_{H^{s'}}\right)=0,
\end{equation}
for any constant  $s'\in [0, 3)$.  Moreover, one can also obtain that
\begin{equation}\label{decaygg}
\begin{split}
\sup_{0\leq t \leq T_*}\left(\Big\|\Big((\rho^{\epsilon})^{\frac{\gamma-1}{2}}-\rho^{\frac{\gamma-1}{2}}\Big)(t)\Big\|_1+\|\big(u^\epsilon -u\big)(t)\|_1\right)\leq& C\epsilon,\\
\sup_{0\leq t\leq
T_*}\left(\Big|\Big((\rho^{\epsilon})^{\frac{\gamma-1}{2}}-\rho^{\frac{\gamma-1}{2}}\Big)(t)\Big|_{D^2}+|\big(u^\epsilon
-u)(t)|_{D^2}\right)\leq& C\sqrt{\epsilon},
\end{split}
\end{equation}
where $C>0$ is a constant depending only on the fixed
 constants $A, \delta, \gamma, \alpha, \beta, T_*$ and $ \rho_0, u_0$.

Further more, if  the condition (\ref{regco}) holds,  one has  $\frac{2}{\delta-1}\geq 3$. Then from (\ref{shou1gg})-(\ref{decaygg}), it is not hard to see that  
\begin{equation}\label{decay11gg}
\begin{split}
\lim_{\epsilon \to 0}\sup_{0\leq t\leq
T_*}\left(\big\|\big(\rho^{\epsilon}-\rho\big)(t)\big\|_{H^{s'}}+\big\|\big(u^\epsilon
-u\big)(t)\big\|_{H^{s'}}\right)=&0,\\
\sup_{0\leq t \leq T_*}\left(\big\|\rho^{\epsilon}-\rho)(t)\big\|_1+\|\big(u^\epsilon -u\big)(t)\|_1\right)\leq& C\epsilon,\\
\sup_{0\leq t\leq
T_*}\left(\big|\big(\rho^{\epsilon}-\rho\big)(t)\big|_{D^2}+|\big(u^\epsilon
-u)(t)|_{D^2}\right)\leq& C\sqrt{\epsilon}.
\end{split}
\end{equation}

Thus the proof of Theorem 1.3 is finished.
\end{proof}

\begin{remark}\label{ding}
For the case $\delta=1$ with  vacuum at the far field, 
via introducing two different symmetric structures in Ding-Zhu  \cite{ding},
 some uniform estimates with respect to the viscosity coefficients for
$\displaystyle\Big(\rho^{\frac{\gamma-1}{2}}, u\Big)$ in $H^3(\mathbb{R}^2)$
and $\displaystyle\nabla \rho/\rho$ in $L^6\cap D^1(\mathbb{R}^2)$  have been obtained, which
lead the convergence of the regular solution of  the viscous flow to
that of the  inviscid flow still   in $L^{\infty}([0, T]; H^{s'}(\mathbb{R}^2))$ (for any
$s'\in [2, 3)$)  with the rate of $\epsilon^{2(1-\frac{s'}{3})}$.  Their conclusion also applies to  the 2-D shallow water equations \eqref{eq:1.1ww}.\end{remark}

\bigskip

{\bf Acknowledgement:}
This research  was funded in part
by National Natural Science Foundation of China under grants 11201308, 11231006 and 11571232, and Natural Science Foundation of Shanghai under grant 14ZR1423100. Y. Li was also funded by Shanghai Committee of Science and Technology under grant 15XD1502300. 

\bigskip

{\bf Compliance with Ethical Standards:} The authors are proud to comply with the required ethical standards
by Archive for Rational Mechanics and Analysis. 

\bigskip

{\bf Conflict of Interest:} The authors declare that they have no conflict of interest.

\end{document}